\documentclass{amsart}
\usepackage[margin=1in]{geometry}

\usepackage{amssymb, amsmath, amsthm, amsfonts, mathtools}

\usepackage{hyperref}
\usepackage{enumerate}
\usepackage{stmaryrd}

\usepackage{diagbox}
\usepackage{booktabs}

\usepackage{tikz, tikz-cd}
\usetikzlibrary{arrows}

\usepackage[title]{appendix}

\DeclareMathOperator{\Pic}{Pic}
\DeclareMathOperator{\lcm}{lcm}
\DeclareMathOperator{\Aut}{Aut}
\DeclareMathOperator{\Eff}{Eff}
\DeclareMathOperator{\Hom}{Hom}

\DeclareMathOperator{\NE}{NE}

\DeclareMathOperator{\Proj}{Proj}

\DeclareMathOperator{\Inv}{Inv}

\newcommand{\der}{\mathrm{der}}

\DeclarePairedDelimiter{\floor}{\lfloor}{\rfloor}
\DeclarePairedDelimiter{\bangle}{\langle}{\rangle}

\newcommand{\thickslash}{\mathbin{\!\!\pmb{\fatslash}}}
\newcommand{\ZZ}{\mathbb{Z}}
\newcommand{\PP}{\mathbb P}
\newcommand{\CC}{\mathbb C}
\newcommand{\EE}{\mathbb{E}}
\newcommand{\QQ}{\mathbb Q}
\newcommand{\fX}{\mathfrak X}
\newcommand{\fY}{\mathfrak Y}
\newcommand{\cO}{\mathcal O}
\newcommand{\cL}{\mathcal L}
\newcommand{\cC}{\mathcal C}
\newcommand{\cE}{\mathcal E}
\newcommand{\cH}{\mathcal H}
\newcommand{\cF}{\mathcal F}
\renewcommand{\sslash}{\mathord{/\mkern-6mu/}}

\newcommand{\TT}{\mathbb T}
\newcommand{\LL}{\mathbb L}

\newcommand{\cI}{\mathcal I}
\newcommand{\ocI}{\overline{\mathcal I}}

\newcommand{\one}{\mathbf 1}
\newcommand{\bt}{\mathbf t}

\newcommand{\pt}{\mathrm{pt}}
\newcommand{\ev}{\mathrm{ev}}
\newcommand{\CR}{\mathrm{CR}}
\renewcommand{\ss}{\mathrm{ss}}
\newcommand{\amb}{\mathrm{amb}}
\newcommand{\Mbar}{\overline{\mathcal M}}

\newcommand{\vir}{\mathrm{vir}}
\newcommand{\mov}{\mathrm{mov}}
\newcommand{\fix}{\mathrm{fix}}

\newtheorem{theorem}{Theorem}[subsection]

\newtheorem{lemma}[theorem]{Lemma}

\newtheorem{definition}[theorem]{Definition}

\newtheorem{corollary}[theorem]{Corollary}
\newtheorem{proposition}[theorem]{Proposition}
\newtheorem{question}[theorem]{Question}
\theoremstyle{remark}
\newtheorem{remark}[theorem]{Remark}

\theoremstyle{definition}

\begin{document}

\title[Gromov--Witten Invariants of Non-Convex Complete Intersections]{Gromov--Witten Invariants of Non-Convex Complete Intersections in Weighted Projective Stacks}
\author{Felix Janda, Nawaz Sultani and Yang Zhou}

\address[F. Janda]{Department of Mathematics\\
University of Illinois Urbana--Champaign\\
Urbana, IL 61801 \\
U.S.A.}
\email{fjanda@illinois.edu}

\address[N. Sultani]{Institute for Mathematics \\
  Academia Sinica \\
  10617 Taipei \\
  Taiwan}
\email{sultani@gate.sinica.edu.tw}

\address[Y. Zhou]{Shanghai Center for Mathematical Sciences \\
  Fudan University \\
  Shanghai \\
  China}
\email{y\_zhou@fudan.edu.cn}

\begin{abstract}
  In this paper we compute genus-zero orbifold Gromov--Witten invariants of
  Calabi--Yau threefold complete intersections in weighted projective stacks,
  regardless of convexity conditions. The traditional quantum Lefschetz principle may fail
  even for invariants with ambient insertions. Using quasimap wall-crossing, 
  we are able to compute invariants
  with insertions from a specific subring of the Chen--Ruan cohomology, which contains
  all the ambient cohomology classes.

  Quasimap wall-crossing gives a mirror theorem expressing the $I$-function in
  terms of the $J$-function via a mirror map.
  The key of this paper is to find a suitable 
  GIT presentation of the target space, such that the mirror map
  becomes invertible. 
  An explicit formula for the $I$-function is given for all those target spaces
  and many examples with explicit computations of invariants are
  provided.
\end{abstract}

\maketitle

\section{Introduction} 
Let $X \subset Y$ be a smooth complete intersection in a smooth projective
variety defined by a generic section of a vector bundle $E$.
When $E$ is convex, i.e.\ when $H^1(C, f^*E) = 0$ for any genus-zero stable map
$f: C \to Y$, the spaces $H^0(C, f^*E)$ define a vector bundle $E_{0, k, \beta}$
over the moduli stack of stable maps $\Mbar_{0, k}(Y, \beta)$. The quantum
Lefschetz hyperplane theorem \cite{Giv98a, Le01, CG07, CCIT09, CGIJJM12} then
states that the virtual class \cite{LT98, BeFa97, Be97} 
on $\Mbar_{0, k}(X, \beta)$ satisfies the relation
\[ i_*\left[\Mbar_{0, k}(X, \beta)\right]^{\vir} = e\left(E_{0, k, \beta}\right)
\cap \left[\Mbar_{0, k}(Y, \beta)\right]^{\vir} \]
where $i$ is the inclusion morphism. This implies the genus-zero Gromov--Witten
invariants of $X$ with ambient insertions can be computed in terms of the,
generally easier to compute, $E$-twisted genus-zero invariants of $Y$.

This theorem is of fundamental importance due to its use in the proof of the
genus-zero mirror theorem for complete intersections in toric varieties, which
includes the case of the quintic threefold $Q_5 \subset \PP_4$ \cite{Giv98a,
LLY97}.
The failure of the analogous convexity condition in genus $g > 0$ (see e.g.\
\cite{Giv98b}), and the subsequent lack of a quantum Lefschetz-type result, is
one of the many reasons for the greatly increased difficulty in computing $g>0$
Gromov--Witten invariants for complete intersections.

Unfortunately, convexity rarely holds when $Y$ is more generally a proper
Deligne--Mumford stack, even for genus zero \cite{CGIJJM12}. In 
  the setting of \cite{CGIJJM12},
  assuming the positivity condition there,
$E$ being convex is equivalent to being isomorphic to a vector bundle pulled
back from the coarse moduli space of $Y$  \cite[Remark~5.3]{CGIJJM12}, which is a highly restrictive
condition. As a result, the classical quantum Lefschetz arguments cannot be used
to compute genus-zero orbifold Gromov--Witten invariants for most complete
intersections in toric stacks.

The purpose of this paper is to introduce a method for computing genus-zero
orbifold Gromov--Witten invariants that circumvents the need for convexity. The
invariants we can compute are those whose insertions are what we call
``admissible" classes, which include all ambient classes. As a result, we are
able to compute all the invariants, if not more, that a quantum Lefschetz-type
theorem would allow one to compute.
In this paper, we restrict ourselves to the case of Calabi--Yau threefolds in
weighted projective stacks, but our techniques can be applied more generally to
arbitrary complete intersections in toric stacks, which we remark on in Section
\ref{sec:toric-stacks}.

\subsection{Summary and Main Results} 
Let $\PP(\vec{w}):= \PP(w_0, \dots, w_n)$ denote the weighted projective stack
given by the GIT stack quotient
\begin{equation}
  \label{eq:wp-stack}
\PP(\vec{w}) = [ \CC^{n+1} \sslash_\theta \CC^*] = [ (\CC^{n+1})^{\ss}/ \CC^*]
\end{equation}
where the $\CC^*$-action is given by the weights $w_i$, and $\theta$ is the
character $\theta(z) = z$. Throughout the paper, we will always assume that
$\gcd(w_0, \dots, w_n) = 1$. Let $E := \bigoplus_{j=1}^{n-3} \cO(b_j)$ be a
split vector bundle of rank $n-3$ on $\PP(\vec{w})$ satisfying the Calabi--Yau
condition, i.e.\ $\sum_{j=1}^{n-3} b_j = \sum_{i=0}^n w_i$. We assume that the
vanishing of a generic section of $E$ defines a smooth substack, which we denote
by
\[
  X = [W \sslash_\theta \CC^*] \subset \PP(\vec{w})
\]
where $W$ is the affine cone of $X$. By our assumptions, $X$ will be a
Calabi--Yau threefold (CY3), in the sense that the canonical bundle $K_X$ is
trivial. The goal is to compute the genus-zero Gromov--Witten theory of $X$.

Our approach uses the orbifold quasimap theory of \cite{CCK15} in conjunction
with the wall-crossing result of \cite{Zh19P}. Let $Q_{0,n}^\epsilon(X, \beta)$
denote the moduli stack of $\epsilon$-stable, genus $0$, $n$-marked quasimaps to
$X$, for $\epsilon \in \QQ_{>0} \cup \{0^+, \infty\}$. This family of moduli
stacks provides various compactifications of the moduli stack of maps from
irreducible twisted marked curves to $X$ where, in particular, the $\epsilon =
\infty$ moduli stack is the usual moduli stack of stable maps as in
\cite{AGV08}.

Let $H^*_{\CR}(X)$ denote the Chen-Ruan cohomology of $X$. Using the $\epsilon=
\infty$ moduli stack, we define a generating series of genus-zero Gromov--Witten
invariants $J(\bt(z), q, z)$ (see \eqref{eq:J}) with generic insertion $\bt(z)
\in H^*_{\CR}(X)[z]$. This function keeps track of all genus-zero Gromov--Witten
invariants of $X$. On the $\epsilon = 0^+$ side, we have a corresponding
$H^*_{\CR}(X)$-valued series $I(q, z)$, which can be defined via equivariant
localization (see \eqref{eq:I}), and can often be explicitly computed. We will
denote these two series as the \textit{J-function} and \textit{I-function}
respectively. The wall-crossing result of \cite{Zh19P} then states the
relationship between the two functions is given by
\begin{equation} \label{eq:mirror-intro}
J( \mu(q,-z), q, z) = I(q, z)
\end{equation}
for $\mu(q, z) = [zI(q, z) - z]_+$ where $[\cdot]_+$ refers to taking the part
of the series whose terms have non-negative powers of $z$.

Using \eqref{eq:mirror-intro}, we can compute a specialized $J$-function, and
thus linear combinations of Gromov--Witten invariants, from the more easily
computable $I$-function. The invariants obtained this way are dependent on the
series $\mu(q, z)$, which may only involve few of the possible cohomology
insertions to the theory. To allow for recovering invariants with more possible
insertions, we will use the fact that the $I$-function is dependent on the GIT
presentation of $X$, whereas the $J$-function is not.

We first define a subring
\[
  \cH \subset H^*_{\CR}(X),
\]
which we call the \textit{admissible state space} (see Section
\ref{sec:extendable-classes}). We call the cohomology classes in $\cH$ the
\textit{admissible classes}. Notably, these include the ambient cohomology
classes, which are those obtained by pullback from $H^*_{\CR}(\PP(\vec w))$, but
also include non-ambient classes which are Poincar\'e dual to cycles of the form
$x_{i_1} = \cdots = x_{i_k} = 0$, where the $x_{i_j}$ are coordinate functions
on $\PP(\vec w)$. We then fix a specific choice of degree $2$ admissible classes
(see Section~\ref{sec:extension-data})
\[\{\phi_1, \dots, \phi_m\}, \quad \phi_i \in \cH \]
and form a new GIT presentation of $X$ where we add an additional $\CC^*$-action for each of the above classes
\[ X = \left[W_e \sslash_{\theta_e} (\CC^*)^{m+1} \right]. \]
Here, the weight matrix associated to the torus action and the affine scheme
$W_e \subset \CC^{n+m+1}$ are determined explicitly by the classes $\phi_i$ (see
Section \ref{sec:extension-data}). This can be seen as a GIT analog of Jiang's
$S$-extended stacky fan \cite{Ji08} for a specified $S$, hence we similarly call
the resulting presentation an \textit{extended GIT presentation}.
 
Now consider the $J$-function $J(\bt = \sum_{i=1}^m t_i \phi_i, Q, z)$, which is
a formal power series in the $t_i$ and $Q$. Our main result is that the
$I$-function computed from the extended GIT presentation recovers this series.

\begin{theorem}[Theorem~\ref{thm:I-general}, Corollary~\ref{cor:wallc-CY3}, Theorem~\ref{thm:invertibility}] 
  Given $X$ as above, there is an explicit $I$-function
  \[
    I(q_0, \dots, q_m, z)
  \]
  that is a power series in the formal variables $q_0, \dots, q_m$, and an
  explicit invertible change of coordinates
  \[ (q_0, \dots, q_m) = (Q, t_1, \dots, t_m)\]
  such that
  \[
    \textstyle
    \exp{\left(-\frac{ I_1}{zI_0}\right)}
    \frac{I}{I_0}  = J( \sum_{i=1}^m t_i \phi_i, Q, z), \]
 where $I_0 \in H^0(X) \llbracket q_0 ,\ldots, q_m \rrbracket$,
  $I_1 \in H^2(X) \llbracket q_0 ,\ldots, q_m \rrbracket$
  are two of the
  components of the $I$-function.
\end{theorem} 

The invertibility of the change of coordinates lets us extract the individual
Gromov--Witten invariants in the $J$-function. In particular, the above theorem
and the specifications of the set $\{\phi_1, \dots, \phi_m\}$ immediately imply
the following.
\begin{corollary}
  There is an algorithm for computing any genus-zero Gromov--Witten invariant of $X$ with insertions coming from $\cH \subset H^*_{\CR}(X)$.
\end{corollary}

In the CY3 case, the most interesting invariants are those whose insertions are
degree two classes that come from the cohomology of the twisted sectors of the
rigidified inertia stack $\overline{\cI} X$, as standard reductions such as the divisor equation (see
e.g.\ \cite[Theorem 8.3.1]{AGV08}) do not apply.
We include many examples in Section \ref{sec:examples} that showcase the
computations of such invariants, including an invariant with a non-ambient
insertion (see Section \ref{sec:x24}).

\subsection{Complete Intersections in Toric Stacks} \label{sec:toric-stacks}

The results of this paper can be extended to the more general case of complete intersections in toric stacks.
The main details and proofs of this generalization can be found in the second author's thesis \cite{sultanithesis}.
In this paper, we stick to the case of a CY3 in a weighted projective stack as this case highlights the main ideas of the construction while avoiding the substantial increase in notation and combinatorics in the general case.

As soon as either the Calabi-Yau or the 3-dimensionality condition on the
complete intersection $X$ is removed, it no longer suffices to extend by
admissible classes of degree $2$, but one has to consider admissible classes
$\{\phi_1, \dots, \phi_m\}$ of arbitrary (and possibly fractional) degrees.
Except for this, the general GIT extension is similar to the one discussed in this article.

Additionally, as soon as the GIT extension involves admissible classes of degree greater than $2$, the corresponding $I$-function now has terms with arbitrarily positive powers of $z$.%
\footnote{This may happen even if $X$ is semi-positive.}
As a result, the analog of Theorem~\ref{thm:I-general} has to be stated in terms of Givental's Lagrangian cone (e.g.\ see \cite{Tseng10} for an orbifold version) and proved using a technique known as Birkhoff Factorization.
A lemma about the derivatives of the $I$-function similar to Lemma~\ref{lem:I-derivative} below is crucial to proving the invertibility of the Birkhoff factorization procedure.
In the end, one obtains an explicit, though complicated algorithm for computing any genus-zero Gromov--Witten invariant of $X$ with admissible insertions.

\subsection{Related Works}

During the preparation of this paper several related works have appeared.
We mention the ones we are aware of, as well as their relation to our results. 

We start by mentioning the relation to the $S$-extended toric stack
mirror theorem and computations in \cite{CCIT15, CCIT19}, which
predate our work on this problem.
For toric stacks, the GIT extensions in this paper can be realized as
$S$-extensions via the translation outlined in \cite{CIJ18}, and
recover these earlier results.
For complete intersections, we remove the previously imposed convexity
hypotheses and include details on how to extend the data of the
defining vector bundle.
We also provide the extension for all admissible classes, which are
more than the classes parameterized by the box of the stacky fan, and
include computations when the $S$-$\sharp$ condition of \cite{CCIT19}
does not hold. 

In a closely related independent approach, Wang \cite{Wa19P, Wa23P}
also circumvents the convexity issues for complete intersections in
toric stacks, showing that the big $I$-function obtained from
quasimap theory gives a slice of the Lagrangian cone in this case.
The $I$-function formula in his paper agrees with our extended one
when applied to an appropriate GIT extension of the target, and his
wall-crossing result is also applicable to our situation to give
\eqref{eq:mirror-intro}.
Jun Wang's work has a different focus than this work.
Wang focusses on establishing general mirror theorems and a version of
the quantum Lefschetz principle, while our focus is on how to
explicitly compute genus-zero Gromov--Witten invariants.

In \cite{Gu19P}, Gu\'er\'e develops a new technique called ``Hodge
Gromov--Witten theory'', which allows computing all genus-zero and
certain higher-genus Gromov--Witten invariants of many not necessarily
convex hypersurfaces in weighted projective stacks.
The two hypersurface examples $X_7$ and $X_{17}$ that we consider in
Section~\ref{sec:examples} should be accessible to Hodge Gromov--Witten theory, and it
would be interesting to verify that it leads to the same results.

In \cite{HeSh21P}, Heath and Shoemaker develop a general quantum
Lefschetz and Serre duality statement for $2$-pointed quasimaps to not
necessarily convex orbifold complete intersections $X$ in a stacky GIT
quotient $Y = [W \sslash_\theta \CC^*]$.
Combined with quasimap wall-crossing, this allows computing the
Gromov--Witten invariants of $X$ with ambient insertions in terms of
those of $Y$.
We expect that applying the methods of Heath and Shoemaker to the
extended GIT presentation of $Y = \PP(\vec w)$ recovers the results of
this paper.

In \cite{Webb}, Webb
proves the abelian/non-abelian correspondence for orbifolds, and as a consequence, constructs an $I$-function for the Gromov--Witten theory of complete intersections in toric stacks.
In \cite{SuWe}, Webb and the second author use this and the GIT extension method in order to compute examples of Gromov--Witten invariants of complete intersections in non-abelian quotients.

\subsection*{Plan of the Paper}

In Section \ref{sec:background}, we recall background material on weighted
projective stacks and quasimap theory. To those readers familiar with this material, we note that we make a
small, but necessary, modification to the usual perfect obstruction theory for
$Q^{\epsilon}_{g, k}(X, \beta)$ (see \eqref{eq:qpot}). In Section
\ref{sec:extended-git}, we introduce the definition and properties of the
admissible state space $\cH$. We then detail the construction of the extended
GIT presentation from a choice of classes $\phi_i \in \cH$. Section
\ref{sec-I-computation} is devoted to the localization computation of the
$I$-function obtained from the extended GIT. In Section \ref{sec:mirror}, we
apply the wall-crossing \eqref{eq:mirror-intro}. We show that our $I$-function
results in a $J$-function with generic insertion containing all the classes
$\phi_i$, and show that the mirror map is invertible. Finally, in Section
\ref{sec:examples}, we show how this method can be used to explicitly compute
invariants for many examples.

\subsection*{Conventions and Notation}

All stacks and schemes are defined over $\CC$. Cohomology groups are always
assumed to have rational coefficients. 

We will often use the word ``orbifold'' interchangeably with the phrase ``smooth
Deligne--Mumford stack,'' but since we work in algebraic geometry, will always
use the latter definition.

We will also assume that the reader is fairly comfortable with the basics of
orbifold Gromov--Witten theory; good references for the necessary definitions in
the subject are \cite{Ab08, AGV08}, and the conventions of the theory presented
in these references are the ones we will use.
The terms ``orbifold curve'' or ``stacky curve'' will refer to a (balanced)
twisted curve in the sense of \cite{AbVi02} (in particular with non-trivial
isotropy only at special points, nodes are balanced, etc.).

Throughout the paper, we assume that $\gcd(w_0, \dots, w_n) = 1$ where $w_i$ are
the weights of the weighted projective stack. 

For $k \in \QQ$, we define $\lfloor k \rfloor$ to be the largest integer $n \in
\ZZ$ such that $n \leq k$, and $\langle k \rangle = k - \floor{k} $ to be the
fractional part of $k$.

\subsection*{Acknowledgments}

The authors would like to thank J\'er\'emy Gu\'er\'e, Yunfeng Jiang, Sheldon Katz,
Aaron Pixton and Rachel Webb for many helpful discussions.
The third author is partially supported by
the National Key Res.\ and Develop.\ Program of China \#2020YFA0713200.
The first author was partially supported by NSF grants DMS-2239320,
DMS-1901748 and DMS-1638352.
The second author is particularly grateful to Yongbin Ruan for
suggesting the initial problem of studying non-convex complete
intersections.
The third author is also partially supported by
Shanghai Pilot Program for Basic Research-Fudan Univ.\ 21TQ1400100 (22TQ001) and
NSFC for Innovative Research Groups (No. 12121001).
The third author is
a member of the Key Laboratory of Mathematics for Nonlinear Sciences,
Fudan University, and would also like to thank the
support of Alibaba Group as a DAMO Academy Young Fellow, 
the support of Xiaomi Corporation as a Xiaomi Young Fellow.

\subsection*{Data availability statement}

Data sharing is not applicable to this article as no datasets were generated or analysed during the current study.

\subsection*{Declaration of AI usage}

During the journal revision process, we used ChatGPT 5.6 Sol to identify
typographical errors, notational inconsistencies, and other minor, locally
fixable errors in the manuscript. No edit to the manuscript is directly made by AI.

\section{Preliminaries} \label{sec:background}

We begin by reviewing some background material related to
weighted projective stacks and the theory of quasimaps.

\subsection{Weighted Projective Stacks} 
Let $w_0, \dotsc, w_n$ be positive integers.
The weighted projective stack $\PP(\vec w) = \PP(w_0, \dotsc, w_n)$ is defined as
the GIT stack quotient
\begin{equation*}
  \PP(w_0, \dotsc, w_n)
  = \left[\CC^{n+1} \sslash_\theta \CC^*\right] = \left[(\CC^{n + 1} \setminus
    \{0\})/ \CC^*\right],
\end{equation*}
where $\theta$ is the character $\theta(\lambda) = \lambda$ and $\CC^*$ acts with weights $(w_0, \dotsc, w_n)$, or in other
words, the $\CC^*$-action is given by
\begin{equation*}
  \lambda \cdot (x_0, \dotsc, x_n)
  = (\lambda^{w_0} x_0, \dotsc, \lambda^{w_n} x_n).
\end{equation*}
In particular, weighted projective stacks are proper Deligne--Mumford stacks over
$\CC$. 

Just as in the non-weighted case, the Picard group of $\PP(w_0, \dots, w_n)$ is $\ZZ$, where all the line bundles are given by linearizations of the trivial bundle on $\CC^{n+1}$. That is, the line bundle $\cO(n)$ is defined as 
\begin{equation}\label{eq:line-bundle}
\cO(n) := \left[(\CC^{n + 1} \setminus \{0\}) \times \CC / \CC^*\right] 
\end{equation}
where the action in the fiber coordinate $v$ is given by the character $\lambda \cdot v = \lambda^n v$. 

A map from a stack $S$ to $\PP(w_0, \dotsc, w_n)$ corresponds to a
line bundle $L$ on $S$ and $(n+1)$ sections
$x_i\in H^0(S, L^{\otimes w_i}), i = 0 ,\ldots, n$ without common
zeros.
In particular, the identity map of $\PP(w_0, \dotsc, w_n)$ corresponds
to the line bundle $\cO(1)$ on
$\PP(w_0, \dotsc, w_n)$ and the tautological sections $x_i$ of
$\mathcal O(w_i) = \mathcal O(1)^{\otimes w_i}$.
These fit into the Euler sequence
\begin{equation} \label{eq:wEuler}
  0 \to \cO \to \bigoplus_{i = 0}^n \cO(w_i) \to T_{\PP(w_0, \dotsc, w_n)} \to 0
\end{equation}
describing the tangent bundle of $\PP(w_0, \dotsc, w_n)$.
The canonical bundle of $\PP(w_0, \dotsc, w_n)$ is thus
$\cO(-w_0 - \dotsb - w_n)$.

\subsection{Chen--Ruan Cohomology} \label{sec:CR-cohom}
The state space for orbifold Gromov--Witten theory is the Chen--Ruan cohomology
$H^*_{\mathrm{CR}}(X)$
\cite{CR04} 
of the target $X$, which arises as a particular graded ring structure on the
singular cohomology of the inertia stack of $X$. We review the necessary details
for our situation.

Let $\cI\PP(\vec{w})$ denote the inertia stack of $\PP(\vec w)$, defined as the fiber product
\[ \begin{tikzcd}
\cI \PP(\vec w) \ar[r] \ar[d] & \PP(\vec w) \ar[d, "\Delta"] \\
\PP(\vec w) \ar[r, "\Delta"] & \PP(\vec w ) \times \PP(\vec w) 
\end{tikzcd} 
\]
where $\Delta$ is the diagonal morphism. The inertia stack has a decomposition into connected components
\[
  \cI\PP(\vec{w}) = \bigsqcup_{\lambda \in \mathbb C^*} [\big(
  \mathbb C^{n+1} \setminus \{0\}
  \big)^{\lambda} / \mathbb C^*]
\]
where $\big(
\mathbb C^{n+1} \setminus \{0\} \big)^{\lambda}$ is the fixed-point subscheme, which is nonempty if and only if
$\lambda = \exp(2\pi\sqrt{-1} \alpha)$ where $\alpha$ is of the form
\begin{equation} 
\label{def:alpha}
  \alpha = \frac{k}{w_i},\quad \text{for some } i = 0, \ldots, n, \text{ and }k = 0
  ,\ldots, w_i - 1.
\end{equation}
Subsequently, we will set
\begin{equation} \label{eq:P-alpha}
 \mathbb P_{\alpha} :=  \left[\big(\mathbb C^{n+1} \setminus \{0\}
  \big)^{\lambda} / \mathbb C^*\right], \text{ for }\lambda = \exp(2\pi\sqrt{-1} \alpha)
  \end{equation}
so that we can write the decomposition as $\cI \PP(\vec w) = \bigsqcup_\alpha \PP_\alpha$ over all distinct $\alpha$ in \eqref{def:alpha}.

Now let $X$ be a smooth complete intersection in $\PP(\vec w)$ defined by a generic section of a split bundle $E = \bigoplus_j \cO(b_j)$, or equivalently the vanishing of generic homogeneous polynomials $F_j$ of degree $b_j$. 

Using $\cI X \subset \cI \mathbb P(\vec {w})$, we have a similar decomposition for $\cI X$
\begin{equation}
  \label{eq:X-alpha}
  \cI X = \bigsqcup_\alpha X_\alpha, \quad X_{\alpha} = \mathbb P_{\alpha} \cap \cI X.
\end{equation}
We call the components $X_\alpha$ for $\alpha \neq 0$ \textit{twisted sectors},
and the component $X_0 \cong X$ the \textit{untwisted sector}.

The rigidified inertia stack (cf.\ \cite[Section~3]{AGV08}) is closely related. It is defined to be
  \[
    \overline{\cI} X = \bigsqcup_\alpha X_{\alpha} \thickslash \langle
    \exp(2\pi\sqrt{-1} \alpha) \rangle,
  \]
  where $X_{\alpha} \thickslash \langle \exp(2\pi\sqrt{-1} \alpha) \rangle$ is
  the rigidification of $X_{\alpha}$ obtained by quotienting the automorphism
  group by the subgroup generated by $\exp(2\pi\sqrt{-1} \alpha)$.
  In particular, $\overline{\cI} X$ and $\cI X$ have the same cohomology with
  rational coefficients.
  However, it will be 
  more convenient to use $\overline{\cI} X$ when defining the invariants (cf.\
  Remark~\ref{rmk:state-space-convention}).

\begin{lemma} \label{lem:twisted-CI}
  The twisted sector $X_\alpha$ is a complete intersection in $\PP_\alpha$,
  defined by the restriction of the polynomials $F_j$ such that $\alpha b_j
  \in \ZZ$.
\end{lemma}
\begin{proof}
  First note that $F_j |_{\PP_\alpha} = 0$ whenever $\alpha b_j
  \not\in \ZZ$. This is because $\PP_\alpha$ is the linear subspace whose
  coordinates are those $x_i$ such that $\alpha w_i \in \ZZ$ by
  \eqref{eq:P-alpha}.

  Now consider any closed point $p \in X_\alpha$. Because $X$ is a complete
  intersection, the differential map
  \[
    dF\colon T_p\PP(\vec w) \longrightarrow E|_p
  \]
  is surjective, where $F = (F_1 ,\ldots, F_{n-3})$ and 
  $E := \bigoplus_{j=1}^{n-3} \cO(b_j)$ as before.
  Note that $\exp(2\pi \sqrt{-1} \alpha)$ acts on $\mathcal O(b_j)|_{p}$ by
  $\exp(2\pi\sqrt{-1} \alpha b_j)$ and 
  $dF$ is equivariant.
  Taking the subspaces fixed by the action of $\exp(2\pi \sqrt{-1} \alpha)$, we
  get a surjective map
  \[
    \textstyle
    T_p\PP_{\alpha} \longrightarrow \bigoplus_{\alpha b_j \in \mathbb
      Z}\mathcal O(b_j)|_p.
  \]
  This proves the lemma.
\end{proof}

We also note for future use a dimension restriction on the twisted sectors when $X$ is CY3.
\begin{lemma}  \label{lem:CY3-sector-dim}
When $X$ is CY3, we have that $\dim(X_\alpha) \neq 2$ for all $\alpha$. 
\end{lemma}
\begin{proof}
For any closed point $p \in X_\alpha$, the tangent space $T_pX_\alpha$ is identified with the subspace of the tangent space $T_pX$ that is fixed by the action of $\exp(2\pi\sqrt{-1}\alpha)$. The Calabi--Yau condition implies that the determinant of the action is 1, hence the fixed subspace cannot be of codimension one. 
\end{proof}

  The Chen--Ruan cohomology group is defined to be, as a vector space,
  \begin{equation*}
    H^*_{\CR}(X) =  H^{*}(\overline{\cI} X, \mathbb Q)  \cong H^{*}({\cI} X, \mathbb Q),
  \end{equation*}
  but with different grading, pairing and ring structure.
  Under the isomorphism $\varpi^*: H^{*}(\overline{\cI} X) \cong H^{*}({\cI} X)$ induced by the rigidification map
  \[
    \varpi: \cI X \to \overline{\cI} X,
  \]
  $H^{i}( X_{\alpha})$ is shifted into degree $i +
  2\iota_{\alpha}$, where $\iota_\alpha \in \QQ$ denotes the \textit{age}
or \textit{degree-shifting number} of the tangent bundle $TX$ pulled back to
$X_\alpha$ (see \cite[Section~7.1]{AGV08}).

Age is additive in exact sequences, so the numbers $\iota_\alpha$ can be easily computed from the Euler sequence \eqref{eq:wEuler}, as well as the normal sequence 
\begin{equation*} 
  0 \to TX \to (T\PP(\vec{w}))|_X \to E|_X \to 0.
\end{equation*}
Using that the age of $\cO(n)$ on $X_\alpha$ is $\bangle{\alpha n }$, we obtain
\begin{equation*} 
  \iota_{\alpha}
  = \sum_{i=0}^n \bangle{\alpha w_i}  - \sum_{j=1}^{n-3} \bangle{\alpha b_j}
  = \sum_{j=1}^{n-3} \floor{\alpha b_j} - \sum_{i=0}^{n} \floor{\alpha w_i}.
\end{equation*}

Throughout this paper, the degree of a class will always refer
  to the Chen--Ruan degree unless otherwise noted.

The inertia stack $\cI X$ has an
involution $\Inv$ given component-wise by the involutions $\Inv_\alpha: X_\alpha
\to X_{\bangle{1-\alpha}}$ defined by inverting the banding on the gerbe
(\cite[Section 3.5]{AGV08}). It also induces an involution on the
  rigidified inertia stack $\overline{\mathcal I}X$.

  The non-degenerate pairing $(\cdot, \cdot)$ on
$H^*_{\CR}(X)$ is defined to be 
\begin{equation} \label{def:pairing}
 (\gamma_1, \gamma_2) = \int_{\cI X} \varpi^*\gamma_1 \cup \Inv^*\varpi^*
 \gamma_2, \quad \gamma_1, \gamma_2 \in H^*_{\mathrm{CR}}(X).
 \end{equation}

The Chen--Ruan product $*$ is then defined by the relation
\begin{equation} \label{eq:CR-product}
 ( \gamma_1 * \gamma_2, \gamma_3) = \langle \gamma_1, \gamma_2, \gamma_3 \rangle_{0, 3, 0}^X 
\end{equation}
where $\gamma_1, \gamma_2, \gamma_3 \in H^*_{\CR}(X)$, and the right-hand side
is the three-pointed genus-$0$, degree-$0$ Gromov--Witten invariant. As a
general property we have for $\gamma_1 \in H^*(X_\alpha)$ and $\gamma_2 \in
H^*(X_\beta)$ that $\gamma_1 * \gamma_2 \in H^*(X_{\bangle{\alpha+\beta}})$.
In
particular, we have that the Chen--Ruan product agrees with the usual cup product
on $H^*(X)$ when $\alpha=\beta=0$. However, more explicit descriptions of the
product structure are difficult since the three-pointed invariants can be
non-trivial to compute. We refer to \cite{Ji07, BMP09} for more details on the
ring structure of the Chen--Ruan cohomology of weighted projective stacks.

\subsection{Quasimaps} \label{sec:quasimap}
We compute our $I$-functions via quasimaps to GIT targets, first constructed by Marian, Oprea and Pandharipande in
\cite{MOP11} (under the name of ``stable quotients'') and by Ciocan-Fontanine and Kim \cite{CKM14} for
schemes, and later extended to Deligne--Mumford stacks by Cheong, Ciocan-Fontanine and Kim in \cite{CCK15}.
We review some of the definitions and constructions below.

Let $(W, G, \theta)$ be a tuple consisting of
\begin{itemize}
\item $W$ an affine scheme of finite type over $\CC$.
\item $G$ a reductive algebraic group acting on $W$,
\item $\theta\colon G \to \CC^*$ a character of $G$.
\end{itemize}
For the rest of the paper, we will always assume that the $\theta$-semistable
locus $W^{\mathrm{ss}}$ is nonempty and smooth, and is equal to the
$\theta$-stable locus $W^{\mathrm{s}}$. From this data, we can construct the
following two stack quotients
\[
  X:= \left[W\sslash_\theta G\right] \hookrightarrow \left[W/G\right] = \fX,
\]
where the inclusion is an open embedding induced by the open embedding
$W^{\mathrm{ss}} \hookrightarrow W$. Note that by our assumption, the GIT stack
quotient $X$ will always be a quasi-compact Deligne--Mumford stack and the
coarse moduli morphism to the non-stacky GIT quotient $X \to W \sslash_\theta G$
is proper. On the other hand, the stack quotient $\fX$ is in general an Artin
stack.

Now let $(\cC, x_1, \dots, x_k)$ denote a $k$-pointed twisted curve with
balanced nodes, as in \cite[Section~4]{AGV08}.
\begin{definition} \label{def:quasimap} A $k$-pointed prestable quasimap to $X$
  is a representable morphism
  \[
    f\colon (\cC, x_1, \dots, x_k) \to \mathfrak{X}
  \]
  where $f^{-1}(\mathfrak{X} \setminus X)$ is a zero-dimensional substack that
  is disjoint from the nodes and markings.
\end{definition}
The substack $f^{-1}(\mathfrak{X} \setminus X)$ is called the \emph{base locus}
of the quasimap, and its points are referred to as \emph{base points}.

The \emph{curve class} (or \emph{degree}) $\beta$ of a quasimap is defined to be
the homomorphism
\begin{equation*}
  \beta \in \Hom(\Pic(\mathfrak{X}), \QQ), \quad
  \beta(L) = \deg(f^*L), ~\forall L\in \mathrm{Pic}(\mathfrak X),
\end{equation*}
where the degree of a line bundle on a twisted curve is defined as in
\cite{AGV08}.

Let $e$ denote the least common multiple of the integers $|\Aut(p)|$, where $p$
is a geometric point of $X$, and let $L_\theta$ be the line bundle on $\fX$
induced by $\theta$, as in \eqref{eq:line-bundle}. Let $\phi: (\cC, x_1, \dots,
x_k) \to (C, \underline{x_1}, \dots, \underline{x_k})$ be the coarse moduli map
of a twisted curve. For any $\epsilon \in \QQ_{>0}$, a quasimap $f$ is said to
be $\epsilon$-stable if
\begin{enumerate}[(i)]
\item the $\mathbb Q$-line bundle
  $\phi_*\left((f^*L_\theta)^{e}\right)^{\epsilon/e} \otimes \omega_{C, \log}$
  is positive, i.e.\ it has positive degree on each irreducible component of
  $C$, where $\omega_{C, \log} = \omega_C(\sum_{i=1}^k \underline{x_i})$ is the
  log-dualizing sheaf of $C$,
\item $\epsilon \ell(x) \leq 1$ for all base points $x \in \cC$, where $\ell(x)$
  is the length of the base locus scheme at the point $x$ (see \cite[Def.
  7.1.1]{CKM14}).
\end{enumerate}
For this paper, we will primarily concern ourselves with the limits $\epsilon
\to +\infty$ and $\epsilon \to 0^+$. These limits are well-defined, resulting in
stability conditions denoted by $\epsilon = \infty$ and $\epsilon = 0^+$
respectively.
For $\epsilon = \infty$, an $\epsilon$-stable quasimap is the same as a twisted
stable map, in the sense of \cite{AGV08}.
For $\epsilon = 0^+$, the length condition is trivially satisfied, and the
positivity condition disallows any rational tails.

For $(W, G, \theta)$, there exists a moduli stack $Q_{g,k}^\epsilon(X, \beta)$
of genus-$g$, $k$-pointed $\epsilon$-stable quasimaps to $X$ with degree
$\beta$. These stacks are all Deligne--Mumford stacks, proper over the affine
quotient $W\sslash_0 G$. For each moduli stack, there is a universal curve $\pi:
\mathfrak C^\epsilon \to Q_{g,k}^\epsilon(X, \beta)$ and a universal map $f:
\mathfrak C^\epsilon \to \mathfrak X$. In addition, there is a natural forgetful
morphism from $Q_{g, k}^\epsilon(X, \beta)$ to the moduli space of prestable
twisted curves with genus $g$ and $k$ markings, $\mathfrak M := \mathfrak
M_{g,k}^{\mathrm{tw}}$.

For the purpose of constructing a perfect obstruction theory, we require the
following additional data:
\begin{itemize}
\item $G$-representations $V$ and $V_E$,
\item A $G$-equivariant closed embedding $W \hookrightarrow V$,
\item A $G$-equivariant morphism $s\colon V \to V_E$,
\end{itemize}
such that
\begin{itemize}
\item $W$ is identified with the scheme-theoretic zero locus of $s$,
\item $\dim W^{\mathrm{ss}} = \dim V - \dim V_E$.
\end{itemize}
Note that we do not require that the affine scheme $W$ has only l.c.i.\
singularities as in \cite{CKM14}, as our later constructions may produce $W$
that are not l.c.i. Despite this, the moduli stacks $Q_{g,k}^\epsilon(X,
\beta)$ still carry a perfect obstruction theory. 

In what follows, we will use
the language of derived stacks to construct the perfect obstruction theory. We refer to \cite{khan2023lectures} for the conventions used, as well as a good reference for readers unfamiliar with derived algebraic geometry. 
For those who prefer a more classical approach, we outline an alternative construction in Remark \ref{rem:POT-no-derived} that avoids the use of derived stacks.

Let $E = [V \oplus V_E / G]$, so that $s$ may be identified with a section
$s\colon [V / G] \to E$ of the vector bundle $E \to [V / G]$. Then define
$\fX^{\der}$ as the derived zero locus of $s$, that is the derived stack fitting
into a (homotopy) cartesian diagram
\begin{equation} \label{eq:int-diagram}
  \begin{tikzcd}
    \fX^{\der} \ar[hook, r] \ar[hook,d] & \left[V/G\right] \ar[d, "s"]\\
    \left[V/G\right] \ar[r, "0"] & E
  \end{tikzcd}.
\end{equation}
By the universal property of fiber products, there is a closed
embedding\footnote{Note there is a fully faithful functor from the category of
  classical stacks to derived stacks, defined by giving the classical stack the
  trivial derived structure, and in this way, we may view $\fX$ also as a
  derived stack.} $i\colon \fX \to \fX^\der$ inducing an isomorphism on
underlying classical Artin stacks. In particular, $\fX^{\der}$ is a derived
extension of $\fX$. As in \cite[Example 8.7.4]{khan2023lectures}, $\fX^{\der}$
is quasi-smooth, i.e.\ the cotangent complex of the derived stack $\fX^{\der}$
is of Tor-amplitude $[-1, 1]$. Moreover, $\LL_{\fX^{\der}}$ fits into a
distinguished triangle
\begin{equation} \label{eq:derived-cotangent} \LL_{[V/G]}|_{\fX^{\der}} \to
  \LL_{\fX^{\der}} \to E^\vee[1] \xrightarrow{+1},
\end{equation}
see e.g.\ the proof of \cite[Proposition 2.3.8]{Khan2019virtual}. As a special
case, if $W$ is l.c.i., then $\fX = \fX^{\der}$, and $\LL_{\fX^{\der}}$ recovers
the ordinary cotangent complex of $[W/G]$.

Let $\TT_{\fX^\der} = (\LL_{\fX^{\der}})^\vee$ denote the tangent complex of
$\fX^{\der}$.
\begin{lemma} \label{lem:qpot}
  $Q_{g, k}^\epsilon(X, \beta)$ over $\mathfrak M$ admits a relative perfect obstruction theory 
   of the form
  \begin{equation} \label{eq:qpot}
    \left(R\pi_*\left(f^*i^*\TT_{\fX^{\der}}\right) \right)^\vee \to \LL_{Q_{g,
        k}^{\epsilon}(X, \beta)/\mathfrak M}.
  \end{equation}
\end{lemma}
\begin{proof}
  The mapping stack $\Hom_{\mathfrak M}(\mathfrak C, \fX \times \mathfrak M)$,
  where $\mathfrak C$ is the universal curve on $\mathfrak M$, also admits a
  derived extension by using the internal Hom in the category of derived stacks
  (see \cite[Theorem~5.1.1]{HLP23}). Since the moduli stacks
  $Q_{g,k}^\epsilon(X, \beta)$ are open substacks of $\Hom_{\mathfrak
    M}(\mathfrak C, \fX \times \mathfrak M)$, it follows from \cite[Proposition
  2.1]{STV15} that they also inherit a derived extension which we will denote by
  $Q_{g,k}^{\epsilon, \der}(X, \beta)$. The following diagram relates the
  classical stacks and their derived extensions
  \begin{equation} \label{eq:derived-diagram}
    \begin{tikzcd} 
      \fX \ar[r, "i"] & \fX^{\der} \\
      \mathfrak C \times_{\mathfrak M} Q_{g,k}^\epsilon(X, \beta) \ar[u, "f"] \ar[swap, d, "\pi"] \ar[r] & \mathfrak C \times_{\mathfrak M} Q_{g,k}^{\epsilon, \der}(X, \beta) \ar[swap, u, "\tilde f"] \ar[d, "\tilde \pi"] \\
      Q_{g,k}^\epsilon(X, \beta) \ar[r, "j"] & Q_{g,k}^{\epsilon, \der}(X,
      \beta).
    \end{tikzcd}
  \end{equation}
  Here, the horizontal arrows are the natural closed embeddings. By the proof of
  \cite[Proposition 4.3.1]{MR2018}, we have that
  \[
    \LL_{Q_{g,k}^{\epsilon, \der}(X, \beta)/\mathfrak M} \simeq
    \left(R\tilde\pi_*\left(\tilde f^*\TT_{\fX^{\der}}\right) \right)^\vee.
  \]
  and by \cite[Proposition 1.2]{STV15}, the morphism
  \begin{equation*}
    j^* \LL_{Q_{g,k}^{\epsilon, \der}(X, \beta)/\mathfrak M} \to \LL_{Q_{g,k}^{\epsilon}(X, \beta)/\mathfrak M}
  \end{equation*}
  induced by $j$ defines an obstruction theory on $Q_{g,k}^\epsilon(X, \beta)$
  relative to $\mathfrak M$. Furthermore, we note that by commutativity of the
  top square and applying the base-change formula (e.g.\
  \cite[Lemma~A.1.3]{HLP23}) for the bottom square in
  \eqref{eq:derived-diagram}, we have that
  \[ j^*\left(R\tilde\pi_*\left(\tilde f^*\TT_{\fX^{\der}}\right) \right)^\vee
    \simeq \left(R\pi_*\left(f^*i^*\TT_{\fX^{\der}}\right) \right)^\vee. \]

  It remains to show that this complex is perfect. To see this, we note that the
  affine scheme $W$ has a quasi-smooth derived extension $W^{\der}$ via a
  pre-quotient lift of the intersection diagram \eqref{eq:int-diagram}, and this
  has a corresponding two-term tangent complex $\TT_{W^{\der}}$. Additionally,
  our assumptions imply the derived structure is trivial on $W^{\ss}$ so that
  the restriction $\TT_{W^{\der}}|_{W^{\ss}}$ is a one-term complex. This
  implies that for every geometric fiber $C$ of $\pi$ the sheaf
  $H^1(f^*\TT_{W^{\der}}|_C)$ is torsion. It then follows from the same
  arguments of \cite[Section 4.5]{CKM14}, applied to $\TT_{W^{\der}}$, that
  \eqref{eq:qpot} is perfect. \end{proof}

\begin{remark}
\label{rem:POT-no-derived}
  The perfect obstruction theory \eqref{eq:qpot} may also be constructed without
  derived algebraic geometry by noting that the complex $\TT_{W^{\der}}$ is the
  perfect obstruction theory on $W$ obtained from applying the basic example of
  Behrend--Fantechi \cite[Section 6]{BeFa97} to the pre-quotient lift of the
  intersection diagram \eqref{eq:int-diagram}. One can then argue as in
  \cite[Section 4.5]{CKM14}, with some slight modifications to show that the
  resulting complex on $Q_{g,k}^{\epsilon}(X, \beta)$ is indeed perfect.
\end{remark}

When $W$ is l.c.i., $\TT_{\fX^\der}$ is the classical tangent complex, and hence
\eqref{eq:qpot} matches the obstruction theory of \cite{CCK15}. Let
$[Q_{g,k}^\epsilon(X, \beta) ]^{\vir}$ denote the virtual cycle associated to
this perfect obstruction theory \cite{LT98,BeFa97}.

Since the markings are disjoint from the base locus, we can define evaluation
morphisms
\[
  \ev_i: Q_{g,k}^\epsilon(X, \beta) \to \overline \cI X
\]
as in \cite[Section~4.4]{AGV08} and \cite[Section~2.5.1]{CCK15}.
And the quasimap invariants are defined to be
\begin{equation*} 
  \langle \gamma_1\psi_1^{a_1}, \dots, \gamma_k\psi_{k}^{a_k}
  \rangle_{g,k,\beta}^{\epsilon} = \int_{[Q_{g,k}^\epsilon(X, \beta)]^\vir}
  \prod\limits_{i=1}^k \ev_i^*\gamma_i \psi_i^{a_i}, \quad \gamma_i \in
  H^*_{\CR}(X) = H^*(\overline{\cI}X),
\end{equation*}
where the $\psi_i$'s are the $\psi$-classes on the coarse curves.
When $\epsilon = \infty$, the moduli space becomes the usual moduli of twisted
stable maps, denoted by  $\mathcal K_{g, k}(X, \beta)$ in \cite{AbVi02}. And these invariants are precisely
the Gromov--Witten invariants of $X$.

\begin{remark}
  \label{rmk:state-space-convention}
  There are two different ways of treating orbifold markings, as is explained in \cite[Section~6.1.3]{AGV08}.
  The convention here is the same as \cite{CCK15}, but different from \cite{Zh19P}.
  Namely, the orbifold markings here are gerbes that do not carry canonical trivializations,
  so that the evaluation maps land in the \textit{rigidified} inertia stack $\overline \cI X$ rather
  than the inertia stack $\cI X$ (cf.~\cite{Ab08}).

  We now explain how those conventions compare with each other.
  The natural morphism $\cI X \to \overline{\cI}X$ induces \textit{two}
  isomorphisms between $H^*(\overline{\cI} X)$ and $H^*({\cI} X)$, namely
  pushforward and pullback. They differ by some locally constant factor.
  In \cite{CCK15}, $H^*(\overline{\cI} X)$ is used as the state space, while
  \cite{Zh19P} uses $H^*({\cI} X)$ for a technical reason. As is explained in
  \cite[Section~1.6]{Zh19P}, the correlators in \cite{CCK15} and \cite{Zh19P}
  coincide when one identifies $H^*(\overline{\cI} X)$ and $H^*({\cI} X)$ via pushforward.

  In this paper, we technically \textit{only} use $H^*(\overline{\cI} X)$ as our state
  space, and the correlators are the same as those in \cite{CCK15}. However, for notational convenience, especially in
  Section~\ref{sec:extendable-classes}, we will often use ${\cI} X$.
  When doing so, we tacitly identify $H^*(\overline{\cI} X)$ and $H^*({\cI}
  X)$ via pullback. Thus, for example, when we write the fundamental class
  $\mathbf 1_{X_{\alpha}}$, we will actually mean the fundamental class of the
  corresponding component in $\overline{\cI} X$.

\end{remark}

\subsection{Stacky Loop Space} \label{sec:I-function}
The $I$-function is usually defined via $\CC^*$-localization on a
variant of the $\epsilon = 0^+$ quasimap moduli space known as the $0^+$-graph-quasimap
moduli space \cite[Section~2.5]{CCK15}.
However, the $\epsilon = 0^+$ stability conditions force the relevant quasimaps in the computation to take a simple form, with source curves all isomorphic to $\PP(1,r)$. As a result, we may define our $I$-function via an alternative, simpler moduli space,
known as the stacky loop space, following
\cite[Section~4.2]{CCK15}.
By \cite[Lemma~4.8]{CCK15}, the two $I$-function definitions are equivalent. 

For any positive integer $r$, let
\[
Q_{\PP(1,r)}(X, \beta) \subset \Hom(\PP(1,r), \mathfrak{X})
\]
be the substack of representable morphisms $\PP(1,r) \to \mathfrak{X}$
which are quasimaps of curve class $\beta$, i.e., the generic point of
$\PP(1,r)$ falls into $X$.
Here, the domain curve is a fixed weighted projective line $\PP(1,r)$ with one
marking at the possibly stacky point. We will refer to $\CC$-points on
$\PP(1,r)$ via the coordinates $[x:y]$ on the underlying scheme $\Proj
\CC[x,y]$, where $x$ has weight $1$ and $y$ has weight $r$. We also set notation
for the special points
\begin{equation} \label{eq:special-points}
0 := [1:0], \quad \infty: = [0:1]
\end{equation}
so that we can say the marking has underlying point $\infty$. 

We define the stacky loop space as
\begin{equation*}
  Q_{\PP(1, \bullet)}(X, \beta) := \bigsqcup_{r = 1}^\infty Q_{\PP(1, r)}(X, \beta).
\end{equation*}
Because our morphisms are representable and $X$ is a Deligne--Mumford stack of finite type, only
finitely many terms of the coproduct are nonempty.

This moduli stack has a universal curve
$\pi\colon \mathfrak C_\beta \to Q_{\PP(1 ,\bullet)}(X, \beta)$, which is trivial over
each component of $Q_{\PP(1 ,\bullet)}(X, \beta)$, together with a
universal morphism $f\colon \mathfrak C_\beta \to \mathfrak{X}$.
It also comes equipped with an absolute perfect obstruction theory
\begin{equation*} 
  \phi_{Q_{\PP(1, \bullet)}(X, \beta)}\colon (R\pi_*f^* \TT_{\fX^\der})^\vee \to \LL_{Q_{\PP(1, \bullet)}(X, \beta)}
\end{equation*}
and an associated virtual class $[Q_{\PP(1, \bullet)}(X, \beta)]^\vir$.
In the case that $W$ is not l.c.i., this requires arguments similar to the ones of the previous section.

The stacky loop space has a $\CC^*$-action on it induced, by precomposition, from the $\CC^*$-action on $\PP(1,r)$ given by
\begin{equation} \label{eq:action}
\lambda \cdot [x : y] = [x : \lambda y].
\end{equation}
Looking at the $\CC^*$-fixed loci in the moduli stack, there is a component, which we will denote by $F_\beta$, which parameterizes quasimaps where the only
base point is located at $0 =[1:0]$, and the entire class $\beta$ is
supported over $0$, i.e.\ it is a base point of length $\beta(L_\theta) \in \QQ$. 
The perfect obstruction theory is equivariant under this
$\CC^*$-action, so $F_\beta$ inherits a virtual fundamental class
$[F_\beta]^\vir$ and a virtual normal bundle
$N^{\vir}_{F_\beta / Q_{\PP(1, \bullet)}(X, \beta)}$ as in \cite{GrPa99}.

  We view $\infty$ as the unique (fixed) marking, and label it by $\star$. It is a gerbe over
  $Q_{\PP(1, \bullet)}(X, \beta)$ banded by the group $\mu_r$ of $r$th roots of unity.
  Let $\ev_\star \colon Q_{\PP(1, \bullet)}(X, \beta) \to \ocI X$
  be the evaluation map at $\star$, and let $\hat{\ev}_\star$ be the composition of $\ev_\star$
  with the involution on $\overline{\cI} X$ inverting the band.
  Let $\mathbf r$ be the $\mathbb Z$-valued function on
  $\overline{\mathcal I}X$ whose value on $X_{\alpha}$ is the order of group
  generated by $\exp(2\pi\sqrt{-1} \alpha)$. Thus $\mathbf r \circ
  \hat{\ev}_{\star}$ is constant on $Q_{\PP(1, r)}(X, \beta)$ of value $r$.
\begin{definition}
  The $I$-function is defined as
    \begin{equation} \label{eq:I}
      I(q, z) = \sum_{\beta \in \Eff(W, G, \theta)} q^\beta \mathbf r \cdot (\hat\ev_\star)_* \left(
        \dfrac{ [F_\beta]^\vir}{ e^{\CC^*}(N^{\vir}_{F_\beta / Q_{\PP(1, \bullet)}(X,
            \beta)})} \right),
    \end{equation}
  where $\Eff(W, G, \theta)$ is the set of all possible degrees of quasimaps
  from $\PP(1, r)$ to $(W, G, \theta)$ and $z$ is the localization parameter with
  respect to the $\CC^*$-action.
\end{definition}
Since we are only interested in computing numbers in this paper, we will
view the $I$-function as a series valued in the cohomology
$H^*_{\mathrm{CR}}(X)$ via Poincar\'e duality.

\section{The Extended GIT}\label{sec:extended-git}

In this section, we will describe how to modify the GIT presentation of $X$
so that the resulting $I$-function is expressive enough to capture the
desired Gromov--Witten invariants of the target.

For the rest of the paper, we work with a rank $n-3$ split vector bundle $E$ on $\PP(\vec w)$ and a smooth complete intersection $X$ defined by a generic section of $E$, fitting in the following diagram
  \[
  \begin{tikzcd}
   & E = \bigoplus_{j=1}^{n-3} \cO(b_j) \arrow{d} \\
   V(s) = X \arrow[hookrightarrow]{r} & \PP(\vec w) \arrow[bend right=25, swap] {u}{s}
   \end{tikzcd}.
   \]
We require that $\sum_{j=1}^{n-3} b_j = \sum_{i=0}^n w_i$ so that $X$ satisfies
the Calabi--Yau condition. We will also assume that $X$ is not contained in any
linear subspace of $\PP(\vec w)$, which, in addition to the assumption that
$\gcd(w_0 ,\ldots, w_n) = 1$, ensures that the generic isotropy group of $X$ is
trivial.

We initially start with the GIT presentation $X = [W \sslash_\theta \CC^*]$ that
is inherited from the standard GIT presentation \eqref{eq:wp-stack} of $\PP(\vec
w)$. Here, $W \subset \CC^{n+1}$ is the affine cone of $X$, and can be written
as the vanishing set $W = V(F_1, \dots, F_{n-3})$, where $F_j$ is a
quasi-homogeneous degree $b_j$ polynomial in the weighted variables $x_0, \dots,
x_n$.

We always assume that the defining equations are general.

\subsection{Admissible State Space} \label{sec:extendable-classes}
We start by defining a subgroup of cohomology classes $\cH \subset H^*_{\CR}(X)$
that we will call the \textit{admissible state space}. The classes in $\cH$ will
be called \textit{admissible classes}, and represent all the possible insertions
in the Gromov--Witten invariants that we are able to compute with our method.
\footnote{
  By working with generic defining equations and restricting to the subspace $\mathcal H$, we
  are able to uniformly work with all CY3 complete intersections in weighted projective
  stacks.
  One may be able to compute even more invariants by working with special
  defining equations, and an appropriately chosen extended GIT, at least in some special cases, e.g.\ see Remark~\ref{rem: x24}
}

The admissible state space $\cH$ is defined as
\[ \cH = \bigoplus \cH^i_\alpha\]
where 
\[ \cH^i_\alpha \subset H^{i-2\iota_\alpha}(X_\alpha) \] 
is defined by the following specifications: 

\begin{itemize}
\item For $i$ odd, set $\cH^i_\alpha = 0$ 
\item For $i$ even and $\dim(X_\alpha) > 0$, set $ \cH^i_\alpha = H^{i - 2\iota_\alpha}(X_\alpha)$. Note that when $\alpha \neq 0$, Lemma \ref{lem:CY3-sector-dim} and our assumption on the generic isotropy groups implies that $\dim(X_\alpha) = 1$. In this case, we further have $\cH^i_\alpha = H^{i - 2\iota_\alpha}(\PP_\alpha)$.
\item
  For $i$ even and $\dim(X_\alpha) = 0$, we have that $X_\alpha$ is finitely
  many reduced points in $\PP_\alpha$. Given a (possibly empty) subset $\Lambda
  \subset \{ k \mid \alpha w_k \in \ZZ\} =: C_\alpha$, let $U_\Lambda
  \subset \PP_\alpha$ be the stacky torus
  \begin{equation} \label{eq:U-Lambda}
 U_\Lambda = \{ p \in \PP_\alpha \, | \, x_k(p) = 0 \text{ for all }k \in \Lambda, \, x_k(p) \neq 0 \text{ for all } k \in C_\alpha \setminus \Lambda \}.
\end{equation}
There is a standard stratification of $\PP_\alpha$ given by the disjoint union of all the $U_\Lambda$, hence $X_\alpha$ is also a disjoint union of $U_\Lambda \cap X_\alpha$. We call a substack $Y \subset X_\alpha$ \textit{special} if it can be written as 
 \begin{equation} \label{def:special}
Y = \overline{U_\Lambda} \cap X_\alpha 
 \end{equation}
 for some subset $\Lambda$.
 We then set $\cH^i_\alpha $ to be the span of the fundamental classes $\one_{Y}$, ranging over all special substacks $Y$. 
\end{itemize}

Let $H^*_{\amb}(X) \subset H^*_{\CR}(X)$ be the space of ambient cohomology classes, which
by definition is the image of the restriction $H^*_{\mathrm{CR}}(\mathbb P(\vec
w)) \to H^*_{\mathrm{CR}}(X)$.
Thus,
by an orbifold variant of the Lefschetz hyperplane theorem \cite[Section
2.5]{GM88}, it is easy to see that
$\mathcal H$ is
spanned by $H^*_{\amb}(X)$ and the classes $\mathbf 1_Y$ for special cycles $Y$.
And in particular,
\[
  H^*_{\amb}(X) \subseteq \cH \subseteq H^*_{\CR}(X).
\]
An example where both the above inclusions are strict can be found in Section \ref{sec:x24}.

Fix some $\alpha$ such that $\dim(X_\alpha) = 0$.
  It is easy to see that
  \begin{equation}
    \label{eq:basis-1st-time}
    \{\mathbf 1_{Y} \mid Y = \overline{U_{\Lambda}} \cap X_{\alpha}, {U_{\Lambda}} \cap
    X_{\alpha} \neq \emptyset\}
  \end{equation}
  is a canonical choice of basis for $\mathcal H \cap H^0(X_{\alpha})$.
  In the following, we first give a necessary condition (Lemma~\ref{lem:incidence}) for $U_{\Lambda} \cap
  X_{\alpha}$ to be nonempty, which will be crucial to constructing the extended
  GIT (cf.\ Lemma~\ref{lem:F-extension} and Lemma~\ref{lem:canonical-bundle}).
  We then give a practical
  algorithm for finding all such $\Lambda$ (Remark~\ref{rmk:finding-Y}).

By reindexing, we may assume without loss of generality that
\[
  \{ i \, | \, \alpha  w_i \in \ZZ\} = \{0, \dots, k\}, \quad \Lambda = \{r+1, \dots, k\}, \quad \text{and} \quad 
  \{ j \, | \, \alpha  b_j \in \ZZ\} = \{1, \dots, k\},
\]
for some non-negative integers $k$ and $r$.
Thus, $\PP_\alpha = \PP(w_0, \dots, w_k)$, from which $X_\alpha$ is cut out by polynomials
$F_1, \dots, F_k$, and $\overline{U_{\Lambda}} = \mathbb P(w_0, \ldots, w_r)$.

We further suppose that

\[
  \Gamma := 
  \{j \mid F_j \text{ is identically zero on } \overline{U_{\Lambda}}\} =
  \{s+1 ,\ldots, k\}.
\]
Thus $X_{\alpha} \cap \overline{U_{\Lambda}}$ is cut out from
$\overline{U_{\Lambda}}$ by the vanishing of $F_j$ for
$j\not \in \Gamma$, hence we always have $\#\Lambda \geq \#\Gamma$.

It will be convenient to set up an incidence stack to study the dependence of $U_\Lambda \cap X_\alpha$ on the choice of $F_1, \dotsc, F_s$ modulo $(x_{r+1},\dotsc, x_n)$.
For $j = 1 ,\ldots, s$, let $B_j$ be the projectivization of
the vector space
\[
  \{G \in \mathbb C[x_0 ,\ldots, x_r] \mid G \text{ is homogeneous of degree }b_j\}.
\]
Here each $x_i$ is viewed as a variable of degree $w_i$, as usual.
Set $B = B_1 \times \cdots \times B_s$.
Consider the incidence stack
\[
  Z = \{ (u, G_1, \dots, G_s) \in U_{\Lambda} \times B\, | \, G_i (u) = 0 \text{ for all $i$} \}.
\]
Let $\mathrm{pr}_U\colon Z \to U_{\Lambda}$  and $ \mathrm{pr}_{B}\colon Z \to
B$ be the two projections. Set $\Lambda^{c} := \{0 ,\ldots, k\} \setminus \Lambda
= \{0 ,\ldots, r\}$.

\begin{lemma} \label{lem:incidence}
  In the situation above, if $U_{\Lambda}\cap X_{\alpha} \neq \emptyset$, then
 $\# \Lambda = \#\Gamma$, the incidence stack $Z$ is
  irreducible, and $\mathrm{pr}_B$ is dominant and generically finite.
\end{lemma}
\begin{proof}
  We first claim that $\dim B_j > 0$ for each $j$.
  If not, we would have $F_j \equiv C x_0^{n_0} \cdots
  x_r^{n_r}\,\, \mathrm{mod}\,\,(x_{r+1} \dots, x_n)$ for some
  constant $C \neq 0$.
  But
  then $F_j$ is non-vanishing on $U_\Lambda$, contradicting our 
  assumption.

  Now it is clear that for generic $u \in
  U_\Lambda$, 
  $\mathrm{pr}_U^{-1}(u)$ is irreducible of dimension $\dim(B)
  - s$, since $G_j(u) = 0$ defines a hyperplane in $B_j$, for each $j$.
  By homogeneity of $U_\Lambda$ with
  respect to the standard scaling action by $(\mathbb C^*)^{r+1}$,
  this holds for all $u \in U_\Lambda$. Since $\mathrm{pr}_U$ is proper and
  representable, it follows
  that $Z$ is irreducible and
  \begin{equation}
    \label{eq:incidence-dimension}
    \dim(Z) = \dim(B) + r - s.
  \end{equation}
  Again, since we have assumed that $U_{\Lambda} \cap X_{\alpha}$ is nonempty
  for general defining equations, $\mathrm{pr}_{B}$ must be dominant. Recall
  that we already have $r \leq s$. Thus by \eqref{eq:incidence-dimension}, we conclude that $r = s$ and $\mathrm{pr}_B$ is generically finite.
\end{proof}

We summarize our discussion in the following corollary and remark.
\begin{corollary}
  \label{cor:special-numerics2}
  Assume $\dim X_{\alpha} = 0$ and $\Lambda \subset \{i\mid \alpha w_i \in
  \mathbb Z\}$. Then, $U_{\Lambda}\cap X_{\alpha} \neq \emptyset$ implies
  \begin{equation}
    \label{eq:nonempty-interior}
    \#\{j \mid F_j \text{ not identically zero on }{U_{\Lambda}}\} =
    \dim U_{\Lambda}.
  \end{equation}
  Conversely, if the identity \eqref{eq:nonempty-interior} holds, then 
  \[
    \overline{U_\Lambda} \cap X_\alpha \neq \emptyset.
  \]
\end{corollary}
\begin{proof}
    The first part is already in Lemma~\ref{lem:incidence}. The second part is
    clear since the intersection of $\dim U_{\Lambda}$ many hypersurfaces in the
    weighted projective stack $\overline{U_{\Lambda}}$ must be nonempty.
\end{proof}

\begin{remark}
  \label{rmk:finding-Y}
There is
  a practical way for finding all $\Lambda$ for which $U_{\Lambda} \cap
  X_{\alpha} \neq \emptyset$.
First, given $\Lambda = \{r+1 ,\ldots, k\}$, it is easy to compute $\Gamma$. Indeed,
    since we have assumed that the defining equations are general,
    \[
      j \in \Gamma \iff b_j = c_0w_0 + \cdots + c_rw_r \text{ for no }c_0
      ,\ldots, c_r \in \mathbb Z_{\geq 0}.
    \]

    Second, it suffices to only consider those $\Lambda$ such 
    that
    \[
      \{w_i \mid i \in \Lambda\} \cap \{w_i \mid i \not\in \Lambda\} = \emptyset.
    \]
    Indeed, suppose $\Lambda = \Gamma = \{r+1 ,\ldots, k\}$  and $w_r =
    w_{r+1}$, then letting $\Lambda^\prime = \{r+2 ,\ldots, k\}$, 
    $F_j$ vanishes on $\overline{U_{\Lambda^\prime}}$ for any $j\in \Gamma$. Then
    $X_{\alpha} \cap \overline{U_{\Lambda^\prime}}$ cannot be zero-dimensional.

    Third, supposing we have found $\#\Lambda =
    \#\Gamma$, we can do a point-count
    to tell whether $U_\Lambda \cap X_\alpha = \emptyset$. 
  
More precisely, we have
  \[
    \deg([   \overline{U_\Lambda} \cap X_\alpha ]) = \frac{b_1 \cdots
      b_r}{w_0 \cdots w_r}.
  \]
  Applying the same computation to all $\Lambda^\prime \supset \Lambda$, using
  the inclusion-exclusion principle, we can compute the number of points in
  $U_\Lambda \cap X_\alpha$. This in particular determines whether $U_\Lambda \cap
  X_\alpha$ is empty.
\end{remark}

Lastly, we remark that the admissible state space $\cH$ is closed under the
Chen--Ruan product, so it can be viewed as an enlargement of the ambient cohomology
$H^*_{\amb}(X)$ in $H^*_{\CR}(X)$.
The left hand side of \eqref{mirroreq} below is a product in $\mathcal
H$. However, the remaining results of this paper are logically independent of
Proposition~\ref{prop:closed-subring}. Nevertheless we include
Proposition~\ref{prop:closed-subring} since it might be interesting in its own
right. Replacing degree $0$ invariants by degree $d$ invariants, the same
proof shows that $\mathcal H$ is closed under quantum product $*_\tau$ for
$\tau$ in $\mathcal H$.

\begin{proposition} \label{prop:closed-subring}
The admissible state space $\cH \subset H^*_{\CR}(X)$ is a subring. And the
restriction of the pairing defined by \eqref{def:pairing} is non-degenerate.
\end{proposition}
\begin{proof}
  We first define a projection operator $P\colon H^{2*}_{\CR}(X) \to H^{2*}_{\CR}(X)$
  whose image is $\mathcal H$.
  For $\dim(X_\alpha) > 0$, we define $P$ to be the identity on $H^{2*}(X_\alpha)$.
  When $\dim(X_\alpha) = 0$, recall that $\mathbb P_{\alpha}$ has a stratification with
  $U_\Lambda$ as open strata, and hence every point $u \in X_\alpha$ belongs to a unique $U_\Lambda$.
  Then for the fundamental class $\one_u$ supported on $u$, we define
    \[
      P(\one_u) = \frac{1}{|U_\Lambda \cap X_\alpha|}
      \mathbf 1_{U_{\Lambda} \cap X_{\alpha}}.
    \]
where $\Lambda$ is such that $u \in U_\Lambda$, and $|U_\Lambda \cap X_\alpha|$
is the number of points in the coarse moduli of $U_\Lambda \cap X_\alpha$.

Since \eqref{eq:basis-1st-time} is a basis for $\mathcal H \cap H^0(X_{\alpha})$,
the image of the operator $P$ is precisely the subspace $\cH$. Moreover, it is
straightforward to check that $P^2 = P$ and that $P$ is self-adjoint under the
pairing defined by \eqref{def:pairing}. Thus, we have that the pairing is
non-degenerate on $\cH$, and we obtain an orthogonal decomposition
\[
  H^{2*}_{\CR}(X) \cong \cH \oplus \cH^\perp
\]
where $\cH^\perp = \ker P$.

Now we show that $\cH$ is closed under the Chen--Ruan product. Let $\{\eta_0,
\dots, \eta_s\}$ be a basis for $H^{2*}_{\CR}(X)$ such that $\eta_i \in \cH$ for
$0 \leq i \leq \ell$ and $\eta_i \in \cH^\perp$ for $\ell+1 \leq i \leq s$. For
$\gamma_1, \, \gamma_2 \in \cH$, we have from the definition of the Chen--Ruan
product \eqref{eq:CR-product} that
\begin{equation} \label{eq:admissible-product}
  \gamma_1 * \gamma_2 = \sum_{i=0}^s \langle \gamma_1, \gamma_2, \eta^i \rangle_{0,3,0} \eta_i 
\end{equation}
where $\{\eta^i\}$ is the basis 
dual to $\{\eta_i\}$ under the  pairing \eqref{def:pairing}. To complete the proof, we make use of the following lemma.

\begin{lemma} \label{lem:monodromy}
Assume $\dim(X_\alpha) = 0$ and $X_\alpha \cap U_\Lambda \neq \emptyset$ for
some $\Lambda$. Then as we vary the defining equations for $X$, the monodromy
action acts transitively, i.e., given any two points $u_1, \, u_2 \in X_\alpha
\cap U_\Lambda$, there exists a path in the space of tuples of generic equations
along which the monodromy takes $u_1$ to $u_2$.
\end{lemma}
 \begin{proof} 
  Recall the incidence stack ${Z} \subset {U_\Lambda} \times B$ in Lemma \ref{lem:incidence},
  which is irreducible of the
  same dimension as $B$. Let $\underline{Z}$ be the coarse moduli of $Z$. The
  projection $\underline{Z} \to B$ is generically \'etale since
  the fiber of a generic point is the coarse moduli of $X_\alpha \cap {U_\Lambda}$. After
  shrinking $B$ we obtain a connected covering space.  Hence the monodromy action
  is transitive on the fibers.
\end{proof}

By this lemma, the operator $P$ is the averaging operator for the
  monodromy action on $H^{2*}_{\CR}(X)$. Since Gromov--Witten invariants are
  monodromy invariant, we have that
\[
  \langle  P(\gamma_1), P(\gamma_2), \gamma_3 \rangle_{0,3,0}
  = \langle  P(P(\gamma_1)), P(P(\gamma_2)), P(\gamma_3) \rangle_{0,3,0}
  = \langle  P(\gamma_1), P(\gamma_2), P(\gamma_3) \rangle_{0,3,0}
\]
for any three classes $\gamma_1$, $\gamma_2$, and $\gamma_3$.
When $\gamma_1, \gamma_2 \in \mathcal H$, applying this to the invariants in \eqref{eq:admissible-product} gives us
\[
  \langle \gamma_1, \gamma_2, \eta^i \rangle_{0,3,0} = \langle \gamma_1,
  \gamma_2, P(\eta^i )\rangle_{0,3,0} = 0 \, \, \text{ for $i \geq \ell+1$}.
\]
Hence $\gamma_1 * \gamma_2 \in \cH$. 
\end{proof}

\subsection{Extension Data} \label{sec:extension-data}
Let $\mathcal H^{2}_{\mathrm{tw}} \subset \mathcal H$ be the set of admissible classes
of Chen--Ruan degree $2$ supported on the twisted sectors. By definition they
are spanned by
\begin{enumerate}[(i)]
\item
  the fundamental classes $\mathbf 1_{X_{\alpha}}$ for $X_{\alpha}$ of dimension $1$ and age $1$
\item
  the fundamental classes 
  $\mathbf 1_Y$ for special
  cycles $Y\subset X_{\alpha}$, where $X_{\alpha}$ has dimension $0$ and age
  $1$.
\end{enumerate}
Since the target $X$ is CY3, using the axioms of Gromov--Witten invariants, any
invariants with insertions from $\mathcal H$ can be computed from
invariants with insertions from $\mathcal H^{2}_{\mathrm{tw}}$, and thus we
restrict the usual $J$-function to $\mathcal H^{2}_{\mathrm{tw}}$.

The new GIT presentation for our target, which we call the \emph{extended GIT}, will
come from a canonical choice of basis 
\begin{equation}\label{eq:t-label}
  \{ \phi_1, \dots, \phi_m \}
\end{equation}
for $\mathcal H^{2}_{\mathrm{tw}}$. It consists of
$\mathbf 1_{X_{\alpha}}$ for each $X_{\alpha}$ in (i) and the list
\eqref{eq:basis-1st-time} for each $X_{\alpha}$ in (ii).

Having chosen \eqref{eq:t-label}, we form the \textit{extended GIT
  presentation}
\begin{equation}
  \label{eq:extended-Pw}
  \left[\mathbb C^{n+1} \times \mathbb C^{m} \sslash_{\theta_e} \mathbb C^* \times (\mathbb C^*)^{m}\right]
\end{equation}
of $\mathbb P(\vec w)$ by adding an additional $\mathbb C$- and $\mathbb C^*$-factor for each
$\phi_i$.
Here, the torus action is given by the weight matrix
\begin{equation}
  A = \label{eq:wmatrix}
  \left(
    \begin{array}{ccc|ccc}
      w_0 &  \cdots &  w_{n} & & 0_{1 \times m} & \\
      \hline
      a_{10}& \cdots & a_{1n}& & \\
      \vdots & \ddots & \vdots & & \mathrm{Id}_{m\times m}\\
      a_{m0} & \cdots & a_{mn} && \\
    \end{array}
  \right),
\end{equation}
where the 
\[
  (a_{i0}, \ldots, a_{in})
\]
part of the row is determined by $\phi_i$ via the following rule: 
\begin{itemize}
\item
  if $\phi_i = \mathbf 1_{X_{\alpha}}$ is from (i), then we set $a_{ij} =
  \floor{\alpha w_j}$
  for $j = 0 ,\ldots, n$,
\item
  if $\phi_i =  \mathbf 1_{\overline{U_{\Lambda}} \cap
    X_{\alpha}}$ is from (ii), then 
  \begin{equation} \label{eq:weight-entry}
    a_{ij} =
    \floor{\alpha w_j} - \delta_{\Lambda}(j) \quad  \text{for $j  = 0 ,\ldots, n$}.
  \end{equation}
  Here, $\delta_{\Lambda}$ is the indicator function 
  \[
    \delta_{\Lambda}(j) =
    \begin{cases}
      1 & \text{for $j \in \Lambda$}, \\
      0 & \text{for $j \not\in \Lambda$}.
    \end{cases}
  \]
\end{itemize}
\begin{remark} 
  The choice of weights \eqref{eq:weight-entry} may be understood both
  combinatorially and geometrically. Combinatorially, it ensures that the change
  of coordinates in the mirror transformation is of the right form (see Lemma
  \ref{lem:I-derivative}). This is crucial for the invertibility of the mirror
  map. Geometrically, there is a way to trade markings mapping into twisted
  sectors of the original GIT stack quotient for base points in the extended GIT.
  An example of this reasoning is provided in Section~\ref{sec:motivation}.
\end{remark}

Finally, the character $\theta_e$ is chosen to be
\begin{equation} \label{eq:extended-char}
  \theta_e: \CC^* \times (\CC^*)^m \to \CC^*, \quad (\lambda_0, \dots,
  \lambda_m) \mapsto \prod_{i=0}^m \lambda_i,
\end{equation}

This finishes the description of the extended GIT for the weighted projective
stack, which is justified in the following lemma and corollary.
\begin{lemma} \label{lem:git-locus}
  The $\theta_e$-semistable locus $(\CC^{m+n+1})^{\ss}_{\theta_e}$ of the GIT data in \eqref{eq:extended-Pw} and \eqref{eq:extended-char} is given by 
  \[ (\CC^{m+n+1})^{\ss}_{\theta_e} = (\mathbb C^{n+1}\setminus \{0\})
    \times (\mathbb C\setminus \{0\})^{m}.\]
  Moreover, it agrees with the $\theta_e$-stable locus.
\end{lemma}
\begin{proof}
  Let $x_0 ,\ldots, x_n, y_1 ,\ldots, y_m$ be the standard coordinates on
  $\mathbb C^{n+1} \times \mathbb C^{m}$.
  Set
  \[
    Z = V(x_0 ,\ldots, x_n) \cup V(y_1) \cup \cdots \cup V(y_m).
  \]
  The goal is to show that $Z$ is equal to the unstable locus.
  For any integer $k>0$, the space
  of $(k\theta_e)$-invariants is spanned by the monomials
  \begin{equation}
    \label{eq:k-theta-invariant}
    x_0^{c_0} \cdots x_n^{c_n} y_1^{d_1} \cdots y_m^{d_m}
  \end{equation}
  with $\sum_{i=0}^{n}c_i w_i = \sum_{i=0}^{n}c_ia_{ji} + d_j = k$, for $j=1,\ldots, m$.
  The unstable locus is the common vanishing locus of these monomials.

  Let $p\in Z$. For any $(k\theta_e)$-invariant
  \eqref{eq:k-theta-invariant}, note that  $c_0 ,\ldots, c_n$ cannot be all zero.
  Then, since $\alpha_j < 1$, we have that $a_{ji} < w_i$ for $a_{ji}$ as in \eqref{eq:weight-entry}, so
  $d_j>0$ for each $j$. Hence we see that the monomial must vanish at $p$.

  On the other hand, for a point $p\not\in Z$, pick $k > 0$, such that $k/w_{i}$ is an integer for each $i = 0 ,\ldots, n$.
  Using $a_{ji} < w_{i}$ for each $j$ as before, we have
  that
  $x_{i}^{k/w_{i}}y_1^{k(1-a_{1i}/w_{i})}\cdots y_m^{k(1-a_{mi} /
    w_{i})}$ is a $(k\theta_e)$-invariant that only involves $x_{i}$ and
  $y_1 ,\ldots, y_m$. Since $x_i$ does not vanish at $p$ for some $i$, we have a monomial that does not vanish at $p$. 
  
  It is straightforward to check that the semistable and stable loci agree.
\end{proof}
\begin{corollary}
  The stacky GIT quotient \eqref{eq:extended-Pw} agrees with $\mathbb P(\vec w)$, or in other words
  \begin{equation*}
    \mathbb P(\vec w) = \left[\mathbb C^{n+1} \times \mathbb C^{m}
      \sslash_{\theta_e} \mathbb C^* \times (\mathbb C^*)^{m}\right].
  \end{equation*}
\end{corollary}
\begin{proof}
  We abbreviate
  \[
    G = \mathbb C^* \times (\mathbb C^*)^m, \quad T = (\mathbb C^*)^m, \quad V =
    (\mathbb C^{n+1}\setminus \{0\}) \times (\mathbb C\setminus \{0\})^{m}, 
    \quad W =  \mathbb C^{n+1}\setminus \{0\}.
  \]
  Consider the inclusion maps
  \begin{equation*}
    W \to V,
    \quad (x_0 ,\ldots, x_n) \mapsto (x_0 ,\ldots, x_n, 1 ,\ldots, 1).
  \end{equation*}
  \begin{equation*}
    \mathbb C^* \to \mathbb C^*\times T, \quad t \mapsto (t, 1 ,\ldots, 1).
  \end{equation*}
  These maps are compatible with the group actions on $V$ and $W$, and thus
  induce
  \begin{equation}
    \label{eq:isom-stack}
    \mathbb P(\vec w) = [W/\mathbb C^*] \to [V/(\mathbb C^*\times T)] = \left[\mathbb C^{n+1} \times \mathbb C^{m}
      \sslash_{\theta_e} \mathbb C^* \times (\mathbb C^*)^{m}\right].
  \end{equation}
  To show it is an isomorphism, consider the fibered diagrams
  \[
    \begin{tikzcd}
      W\times T \ar[r] \ar[d]& {[(V\times \mathbb C^* \times T) / \mathbb C^*] =
        V \times T} \ar[r] \ar[d]& V\ar[d] \ar[r] & \operatorname{Spec} \mathbb C \ar[d] \\
      {[W/\mathbb C^*]} \ar[r]& {[V/\mathbb C^*]} \ar[r]& {[V/(\mathbb C^*\times T)]} \ar[r]& B(\mathbb C^*\times T),
    \end{tikzcd}
  \]
  where the map $V\times T \to V$ is the action of $T$ on $V$. It is easy
  to see that $W\times T \to V$ is an isomorphism. Thus \eqref{eq:isom-stack}
  is an isomorphism after passing to a smooth cover. Hence
  \eqref{eq:isom-stack} is an isomorphism.
\end{proof}

We now turn our attention to the GIT presentation of the complete intersection
$X \subset \PP(\vec w)$. To fit the above GIT model, we need to extend the
affine cone $W$ to a $(\CC^*)^{m+1}$-invariant closed subscheme $W_e \subset
\CC^{n+m+1}$ such that $W_{e} \cap (\CC^{n+1} \times \{1, \dots, 1\}) = W$.

Let $x_0, \dots, x_n, y_1, \dots, y_m$ denote the standard coordinates on
$\CC^{n+m+1}$ and recall the polynomials $F_1, \dots, F_{n-3}$ that define $W$.
The columns of $A$ in \eqref{eq:wmatrix} define a $\ZZ^{m+1}$-valued multi-degree
of the variables $x_0, \dots, x_n, y_1, \dots, y_m$. We will replace each $F_j$
by some $\tilde{F}_{j} \in \CC[x_0,\dots, x_n, y_1, \dots, y_m]$ of multi-degree
\begin{equation}
  \label{eq:multi-degrees}
  (b_{0j} ,\ldots, b_{mj}),
\end{equation}
where $b_{0j} = b_j$, and for $i = 1 ,\ldots, m$, recalling the definition of
$\{\phi_1 ,\ldots, \phi_m\}$ in \eqref{eq:t-label},
\begin{itemize}
\item
  if $\phi_i = \mathbf 1_{X_{\alpha}}$ is from (i), then we set $b_{ij} = \floor{\alpha b_j}$
  for $j = 1 ,\ldots, n-3$,
\item
  if $\phi_i = \mathbf 1_{\overline{U_{\Lambda}} \cap X_{\alpha}}$ is from (ii), then 
  we set
  \begin{equation}
    \label{eq:Gamma}
    \Gamma = \{j \mid \alpha  b_j \in \mathbb Z, F_{j}
    \text{ is identically zero on }U_{\Lambda}\}
  \end{equation}
  and 
  \begin{equation} \label{eq:degree-entry}
    b_{ij} =
    \floor{\alpha  b_j} - \delta_{\Gamma}(j) \quad  \text{for $j  = 1 ,\ldots, n-3$},
  \end{equation}
  where $\delta_{\Gamma}$ is the indicator function defined analogously to $\delta_{\Lambda}$ in \eqref{eq:weight-entry}.
\end{itemize}

\begin{lemma} \label{lem:F-extension} 
  For $k = 1, \dots, n-3$, there exists a unique $\tilde F_k \in \CC[x_0, \dots, x_n, y_1, \dots, y_m]$ such that 
  \begin{enumerate}
  \item $\tilde F_k$ is quasi-homogeneous of multi-degree 
    \[
      (b_{0k} ,\ldots, b_{mk}),
    \]
\item $\tilde F_k(x_0, \dots, x_n, 1, \dots, 1 ) = F_k(x_0 ,\ldots, x_n)$.
\end{enumerate}
\end{lemma}

\begin{proof}
  Because of the quasi-homogeneity assumption, if
  \begin{equation*}
    F_k = \sum_{c_0, \dotsc, c_n} r_{c_0, \dots, c_n} x_0^{c_0} \cdots x_n^{c_n}
  \end{equation*}
  for some coefficients $r_{c_0, \dots, c_n} \in \CC$, we must have
  \begin{equation*}
    \widetilde F_k = \sum_{c_0, \dotsc, c_n} r_{c_0, \dots, c_n} x_0^{c_0} \cdots x_n^{c_n} \cdot y_1^{d_1} \cdots y_m^{d_m},
  \end{equation*}
  where
  \[ 
    d_i = b_{ik} - \sum_{j=0}^n c_ja_{ij}, \quad i = 1 ,\ldots, m.
  \]
  Thus, it suffices to show that each $d_i$ is non-negative.
  Note that $b_k = \sum_{j=0}^n c_jw_j$. 
  Hence if $\phi_i$ is from $H^*(X_{\alpha})$, then 
  \[
    \floor{\alpha b_k}  = \left \lfloor \sum_{j=0}^n c_j \alpha w_j  \right\rfloor 
    \geq \sum_{j=0}^n c_j\floor{\alpha w_j} \geq \sum_{j=0}^n c_j a_{ij}.
  \]
  This proves the desired result when $\phi_i = \mathbf  1_{X_{\alpha}}$ is from (i) or
  $k \not \in \Gamma$ when $\phi_i = \mathbf 1_{\overline{U_{\Lambda}} \cap X_{\alpha}}$
  is from (ii), where $\Gamma$ is introduced above, since in those cases we have
  $b_{ik} = \lfloor \alpha b_k \rfloor$.

  Now suppose
  $\phi_i = \mathbf  1_{\overline{U_{\Lambda}} \cap X_{\alpha}}$
  is from (ii) and $k\in \Gamma$. Then $\floor{\alpha b_k} = \alpha b_k$ and
  $b_{ik} = \alpha b_k  - 1$.
  It suffices to show
  \[
    \sum_{j=0}^n c_j(\floor{\alpha  w_j} - \delta_{\Lambda}(j)) <
    \alpha   b_k  =   \sum_{j=0}^n c_j\alpha  w_j.
  \]
  Thus we will win as long as for some $c_j \neq 0$, either $\alpha w_j\not\in \mathbb Z$ or
  $j\in \Lambda$. Indeed, by the definition of $\Gamma$,
  $F_{k}$ is identically zero on $U_{\Lambda}$. This means each monomial in
  $F_{k}$ must involve some variable $x_j$ such that $\alpha w_j\not\in \mathbb Z$ or
  $j\in \Lambda$. This completes the proof.
\end{proof}

Letting $\tilde F_1, \dots, \tilde F_{n-3}$ be the homogenizations given by
Lemma~\ref{lem:F-extension}, we define the \emph{extended affine cone} as the vanishing locus
\[W_e = V(\tilde F_1, \dots, \tilde F_{n-3}) \subset \CC^{n+m+1}. \]
Consequently, we obtain an extended GIT presentation for $X$ given by
\begin{equation} \label{eq:X-ex-git}
X = \left[W_e \sslash_{\theta_e} (\CC^*)^{m+1}\right].
\end{equation}
where the character and weight matrix are as in \eqref{eq:extended-char} and
\eqref{eq:wmatrix}. Lemma~\ref{lem:git-locus} and the second condition of
Lemma~\ref{lem:F-extension} ensure that this is indeed a GIT presentation for
$X$ that is compatible with quasimap theory.

\medskip
The following diagram introduces the quotient stacks $\mathfrak{X}$ and $\mathfrak{Y}$, and summarizes our construction:
\begin{equation}
  \label{cd:main}
  \begin{tikzcd}
    X= \left[W_e \sslash_{\theta_e} (\CC^*)^{m+1} \right] \arrow[hookrightarrow]{r} \ar[d] & \left[W_e / (\CC^*)^{m+1}\right] \ar[d] =: \mathfrak{X} \\
    \PP(\vec{w})= \left[\CC^{n+m+1} \sslash_{\theta_e} (\CC^*)^{m+1}\right] \arrow[hookrightarrow]{r}  & \left[\CC^{n+m+1} / (\CC^*)^{m+1}\right] =: \mathfrak{Y}
  \end{tikzcd}
\end{equation}

We set
\begin{equation}
  \label{eq:vb-notation}
  \cE = \bigoplus_{j=1}^{n-3}\cO_{\mathfrak Y}(b_{0j}, \dots, b_{mj}),
\end{equation}
where
$\cO_{\mathfrak Y}(b_{0j}, \dots, b_{mj})$ is the line bundle on
$\mathfrak Y$
defined by the mixing
construction via the character
\[
  (\lambda_{0}, \dots, \lambda_m) \mapsto \prod_{i=0}^m
  \lambda_i^{b_{ij}},
\]
as in \eqref{eq:line-bundle}.
Thus $(\tilde F_1, \dots, \tilde F_{n-3})$ is a section of $\mathcal E$ and its
zero locus is $\mathfrak X$.

\begin{remark}
  At least a priori, it can happen that the extended affine scheme $W_e$ is not
  l.c.i., even if $W$ is.
  However, at this point we have only found the following
  (non-CY3) example as an illustration of the possible phenomenon.

  Consider a complete intersection $X_{4,4,4} \subset \PP(1,1,1,1,1,3)$ defined by
  three smooth degree $4$ hypersurfaces. Each of the degree $4$ hypersurfaces must
  pass through the $B\mu_3$ point, hence this point must always be contained in
  $X_{4,4,4}$. If we extend by the class $\one_{1/3}$ corresponding to the
  fundamental class of the twisted sector $X_{1/3} \cong B\mu_3$, the degree $4$
  equations will extend to be of the form
  \[
    f_4(x_0, \dots, x_4) y + f_1(x_0, \dots, x_4)x_5
  \]
  where $f_i$ is a polynomial of degree $i$ and $y$ is the extra coordinate gained
  from the extension. However, we then see that the codimension two locus $V(x_5, y)$
  is contained in $W_e$, hence it cannot be l.c.i.
\end{remark}

When $W_e$ is not l.c.i., we need the more general obstruction
theory of Lemma~\ref{lem:qpot}. We end this section with the following lemma to
justify our choice of the extended GIT presentation.
\begin{lemma} \label{lem:canonical-bundle}
  The extended GIT is a Calabi--Yau target space for quasimap
  theory, in the sense that
  \[
    \det(\mathbb T_{\mathfrak Y}) \otimes \det (\mathcal E )^{\vee} \cong \mathcal
    O_{\mathfrak Y}.
  \]
\end{lemma}
\begin{proof}
  This amounts to saying that the sum of the columns of the weight matrix $A$
  \eqref{eq:wmatrix} is
  equal to the sum of the multi-degrees of $\tilde{F_j}$ for $j =  1,\ldots,
  n-3$. Since $X$ is Calabi--Yau, we have
  \[
    \sum_{i=0}^{n}w_i = \sum_{j = 1}^{n-3} b_j.
  \]
  This verifies the first row.
  Now consider the row corresponding to $\phi_i$, which comes from
  some $H^0(X_{\alpha})$. When $\dim(X_{\alpha}) = 1$, we need to show
  \[
    \sum_{i=0}^n \floor{\alpha w_i} + 1 =  \sum_{j=1}^{n-3} \floor{ \alpha b_j}.
  \]
  This is precisely the condition that $X_{\alpha}$ has age $1$.
  Finally by Lemma~\ref{lem:incidence}, we have $\# \Lambda = \#\Gamma$. This
  completes the proof.
\end{proof}

\section{Extended $I$-function}
\label{sec-I-computation}

In this section, we compute an $I$-function from our extended GIT,
which we call the \textit{extended $I$-function}. We will start by
better understanding the fixed loci $F_\beta$ in Section
\ref{sec:I-geometry}, and then proceed to compute the relevant pieces
of the $I$-function \eqref{eq:I} from the $\CC^*$-equivariant perfect
obstruction theory in Section \ref{sec:I-POT}.

\subsection{Curve Classes and some Notation}
\label{sec:curve-classes}

For the convenience of explicit computation, from now on we slightly change our notion of curve classes.

Recall our construction of the extended GIT presentation in \eqref{cd:main}.
Originally we defined curve classes as group homomorphisms
$\mathrm{Pic}(\mathfrak X) \to \mathbb Q$. 
However, it is more convenient to work with $\Pic(\mathfrak Y)$ than
$\Pic(\mathfrak X)$, since the former is simply the group of characters of
$(\mathbb C^*)^{m+1}$.
The restriction map $\Pic(\mathfrak{Y}) \to \Pic(\mathfrak{X})$ induces a homomorphism
\[
  \Hom(\Pic(\mathfrak{X}), \QQ) \longrightarrow \Hom(\Pic(\mathfrak{Y}), \QQ)
  \cong \QQ^{m+1}.
\]
From now on, for a quasimap we will only remember the image of its curve class
and thus view curve classes as elements in $\Hom(\Pic(\mathfrak{Y}), \QQ)$.

More concretely, for a quasimap $C \to \mathfrak X$, the underlying $(\mathbb
C^*)^{m+1}$-principal bundle is equivalent to a tuple of line bundles

\begin{equation}
  \label{eq:L0-to-Lm}
  \mathcal L_0 ,\ldots, \mathcal L_m
\end{equation}
on $C$, where $\mathcal L_i$ is the pullback of $\mathcal O_{\mathfrak Y}(0
,\ldots, 1 ,\ldots, 0)$ with one ``$1$'' in the $i$-th place. Thus we say that the quasimap is of curve class
\begin{equation}
  \label{eq:curve-class-tuple-e}
  \beta = (e_0 ,\ldots,  e_m),
\end{equation}
if $\deg(\mathcal L_i) = e_i$ for $i = 0 ,\ldots, m$.
Furthermore, $Q^{\epsilon}_{g,n}(X, \beta)$ will denote the moduli of all such stable
quasimaps, etc.

We now introduce additional vector notation along these lines.
\begin{equation} \label{eq:column}
  \chi_{i} =  A_{\bullet i} \in  \mathbb Z^{m+1}, \quad i = 0 ,\ldots, n  + m
\end{equation}
will denote the $i$-th column of the weight matrix $A$ in \eqref{eq:wmatrix}, and
\begin{equation}
  \label{eq:column-multi}
  \xi_j = (b_{0j} ,\ldots, b_{mj}), \quad j = 1 ,\ldots,  n-3
\end{equation}
will denote the multi-degree of $\tilde F_{j}$, as introduced in \eqref{eq:multi-degrees}.

For any vector $\vec v = (v_0, \dots, v_m) \in \ZZ^{m+1}$, we write 
\[
  \beta\cdot \vec v = \sum_{i=0}^{m} e_i v_i,
\]
for the usual dot product. Thus, when viewing $\beta$ as a homomorphism
$\Pic(\mathfrak{Y}) \to \QQ$, we have
\[
  \beta\cdot \vec v = \beta(\cO_{\fY}(v_0, \dots, v_m)),
\]
where the
line bundle $\cO_{\fY}(v_0, \dots, v_m)$ is introduced in \eqref{eq:vb-notation}.

Note that later in Section~\ref{sec:invertibility} we introduce another way of
identifying curve classes as tuples of numbers, in order to make the $I$-function
a power series, as opposed to a Puiseux series.

\subsection{Geometry of $F_\beta$} \label{sec:I-geometry}

Recall the fixed locus $F_\beta \subset Q_{\PP(1, \bullet)}(X, \beta)$, whose
points parameterize quasimaps $f\colon \PP(1, r) \to \mathfrak X$ that are invariant
under the scaling action induced by \eqref{eq:action}, and for which the entire
curve class $\beta$ is supported over the single base point located at the special point $0$, as defined in
\eqref{eq:special-points}.

  Let $ \beta = (e_0 ,\ldots, e_m)$ be a curve class as in
  \eqref{eq:curve-class-tuple-e} such that 
  $F_{\beta}$ is nonempty.
  Then we must have $e_1 ,\ldots, e_m \in \mathbb Z$ and 
  $r$ must be the smallest positive integer such that $re_0 \in \mathbb Z$.
  Indeed, since the marking cannot be a base point, the GIT stability condition
  Lemma~\ref{lem:git-locus} implies that each of $\mathcal  L_1 ,\ldots, \mathcal L_m$
  in \eqref{eq:L0-to-Lm}
  admits a section that is nonzero at the unique marking, given by the last $m$
  coordinate functions on $\mathbb C^{n+1} \times \mathbb C^m$.
  Then $r$ is determined 
  from the representability requirement in Definition~\ref{def:quasimap}.

  We have the following diagram

\begin{equation}
  \label{eq:universal-quasimap-F-beta}
  \begin{tikzcd}
    \mathbb P(1, r) \times F_{\beta} \ar[r,"f"] \ar[d, "\pi"] & \mathfrak{X}\\
    F_{\beta} &
  \end{tikzcd}
\end{equation}
where $f$ is the universal quasimap.
Note that $f$ has a $\CC^*$-equivariant structure, if we equip $\fX$ with the trivial $\CC^*$-action.
Furthermore, by treating any vector bundle on $\fX$
as having trivial $\CC^*$-equivariant structure, we obtain an induced
$\CC^*$-equivariant structure on its pullback by $f$.

  Let $\Sigma_{\star} \subset \mathbb P(1, r) \times F_{\beta}$ be the unique
  marking, denoted by $\star$, which is at $\infty$ of $\mathbb P(1, r)$. Thus
  $\Sigma_{\star} \to F_{\beta}$ is a trivial gerbe banded by $\mu_r$.

\begin{lemma} 
  \label{lem:equivariant-structure-line-bundle}
  For any line bundle $\mathcal L$ on $\mathfrak X$, up to the $\ell$-th power
  map $\mathbb C^*\overset{\ell}{\to} \mathbb C^*$ for some $\ell > 0$,
  the equivariant structure on $f^*\mathcal L$ is the unique
  one such that
  the action on $f^*\mathcal L|_{\Sigma_{\star}\times F_{\beta}}$ is trivial,
  where $\Sigma_\star \subset \PP(1, r)$ is the gerbe over $\star$.
\end{lemma}
\begin{proof}
  After composition with the $\ell$-th power map for some $\ell > 0$, we can assume
  that the $\CC^*$-action on $\Sigma_\star$ is trivial. This means that the
  action map $\sigma: \CC^* \times \Sigma_\star \times F_\beta \to \Sigma_\star
  \times F_\beta$ coincides with the projection $\pi_{23}$ onto the last two
  factors (cf.\ \cite{Ro05}), hence a linearization of $f^*\cL|_{\Sigma_\star \times F_\beta}$ is an
  automorphism of $\pi_{23}^*f^*\cL$ over $\CC^* \times \Sigma_\star \times
  F_\beta$ satisfying the cocycle condition. Unravelling the cocycle condition, we
  see that a linearization is equivalent to a group homomorphism between groups
  over $\Sigma_\star \times F_\beta$
  \[
    \CC^* \times \Sigma_\star \times F_\beta \to
    \underline{\Aut}_{\Sigma_\star \times F_\beta}(f^*\cL|_{\Sigma_\star \times
      F_\beta}) = \CC^* \times \Sigma_\star \times F_\beta.
  \]
  However, we also know that the linearization is determined by data that
  makes $f|_{\Sigma_\star \times F_\beta}$ a $\CC^*$-equivariant morphism, namely the
  2-commutativity of the usual action diagram. Since the $\CC^*$-action is trivial
  on both the source and target, this is equivalent to the data of an automorphism
  of $f|_{\Sigma_\star \times F_\beta}$ satisfying the required compatibility
  conditions. But since $f|_{\Sigma_{\star}\times F_{\beta}}$ factors through the
  separated Deligne--Mumford stack $X$, the group
  $\mathrm{Aut}_{\Sigma_{\star}\times F_{\beta}} (f|_{\Sigma_{\star}\times
    F_{\beta}})$ is a finite group scheme over $\Sigma_{\star}\times F_{\beta}$.
  This means the group homomorphism defining the linearization must be of finite
  order, hence it is trivial and $f^*L|_{\Sigma_\star \times F_\beta}$ has
  trivial linearization.
  
  To show uniqueness, we note any difference in linearization comes from a
  difference in linearization of the trivial bundle. A linearization on
  $\cO_{\PP(1,r) \times F_\beta}$ is an automorphism of $\mathcal O_{\mathbb
    P(1,r)\times F_{\beta} \times \mathbb C^*}$ satisfying the cocycle condition.
  Such an automorphism must be given by multiplication by a non-vanishing
  function. However, a non-vanishing function that is equal to $1$ on
  $\Sigma_\star \times F_\beta \times \CC^*$ must be identically equal to $1$
  along $\mathbb P(1,r) \times F_{\beta} \times \mathbb C^*$. In other words, a
  linearization of $\cO_{\PP(1,r) \times F_\beta}$ that is trivial along
  $\Sigma_\star \times F_\beta$ must be the trivial linearization.
\end{proof}
With the equivariant structure of these line bundles determined, we can compute
the weights of the corresponding sections.
\begin{lemma} 
  \label{lem:computation-on-P1r}
  Given an equivariant line bundle $\cL$ of degree $\frac{d}{r}$ over
  $\PP(1,r)$, suppose that up to an $\ell$-th power map on $\mathbb C^*$, the action on $\mathcal
  L|_{\Sigma_\star}$ is trivial. Then we have
  \begin{itemize}
  \item $H^0(\PP(1,r), \cL)$ is the $\CC$-vector space with basis given by
    $x^iy^j$ where $i+ rj = d$, $i, j \in \ZZ_{\geq0}$,
  \item $H^1(\PP(1,r), \cL)$ is the $\CC$-vector space with basis given by
    $x^iy^j$ where $i+ rj = d$, $i, j \in \ZZ_{<0}$.
  \end{itemize}
  Moreover, the weight of $x^iy^j$ is equal to $\frac{i}{r}$.
\end{lemma}
\begin{proof}
  This first part follows directly from a \u{C}ech cohomology computation (see \cite[Section 5.2]{CCK15}).
  To see the weights,
  note that the weight of $y^r$ is equal to the weight of $\mathcal
  L|_{\Sigma_\star}$, which is zero by assumption. On the other hand, the weight of the
  meromorphic function $\frac{x^r}{y}$ is $1$, by the definition of the action
  on $\mathbb P(1,r)$. Hence the weight of $x$ is $\frac{1}{r}$.
\end{proof}

From now on we fix a section $\sigma_{\star}$ of $\mathbb P(1,r) \to
\operatorname{Spec} \mathbb C$ into the unique orbifold point, which is unique
up to non-canonical $2$-morphisms.
By abuse of notation, we will also write $\sigma_{\star}$ for
$\sigma_{\star}\times \mathrm{Id}_{S} : S \to \Sigma_{\star}\times S$ for any scheme $S$.
This amounts to fixing a trivialization for all such $\mu_r$-gerbes once and for all.

Recall that a section of $\Sigma_{\star}\times S \to S$ is equivalent to a
$\mu_r$-torsor on $S$.
Note that there may be many non-isomorphic $S$-morphisms $S \to \Sigma_{\star} \times S$,
and our $\sigma_{\star}$ corresponds to the trivial $\mu_r$-torsor on $S$.

\begin{lemma} 
  \label{lem:direct-image-positive-line-bundle}
  Let $\mathcal L$ be a line bundle on $\mathfrak X$ such that $f^*\mathcal L$ has
  fiberwise degree $\frac{d}{r}$.
Then we have
  \begin{enumerate}[(1)]
  \item
    if $d\geq 0$,
    \[
      R\pi_*f^*\mathcal  L \cong \bigoplus_{\substack {i+rj = d\\ i,j\in
          \mathbb Z_{\geq 0}}} \mathbb C x^iy^j \otimes
      \sigma_\star^*  f^* \mathcal L,
    \]
  \item
    if $d< 0$,
    \[
      R\pi_*f^*\mathcal  L \cong \bigoplus_{\substack {i+rj = d\\ i,j\in
          \mathbb Z_{< 0}}} \mathbb C x^iy^j \otimes
      \sigma_\star^* f^* \mathcal L[-1],
    \]
  \end{enumerate}
  where the weight of $x^iy^j$ is equal to $\frac{i}{r}$. Moreover, this
  commutes with base change along any $S\to F_{\beta}$.

\end{lemma}
\begin{proof}
  We base change to a scheme $S$, and will write $f^*\mathcal L|_{S}$ for the
  pullback of $f^*\mathcal L$ to $\mathbb P(1,r) \times S$, etc.
  
  Using the fact that $H^1(\mathbb P(1,r), \mathcal O_{\mathbb P(1,r)}) =0$, a
  standard argument (cf.\ Lecture 13 of \cite{mumford1966lectures}) using
  cohomology
  and base change shows that
  \[
    f^*\mathcal L|_{S} \cong \mathcal O_{\mathbb P(1,r)}(d) \boxtimes \mathcal M
  \]
  for $\mathcal M \in \mathrm{Pic}(S)$.
  By
  Lemma~\ref{lem:equivariant-structure-line-bundle},
  up to some $\ell$-th power map on $\mathbb C^*$, it is also
  an isomorphism of $\mathbb C^*$-equivariant line bundles
  where $\mathcal M\in \mathrm{Pic}(S)$ has the trivial $\mathbb
  C^*$-action, and $\mathcal O_{\mathbb P(1,r)}(d)$ is viewed as an equivariant
  bundle whose fiber at infinity has weight zero. The rest
  follows from Lemma~\ref{lem:computation-on-P1r}.
\end{proof}
\begin{corollary}
  \label{cor:invariant-sections-line-bundle}
  Let $\mathcal L$ be a line bundle on $\mathfrak X$ such that $f^*\mathcal L$ has
  fiberwise degree $\frac{d}{r}$. Then if $\frac{d}{r}\in \mathbb Z_{\geq 0}$,
  we have canonical isomorphisms
  \[
    (R\pi_*f^*\mathcal  L)^{\mathbb C^*} \cong \sigma_{\star}^*  f^* \mathcal L.
  \]
  Otherwise, $(R\pi_*f^*\mathcal  L)^{\mathbb C^*} = 0$. Moreover, this
  commutes with base change along any $S\to F_{\beta}$.
\end{corollary}
\begin{proof}
  This follows immediately from Lemma~\ref{lem:direct-image-positive-line-bundle}.
\end{proof}

    We are now ready to
    identify  $F_\beta$ with an explicitly defined substack of $X$, for $\beta =
    (e_0 ,\ldots, e_m)$ as in \eqref{eq:curve-class-tuple-e}.
    We first let
    \[
      W_\beta \subset W
    \]
    be the subscheme defined by the ideal generated by 
    \[
      \{ x_i \mid \beta \cdot \chi_i \not \in \mathbb Z_{\geq 0} \}.
    \]
    Then we set $W_{\beta}^{\mathrm{ss}} = W_{\beta} \cap W^{\mathrm{ss}}$ so that $[W_{\beta}^{\mathrm{ss}}/ \mathbb C^*]$ is a closed substack of $X$.
    \begin{proposition}
      \label{propgamma}
      The composition $f\circ \sigma_\star :F_{\beta} \to X$
      factors through an isomorphism 
      \[
        F_{\beta} \overset{\cong}{\longrightarrow} 
        [W_{\beta}^{\mathrm{ss}}/ \mathbb C^*].
      \]
      In particular, $F_\beta$ is nonempty if and only if
      $e_1, \ldots, e_m \in \mathbb Z_{\geq 0}$ and $W_{\beta}^{\mathrm{ss}} \neq \emptyset$.
    \end{proposition}

    \begin{proof} This is a consequence of \cite[~Corollary 5.2]{CCK15}, and the
      discussion that follows it. For convenience, we present a detailed proof here.

      Recall that for a base scheme $S$, the data of a map $S \to \mathfrak Y =
      \left[\mathbb C^{m+n+1} / (\CC^*)^{m+1} \right]$ consists of line bundles
      $\mathcal M_0 ,\ldots, \mathcal M_m$ on $S$ and sections $t_j \in H^0(S,
      \mathcal M^{\otimes w_j}_0 \otimes \bigotimes_{i=1}^m \mathcal M_i^{\otimes a_{i j}})$
      for $j = 0 ,\ldots, m +
      n$, where $\{a_{ij}\}$ are the entries of the weight matrix \eqref{eq:wmatrix} for the torus action.

      Given such data, suppose that the map factors through
      \begin{equation}
        \label{eq:W-beta-ss-quot}
        [W_{\beta}^{\mathrm{ss}}/ \mathbb C^*] = 
        \big[\big(W_{\beta}^{\mathrm{ss}} \times (\mathbb C\setminus \{0\})^{m}\big) /
        \big(\mathbb C^*\times (\mathbb C^*)^m\big)\big] \subset \mathfrak Y,
      \end{equation}
      for $\beta = (e_0 ,\ldots, e_m)$,
      so that we have an
      object of $[W_\beta^{\ss} / \CC^*](S)$. Note that in this situation, for any $k
      = 1 ,\ldots, m$, we have $t_{n+k} \in H^0(S, \mathcal M_k)$ and it must be
      non-vanishing. Thus the data
      $(\mathcal M_0 ,\ldots, \mathcal M_m, t_0 ,\ldots, t_{m+n})$ is actually equivalent
      to $(\mathcal M_0, t_0 ,\ldots, t_n)$.
      Now form the line bundles
      \[
        \mathcal L_i = \mathcal O_{\mathbb P(1,r)}(re_i) \boxtimes
        \mathcal M_i, \quad i = 0 ,\ldots, m
      \]
      on $\mathbb P(1,r) \times S$ and for $j=0 ,\ldots, n+m$ we set
      \[
        s_j = \begin{cases}
          y^{ \beta \cdot \chi_j} \boxtimes t_j
          & \text{ if }\beta \cdot \chi_j  \in \ZZ_{\geq 0},\\
          0 & \text{ otherwise}.
        \end{cases}
      \]
      By
      definition, the data $\{\mathcal L_i ,s_j\}$ represents a map $\mathbb
      P(1,r)\times S \to \mathfrak X$, and by Lemma \ref{lem:computation-on-P1r} this
      map is $\CC^*$-invariant, hence defines an object of $F_\beta(S)$. This defines
      a functor
      \[ \left[W_\beta^{\mathrm{ss}} / \CC^* \right](S)
        \longrightarrow F_{\beta}(S),
      \] and it is easy to see that it is fully faithful. By the proof of
      Lemma~\ref{lem:direct-image-positive-line-bundle}, and
      Corollary~\ref{cor:invariant-sections-line-bundle}, any object of $F_\beta(S)$
      must be isomorphic to one of the form
      $(\mathcal  L_i, s_j)$ as above, hence the functor is essentially
      surjective. Thus, we get an equivalence of categories. It is also clear that
      this equivalence is inverse to $f\circ \sigma_\star$, as restricting the data
      $\{\mathcal L_i, s_j\}$ to $\Sigma_\star \times S$ recovers the data
      $\{\mathcal M_i, t_j\}$.

      Although we have assumed $F_{\beta} \neq \emptyset$ in the
      beginning of this section, it is easy to see that in general whenever $e_1
      ,\ldots,  e_m\in \mathbb Z_{\geq 0}$ and $W^{\mathrm{ss}}_{\beta} \neq
      \emptyset$, the $\{\mathcal L_i, s_j\}$ constructed above is an $S$-point
      of $F_{\beta}$. Hence $F_{\beta}$ is nonempty.
    \end{proof}

    \begin{remark}
      The proof of Proposition~\ref{propgamma} also describes the universal map
      \eqref{eq:universal-quasimap-F-beta} explicitly in terms of line bundles and
      sections. Indeed, take
      $\mathcal M_i$ to be the tautological line bundles over
      the stack $[W_\beta^{\ss}/\mathbb C^*]$ viewed as a substack of
      $\mathfrak Y$ via \eqref{eq:W-beta-ss-quot}, and take $t_j$ to be the tautological
      sections corresponding to the affine coordinates on $\mathbb C^{m+n+1}$. Then
      the $\{\mathcal L_i, s_j\}$ give the universal map.
    \end{remark}

\subsection{Perfect Obstruction Theory of $F_\beta$} \label{sec:I-POT}
 Throughout this section, we will abuse notation and use $\TT_\fX$ to refer to
 the complex $i^*\TT_{\fX^{\der}}$, as in \eqref{eq:qpot}. When $\fX$ is a
 complete intersection, this complex agrees with the usual tangent complex of
 $\fX$.

    Recall that the perfect obstruction theory $(R\pi_*f^*
    \TT_{\mathfrak{X}})^\vee \to \LL_{ Q_{\PP(1,\bullet)}(X,
      \beta)}$ is equivariant under the $\CC^*$-action scaling the domain curve. Restricting to
    $F_\beta$ and letting
    \begin{equation*}
      \EE: = (R\pi_*f^* \TT_{\mathfrak{X}})|_{F_\beta},
    \end{equation*}
    we have by \cite[Proposition~1]{GrPa99} that the fixed part $\mathbb T_{F_\beta}
    \to \EE^{\fix}$ defines a perfect obstruction theory, hence a virtual cycle
    $[F_\beta]^{\vir}$, while the virtual normal bundle is defined via the moving part of
    $\mathbb E,$
    \begin{equation} \label{def:virtual-normal}
      N^{\vir}_{F_\beta / Q_{\PP(1, \bullet)}(X, \beta)} := \EE^{\mov}.
    \end{equation}
    The goal of this section is to compute $[F_\beta]^{\vir}$ and the
    $\CC^*$-equivariant Euler class $e^{\mathbb C^*}(\EE^{\mov})$. Putting them together, we will obtain a closed formula for the extended $I$-function.

    We can understand $\EE$ through a pair of exact triangles. The first one is
    the generalized Euler sequence (cf.\ \cite[Section~5.2]{CiKi10}) for
    $\mathfrak Y$
    \begin{equation}
      \label{eq:eulerseq}
      \cO^{\oplus m+1}_{\mathfrak{Y}} \to \bigoplus_{i=0}^{m+n} \cO_{\mathfrak{Y}}(\chi_i) \to \TT_{\mathfrak{Y}} \xrightarrow{+1}
    \end{equation}
    where we recall the line bundle $\cO_{\mathfrak{Y}}(\chi_i)$ is defined as
    in \eqref{eq:vb-notation} but with the $i$-th column $\chi_{i}$ of the
    matrix \eqref{eq:wmatrix}.

    The second triangle is obtained by dualizing the conormal triangle
    \cite[~Theorem 8.1]{O07} for the closed embedding $\mathfrak{X}
    \overset{i}{\hookrightarrow} \mathfrak{Y}$. Since $\fX$ is defined by a section of the bundle $\mathcal E$
    introduced in \eqref{eq:vb-notation}, the dual conormal triangle reads as
    \begin{equation} \label{eq:coseq}
    \cE[-1] \to \TT_{\mathfrak{X}} \to
      i^*\TT_{\mathfrak{Y}} \xrightarrow{+1}.
    \end{equation}

    Now consider the universal family \eqref{eq:universal-quasimap-F-beta}.
    Pulling back and pushing forward the exact triangles \eqref{eq:eulerseq} and
    \eqref{eq:coseq} to $F_\beta$ gives the exact triangles
    \begin{equation}
      \label{eq:Reuler}
      R\pi_*f^*\cO^{\oplus m+1}_{\mathfrak{Y}}\to
      R\pi_*f^*\left(\bigoplus_{i=0}^{m+n} \cO_{\mathfrak{Y}}(\chi_i) \right) \to
      R\pi_*f^*\TT_{\mathfrak{Y}} \xrightarrow{+1},
    \end{equation}
    and
    \begin{equation} \label{eq:Rcoseq} \EE \to
      R\pi_*f^*\TT_{\mathfrak{Y}} \to
      R\pi_*f^*\cE \xrightarrow{+1}.
    \end{equation}
    
    With this we can compute the equivariant Euler class of the virtual normal bundle.
    Recall that for any $x\in \mathbb Q$, $\langle  x \rangle := x - \lfloor x \rfloor$ denotes the
    fractional part of $x$.
    \begin{lemma} \label{lem:euler-class}
      For $\beta = (e_0, \dots, e_m)$, we have
      
      \[
        e^{\mathbb C^*}(N^{\mathrm{vir}}_{F_\beta / Q_{\PP(1,\bullet)}(X, \beta)}) =
        \left(\prod_{a=1}^m  e_a! z^{e_a}  \right) \prod_{i=0}^{n}
         \dfrac{ \prod\limits_{\substack{ \langle k \rangle =
              \langle \beta \cdot \chi_i \rangle \\ 0 < k \leq \beta
              \cdot \chi_i}} \left( w_i H |_{F_\beta} + kz\right) } { \prod\limits_{\substack{
              \langle k \rangle = \langle \beta \cdot \chi_i \rangle \\ \beta
              \cdot \chi_i < k <0}} \left(w_iH|_{F_\beta} + kz\right) }
        \cdot \prod_{j=1}^{n-3}
               \dfrac{
          \prod\limits_{\substack{ \langle k \rangle = \langle \beta \cdot
              \xi_j \rangle \\ \beta \cdot \xi_j < k < 0}} \left(b_j
          H|_{F_\beta} + kz\right) } { \prod\limits_{\substack{ \langle k \rangle =
              \langle \beta \cdot \xi_j \rangle \\ 0 < k \leq \beta \cdot
              \xi_j}} \left(b_j H|_{F_\beta} + kz\right) }
      \]
      where $H = c_1(\cO_{\PP(\vec w)}(1))$ is the hyperplane class. 
    \end{lemma}
    \begin{proof}
      Recall that $\mathcal E$ has the form of a split bundle \eqref{eq:vb-notation}. Then taking the moving part of the exact triangles \eqref{eq:Reuler}, \eqref{eq:Rcoseq}, and using
      Lemma~\ref{lem:direct-image-positive-line-bundle}, the result follows immediately. 
    \end{proof}

   Now we turn our attention to the fixed part of the triangles \eqref{eq:Reuler} and \eqref{eq:Rcoseq} in order to understand the virtual cycle $[F_\beta]^{\vir}$.

     Let $\mathbb P_{\beta} \subset \mathbb P(\vec w)$ be the substack
     defined by the vanishing of all the $x_i$ for which $\beta \cdot \chi_i \not \in \mathbb Z_{\geq 0} $.
     Thus
     Proposition~\ref{propgamma} gives $F_{\beta} \cong \mathbb P_{\beta} \cap X$.
     Set
     \[
       \cE_{\beta} = \bigoplus\limits_{\beta \cdot \xi_j \in \ZZ_{\geq 0}}
       \cO_{\mathbb P(\vec w)}(b_j)
     \]
     where we recall $\xi_j$ is the multi-degree of $\tilde F_{j}$ as in \eqref{eq:vb-notation}.
   
   Let
   \[
     \tilde F = (\tilde{F}_1 ,\ldots, \tilde{F}_{n-3})
   \]
   be the defining equations for the extended GIT, 
   as in Lemma~\ref{lem:F-extension}. By abuse of notation, we will also denote
   its descent to $\mathfrak Y$ by $\tilde F$. Recall that it is a section of $\mathcal
   E$ and its restriction to $\mathbb P(\vec w)$ cuts out $X$. 
   \begin{lemma}
     \label{lem:factors-through-E-beta}
       The restriction $\tilde F|_{\mathbb P_{\beta}}$ of the section $\tilde F$ to $\PP_\beta$ factors through $\cE_{\beta}|_{\mathbb P_{\beta}}$. 

       \end{lemma}
   \begin{proof}
     Suppose $\beta\cdot \xi_j \not \in \mathbb Z_{\geq 0}$.
     Then since $\deg(\tilde F_j) = \xi_j$, every term of the polynomial
     $\tilde F_j$ must contain some $x_i$ for which
     $\beta\cdot \chi_i \not\in \mathbb Z_{\geq 0}$, hence $\tilde F_{j}$ vanishes
     on $\mathbb P_{\beta}$.
   \end{proof}
    
    \begin{proposition}
      \label{prop:virtualclass}
      The virtual cycle of $F_\beta$ is the Euler class of
        $\cE_{\beta}|_{\mathbb P_{\beta}}$, localized by the section
        $\tilde F|_{\mathbb P_{\beta}}$. That is,
        \[
          [F_\beta]^{\vir} = e_{\mathrm{loc}, \tilde F|_{\mathbb P_{\beta}}}(\cE_{\beta}|_{\mathbb P_{\beta}}).
        \]
 
          In particular, its virtual dimension is
        \[
          \begin{aligned}
            \mathrm{vdim}~F_\beta = &  \dim \mathbb P_{\beta} -
            \mathrm{rank}~\mathcal E_{\beta} \\
            = & \#\{i\mid \beta \cdot \chi_i \in \mathbb Z_{\geq 0}, 0 \leq i \leq n\} - \#\{j\mid
                \beta\cdot \xi_j \in \mathbb Z_{\geq 0}, 1 \leq j \leq n - 3\}
                - 1.
          \end{aligned}
        \]
    
    \end{proposition}
    \begin{proof}
   Applying Lemma~\ref{lem:direct-image-positive-line-bundle} and Corollary~\ref{cor:invariant-sections-line-bundle} to
       the fixed part of the exact triangle \eqref{eq:Reuler} yields the short exact sequence
        \[
          0 \to \cO^{\oplus m+1}_{F_{\beta}} \to \bigoplus_{\beta\cdot \chi_i \in
            \mathbb Z_{\geq 0}} \cO_{\mathfrak{Y}}(\chi_i)|_{F_{\beta}} \to
         \left(R^0\pi_*f^*\TT_{\mathfrak{Y}}|_{F_\beta}\right)^{\mathrm{fix}} \to 0.
        \]
        Viewing $F_{\beta}$ as a substack of $\mathfrak X$ via
        the embedding $f\circ\sigma_{\star}$ as in Proposition \ref{propgamma},
        it is clear that $F_\beta \subset \PP_\beta$. We then see that the above exact
        sequence is simply the generalized
        Euler sequence for $\mathbb P_{\beta}$ restricted to $F_\beta$. Thus
        $ \left(R^0\pi_*f^*\TT_{\mathfrak{Y}}|_{F_\beta}\right)^{\mathrm{fix}}  \cong
        T_{\mathbb P_{\beta}}|_{F_{\beta}}$. 
        
        Using this identification and again applying
        Lemma~\ref{lem:direct-image-positive-line-bundle}, the fixed part of the
        second exact triangle \eqref{eq:coseq} becomes the exact triangle
        \[
          \mathbb E^{\mathrm{fix}} \to T_{\mathbb P_{\beta}}|_{F_{\beta}} \to \mathcal
          E_{\beta}|_{F_{\beta}}.
        \]
        By Lemma \ref{lem:factors-through-E-beta}, the map $T_{\mathbb
          P_{\beta}}|_{F_{\beta}} \to \mathcal
        E_{\beta}|_{F_{\beta}}$ is given by the differentiation of $\tilde F|_{\PP_\beta}$, 
        hence the virtual cycle is the localized Euler class. 
    \end{proof}

      When computing the $I$-function, we also need to study
      $\hat{\ev}_{\star} = \mathrm{Inv}\circ \ev_{\star}$, where $\ev_{\star}$
      is the evaluation map at the unique marking and $\mathrm{Inv}$ is the
      involution on the inertia stack. In other words, it is defined by
      restricting the universal map to the gerbe at the unique marking, with
      reversed $\mu_r$-banding.
      By definition, it maps into the rigidified 
      inertia stack. However, the choice of $\sigma_\star$
      above Lemma~\ref{lem:direct-image-positive-line-bundle}
      fixes a trivialization
      of the gerbe.
      Hence $\hat{\ev}_{\star}$ naturally lifts to the
      inertia stack $\bigsqcup_\alpha X_{\alpha}$, which we still denote by
      $\hat{\ev}_{\star}$.
      Moreover, it is easy to see that given $\beta = (e_0, \dots, e_m)$,
      the map $\hat{\ev}_{\star}$ factors through
      $X_{\alpha_\beta}$, where 
      \[ \alpha_\beta =
        \langle - e_0 \rangle
      \]
      is the fractional part of $-e_0$.
      Thus, we obtain a closed embedding $\hat{\ev}_{\star}\colon F_{\beta}  \to
      X_{\alpha_\beta}$. The composition of $\hat{\ev}_{\star}$ with the
      natural inclusion $X_{\alpha_{\beta}} \to X$ is precisely the embedding
      $f\circ \sigma_{\star}$ in Proposition~\ref{propgamma}.

      We next compute
      $\mathrm{PD}((\hat{\mathrm{ev}}_{\star})_*([F_{\beta}]^{\mathrm{vir}}))$,
      the Poincar\'e dual of
      $(\hat{\mathrm{ev}}_{\star})_*([F_{\beta}]^{\mathrm{vir}})$.
      By Lemma~\ref{lem:twisted-CI} and Proposition~\ref{prop:virtualclass}, it
      is a cohomology class of degree $2(\mathrm{vdeg}~F_\beta)$, where
      \[
        \mathrm{vdeg}~F_\beta := \# \{ i \, | \, \beta \cdot \chi_i \in
        \ZZ_{<0}, 0 \leq i \leq n \} - \# \{ j
        \, | \, \beta \cdot \xi_{j} \in \ZZ_{<0}, 1 \leq j \leq n - 3 \}
      \]
      is the virtual codimension $\dim X_{\alpha_{\beta}} -
      \mathrm{vdim}~F_{\beta}$.
      \begin{lemma}
        \label{lem:vdeg->=0}
        Whenever $X_{\alpha_\beta} \neq \emptyset$, we have
        $\mathrm{vdeg}~F_\beta \geq 0$.
      \end{lemma}
      \begin{proof}
        When $F_{\beta} \neq\emptyset$, this follows easily from
        $\mathrm{vdim}~F_{\beta} \leq \mathrm{dim}~F_{\beta} \leq
        \mathrm{dim}~X_{\alpha_{\beta}}$. We give an alternative proof without
        assuming $F_{\beta} \neq\emptyset$. Recall from
        Lemma~\ref{lem:twisted-CI} that $X_{\alpha_{\beta}}$ is a complete
        intersection in $\mathbb P_{\alpha_{\beta}}$ defined by those $F_j$ such
        that $\beta \cdot \xi_j \in \mathbb Z$. Suppose $\beta \cdot \xi_j \in
        \mathbb Z_{<0}$, then $F_j \equiv 0$ modulo the ideal generated by $\{x_i
        \mid \beta \cdot \chi_i \in \mathbb Z_{<0}\}$ and $\{x_i \mid \beta
        \cdot \chi_i \not \in \mathbb Z\}$. Thus, $X_{\alpha_{\beta}} \subset
        \mathbb P_{\alpha_{\beta}}$ contains the common vanishing of
        $\{F_j \mid \beta\cdot \xi_j \in \mathbb Z_{\geq 0}\}$ and $\{x_i \mid
        \beta \cdot \chi_i \in \mathbb Z_{<0}\}$. Thus for dimension reasons we
        must have $\mathrm{vdeg}~F_{\beta} \geq 0$.
      \end{proof}
      We define $\mathbf 1_{\beta} \in H^0(X_{\alpha_\beta})$ to be
      \[
        \mathbf 1_{\beta} =
        \begin{cases}
          \mathbf 1_{X_{\alpha_\beta}} & \text{when } \dim(X_{\alpha_\beta}) > 0,\\
          \frac{\deg([X_{\alpha_\beta}])}{\deg([F_{\beta}])}\mathbf
          1_{\hat{\mathrm{ev}}_{\star}(F_{\beta})}
          =
          \frac{\prod_{\beta \cdot \xi_j \in \mathbb Z_{<0}} b_j}
          {\prod_{\beta\cdot \chi_i \in \mathbb Z_{<0}} w_i}
          \mathbf 1_{\hat{\mathrm{ev}}_{\star}(F_{\beta})} & \text{when }\dim(X_{\alpha_\beta}) = 0.
        \end{cases}
      \]
      Note that by Proposition~\ref{propgamma}, in the second case
      $\hat{\mathrm{ev}}_{\star}(F_{\beta}) \subset X_{\alpha_{\beta}}$
      is just the fundamental class of the special cycle $\overline{U_{\Lambda}}\cap X_{\alpha_\beta}$
      introduced in Section~\ref{sec:extendable-classes} if we take
      \[
        \Lambda = 
        \{ i \, | \, \beta \cdot \chi_i \in \ZZ_{<0}\}.
      \]
      Hence $\mathbf 1_{\beta}$ is always in the admissible state space $\mathcal H$.
       Indeed, when the age of
        $X_{\alpha_\beta}$ is one, it is a multiple of one of our chosen classes
      \eqref{eq:t-label} used in the construction of the extended GIT.
      \begin{lemma} \label{lem:virtualclass}
        We have
        \[
          \mathrm{PD}((\hat{\mathrm{ev}}_{\star})_*([F_{\beta}]^{\mathrm{vir}}))
          = \frac{\prod_{\beta\cdot \chi_i \in \mathbb
              Z_{<0}} w_i}{\prod_{ \beta \cdot \xi_j \in \mathbb Z_{<0}} b_j}
          \cdot H^{\mathrm{vdeg}\!\ F_{\beta}} \cdot \mathbf 1_{\beta},
        \]
        where $H$ is the hyperplane class on $X_{\alpha_{\beta}}$.
      \end{lemma}
      \begin{proof}
        First assume that $\dim(X_{\alpha_\beta}) > 0$.
        Recall the definition of
        $\mathbb P_{\alpha}$ from \eqref{eq:X-alpha} and let $i\colon X_{\alpha_\beta}
        \to \PP_{\alpha_\beta}$ be the inclusion. 
        Since $\dim(X_{\alpha_\beta}) \neq 2$ by Lemma~\ref{lem:CY3-sector-dim},
        the Lefschetz hyperplane theorem for orbifolds implies the left hand side
        must be 
          the pullback of $cH^{\mathrm{vdeg}\!\ F_{\beta}}$, for some constant factor $c \in
          \mathbb Q$.
        We then apply Proposition~\ref{prop:virtualclass}.
        The Poincar\'e dual
        of $[\mathbb P_{\beta}]$ in $\mathbb P_{\alpha_{\beta}}$ is equal to
        $\prod_{\beta \cdot \chi_i  \in
          \mathbb Z_{<0}} (w_i H)$. Combining this with
        Lemma~\ref{lem:twisted-CI}, we have  in $H^*(\mathbb P_{\alpha_{\beta}})$
        \[
          \prod_{ \beta \cdot \xi_j \in \mathbb Z} (b_j
          H) \cdot (c H^{\mathrm{vdeg}\!\ F_{\beta}})
          = 
          i_*(\mathrm{PD}((\hat{\mathrm{ev}}_{\star})_*([F_{\beta}]^{\mathrm{vir}})))
          = \prod_{\beta \cdot \chi_i  \in
            \mathbb Z_{<0}} (w_i H)\prod_{\beta \cdot \xi_j  \in \mathbb Z_{\geq
              0}} (b_j H).
        \]
        This determines the constant factor.

        We then come to the case $\dim(X_{\alpha_{\beta}}) = 0$. Then both sides
        are $0$ unless $\mathrm{vdeg}\!\ F_{\beta} = 0$. In case
        $\mathrm{vdeg}\!\ F_{\beta} = 0$, we have $[F_{\beta}] =
        [F_{\beta}]^{\mathrm{vir}}$ and the right hand side becomes $\mathbf 1_{F_{\beta}}$.
      \end{proof}

      Finally we
      obtain the following formula for the $I$-function.

      \begin{theorem}
        \label{thm:I-general}
        The extended $I$-function for $X$ with GIT data \eqref{eq:X-ex-git} is given by
        \[
          I(q, z) = \sum_{\beta  = (e_0 ,\ldots, e_m)}
          \frac{q^\beta}{\prod_{a=1}^m \left(e_a!\right) z^{e_a}}
          \prod_{i=0}^{n} 
          \dfrac{ \prod\limits_{\substack{
                \langle k \rangle = \langle \beta \cdot \chi_i \rangle \\ \beta
                \cdot \chi_i< k \leq 0}} \left(w_iH + kz\right) }
          {\prod\limits_{\substack{ \langle k \rangle =
                \langle \beta \cdot \chi_i \rangle \\ 0 < k \leq \beta \cdot
                \chi_i}} \left( w_i H  + kz\right) } 
          \cdot \prod_{j=1}^{n-3}
          \dfrac{ \prod\limits_{\substack{ \langle k \rangle =
                \langle \beta \cdot \xi_j \rangle \\ 0 < k \leq \beta \cdot
                \xi_j}} \left(b_j H + kz\right) }
         { \prod\limits_{\substack{ \langle k \rangle = \langle \beta \cdot
                \xi_j \rangle \\ \beta \cdot \xi_j < k \leq 0}} \left(b_j
              H + kz\right) } \cdot
          \mathbf 1_{\beta}
        \]
        where $q$ is the Novikov variable and $H$ is the hyperplane class.
      \end{theorem}
      Some remarks are in order before proving the theorem.
      \begin{enumerate}
      \item
        There are two
        equivalent ways of interpreting the formula. One can either take $H$ to be the
        hyperplane class on each $X_{\alpha}$ and use the classical cup product, or
        take $H$ to be the hyperplane class on the untwisted sector $X_0$ and use
        the Chen--Ruan product.
      \item
        There could be powers of $H$
        simultaneously in the numerator and denominator coming from the $k=0$
        terms. We formally define $\frac{H^a}{H^b} = H^{a - b}$. It is easy to check
        we always have $a \geq b$ in the formula.
        Indeed, the powers always simplify to
        $H^{\mathrm{vdeg}\!\ F_{\beta}}$, and we have
        $\mathrm{vdeg}~F_{\beta}\geq 0$ by Lemma~\ref{lem:vdeg->=0}.
        
      \item
        A priori the sum is over all $\beta$ such that $F_{\beta} \neq
        \emptyset$. However, we can sum over the larger set of all $\beta = (e_0 ,\ldots,
        e_m)$ such that $e_1 ,\ldots, e_m \in \mathbb Z_{\geq 0}$ and
        $X_{\alpha_\beta} \neq \emptyset$. Indeed, if $F_{\beta} = 
        \emptyset$, we define $\mathrm{vdim}~F_\beta$ and
        $\mathrm{vdeg}~F_\beta$ using the same explicit formula as before. Then
        by Proposition~\ref{propgamma}, we must have $\mathrm{vdim}~F_{\beta} < 0$. 
        Then $\mathrm{vdeg}~F_{\beta} > \mathrm{dim}~X_{\alpha_{\beta}}$ and we
        have $H^{\mathrm{vdeg}} \cdot \mathbf 1_{\beta} = 0$ automatically.
      \item
          The formula coincides with the $I$-function in \cite{Wa19P}, after
          taking GIT quotient in \cite{Wa19P} to be the one in this paper,
          setting the parameters $t_i$ in \cite{Wa19P} to be zero, and computing
          the term 
          $ i_{*}( s_{E_{\beta},{loc}}^{!} ( [Z_{\beta}^{\mathrm{ss}} /
          ( G / \langle g_{\beta}^{-1} \rangle )])) $ in \cite{Wa19P} explicitly
          as in Lemma~\ref{lem:virtualclass}.
      \end{enumerate}

      \begin{proof}[Proof of Theorem~\ref{thm:I-general}]
        This follows immediately from
        combining Lemma~\ref{lem:euler-class} and Lemma~\ref{lem:virtualclass}.
        The factor in Lemma~\ref{lem:virtualclass} is moved into the $k=0$ terms, and
        the extra factor $\mathbf r$ in the definition of the
        $I$-function \eqref{eq:I} cancels with the degree $\mathbf r^{-1}$ of the
        rigidification map $\mathcal IX \to \overline{\mathcal I}X$
       (cf.\ Remark~\ref{rmk:state-space-convention}).
      \end{proof}

\section{Mirror Theorem}  
\label{sec:mirror}

We will now apply quasimap wall-crossing to the extended $I$-function
computed in the previous section. Quasimap wall-crossing is an explicit relationship between
$\epsilon$-stable quasimap
invariants for different values of $\epsilon$.
It has been established for complete intersections in ordinary
projective space by Ciocan-Fontanine--Kim \cite{CiKi20} and Clader--Janda--Ruan \cite{CJR17P1}, and has recently been
extended by the third author to all (including orbifold) GIT quotients in \cite{Zh19P}.
In particular, the latter applies to our target space.

  In Section~\ref{sec:wc}, we treat the general case of quasimap wall-crossing
  for an orbifold Calabi--Yau target. In this case, we apply the divisor and dilaton 
  equations to rewrite the wall-crossing formula as a change of variables, which
  we call the mirror map. In Section~\ref{sec:invertibility}, 
  we specialize to the case of the extended GIT from
  Section~\ref{sec:extended-git}.
  We carefully choose a basis of the space of
  curve classes so that the $I$-function becomes a power series. We then show that the
  mirror map is invertible, recovering all the genus-zero invariants with
  insertions in the admissible state space $\mathcal H$.

\subsection{Wall-Crossing for Calabi--Yau targets} \label{sec:wc}
\phantom{}

  Let $X \subset \mathfrak X$ be a target space for quasimap theory as in
  Section~\ref{sec:quasimap}. Then $\infty$-stable quasimap invariants agree
  with the Gromov--Witten invariants of $X$, i.e., for any curve class
  $\beta \in \mathrm{Hom}(\mathrm{Pic}(\mathfrak X), \mathbb Q)$, and any
  $\mathbf t_1 ,\ldots, \mathbf t_n \in H^*_{\CR}(X, \QQ)[z]$, we have
  \[
    \langle \bt_1(\psi), \dotsc, \bt_n(\psi)\rangle_{g, \beta}^{\infty}
    = \langle \bt_1(\psi), \dotsc, \bt_n(\psi)\rangle_{g, \beta}^{X}.
  \]
  Here, the right hand side denotes the sum of Gromov--Witten invariants
  $\langle \bt_1(\psi), \dotsc, \bt_n(\psi)\rangle_{g, \gamma}^{X}$ of $X$
  for all effective curve classes $\gamma \in \mathrm{NE}(X)$ such that
  \begin{equation}
    \label{eq:comparison-curve-classes}
    \beta(L) =
    \int_{\gamma} c_1(L|_{X}),
    \quad \forall L \in \mathrm{Pic}(\mathfrak X).
  \end{equation}
  Taking $L$ to be $L_{\theta}$ (cf.\ Section~\ref{sec:quasimap}), whose restriction to $X$ is ample, we see that
  for each $\beta$ there are indeed only finitely many such $\gamma$.

Now let $\{T_p\}$ be a basis for $H^*_{\CR}(X, \QQ)$ and $\{T^p\}$ be the dual
basis.
We write $\mathbf 1 := \mathbf 1_{X_0}$ for the fundamental class of the
  untwisted sector. It is the unit with respect to the Chen--Ruan product.
For generic
$\bt(z) \in H^*_{\CR}(X, \QQ)\llbracket z \rrbracket$, define the
\emph{big $J$-function} as
\begin{equation} \label{eq:J}
  J(\bt(z), q, z) = \mathbf 1 + \frac{ \bt(-z)}{z} +
    \sum\limits_{\substack{\beta \in \mathrm{Eff}(W,G,\theta) \\ k \geq 0}}
  \frac{q^\beta}{k!}\sum_p T_p \left\langle
  \frac{T^p}{z(z-\psi)}, \bt(\psi), \dots, \bt(\psi)\right\rangle_{0, 1+k,
    \beta}^{\infty}
\end{equation}
where the unstable terms (corresponding to empty moduli spaces)
are interpreted as zero.
Define
\begin{equation} \label{eq:mu-general}
  \mu(q, z) = [zI(q,z) - z\mathbf 1]_{+},
\end{equation}
where $[\cdot]_+$ refers to truncating the terms with negative powers
of $z$.
By taking $\epsilon \to 0^+$ and $\bt(z) = 0$ in
\cite[Theorem~1.12.2]{Zh19P}, or by \cite[Theorem~6.7]{Wa19P} in the
toric hypersurface case, we get the following Theorem~\ref{thm:wc}.
We note that Ciocan-Fontanine--Kim--Maulik's definition of quasimap invariants \cite{CKM14} and
their wall-crossing assume that the target space has at worst local complete
intersection singularities.
However, the local complete intersection assumption was only used to show that the canonical obstruction
theory is perfect of amplitude $[-1, 0]$ in \cite[Proposition 4.4.1]{CKM14}.
The proof of Lemma~\ref{lem:qpot} extends to the master space construction in \cite{Zh19P} verbatim.
Thus the wall-crossing formula of \cite{Zh19P} still holds in our setting.
\begin{theorem}[\cite{Wa19P, Zh19P}]
  \label{thm:wc}
  We have
  \begin{equation}
    \label{eq:wall-crossing}
    J( \mu(q,-z), q, z) = I(q, z).
  \end{equation}
\end{theorem}
  Note that $\mu \equiv 0$ modulo $q$. Hence the left hand side is well-defined.

  We say the target is Calabi--Yau if the (derived) tangent complex of
  $\mathfrak X$ has trivial determinant (see Lemmas~\ref{lem:qpot} and
  \ref{lem:canonical-bundle}).
  For such a target, the ages $\iota_{\alpha}$ in
    Section~\ref{sec:CR-cohom} are all integers.
    Recall that in this paper we assume that the generic point has
    trivial isotropy.
  Thus
  \[
    H^2_{\mathrm{CR}}(X) = H^2(X) \oplus H^0((\mathcal IX)_{\mathrm{age} = 1}),
  \]
  where $(\mathcal IX)_{\mathrm{age} = 1} \subset \mathcal IX$ denotes the union of the age $1$ sectors.

\begin{corollary}
  \label{cor:wallc} 

    If the target is Calabi--Yau,
  we have
  \begin{equation}
    \label{eq:mu}
    \mu(q, z) = z(I_0(q) - 1)\cdot \mathbf 1 +  I_1(q) +  I_1^{\prime}(q),
  \end{equation}
  where
  \[
    I_0 \in
    \mathbb Q \llbracket q \rrbracket, \quad
    I_1 \in H^2(X) \llbracket q \rrbracket, \quad
    I^{\prime}_1 \in H^0((\mathcal IX)_{\mathrm{age} = 1}) \llbracket q \rrbracket.
  \]
  Equation \eqref{eq:wall-crossing} can then be rewritten as
  \begin{equation}
    \label{mirroreq}
    \begin{aligned}
      \exp{\left(-\frac{ I_1}{zI_0}\right)}
      \frac{I}{I_0} =
      &
        \mathbf 1 + \frac{\bt}{z}  +
        \sum_{\substack{
        \gamma \in \NE(X), k \geq 0 \\
      (k, \gamma)\neq (1,0),(0,0)
      }}
      \frac{Q^{\gamma}}{k!} \sum_{p} T_p
      \left\langle
      {\frac{T^p }{z(z-\psi)},
      \bt ,\ldots, \bt }
      \right\rangle^{X}_{0,1 + k, \gamma}
    \end{aligned},
  \end{equation}
    where
    \begin{equation}
      \label{eq:change-of-coordinates1}
      \bt = \frac{I_1^\prime}{I_0} \quad \text{and} \quad 
      Q^{\gamma} = q^{\beta_\gamma} \exp(\int_{\gamma}  I_1/I_0),
    \end{equation}
    for each effective curve class $\gamma \in \mathrm{NE}(X)$, where
    $\beta_{\gamma}$
    is the $\beta$ defined in 
    \eqref{eq:comparison-curve-classes}.

\end{corollary}

\begin{proof}
  If we assign degree $2$ to the equivariant parameter $z$ and set the degree of a cohomology class to be its Chen--Ruan
  degree, then the $I$-function is homogeneous of degree $0$.

 Thus
  we have
  \[
    I(q,z) \in \bigoplus_{i\geq 0} z^{-i}H^{2i}_{\mathrm{CR}}(X, \mathbb Q) \llbracket q \rrbracket,
  \]
  hence $\mu(q,z)$ must be of the form \eqref{eq:mu}.
Equation \eqref{eq:wall-crossing} may then  be rewritten as
  \begin{align*}
    I(q, z) &= J(-z(I_0(q) - 1)\cdot \mathbf 1 +  I_1(q) +  I_1^{\prime}(q), q, z) \\
            &= I_0(q) J((I_1(q) +  I_1^{\prime}(q))/I_0(q), q, z) \\
            &= I_0(q) \exp(I_1(q)/zI_0(q)) J(\bt, Q, z),
  \end{align*}
  under the variable change \eqref{eq:change-of-coordinates1}.
  Here, we used the dilaton and divisor equations in the second
  and third equalities, respectively.
\end{proof}
\begin{remark}
  In the previous corollary, it seems more natural to only use curve classes $\beta \in
    \mathrm{Hom}(\mathrm{Pic}(\mathfrak X), \mathbb Q)$ rather than using curve
    classes $\gamma$ in $X$. However, we are not sure if in general the term $\int_{\gamma}
    I_1/I_0$ only depends on $\beta_\gamma$, the image of $\gamma$
  in $\mathrm{Hom}(\mathrm{Pic}(\mathfrak X), \mathbb Q)$.
\end{remark}

\subsection{Invertibility of the Mirror Map}  \label{sec:invertibility}

  We now return to our extended GIT construction for complete intersections. 
  Recall from Section~\ref{sec:extension-data} that $\mathcal
  H^{2}_{\mathrm{tw}}\subset \mathcal H$ is the subspace  of admissible classes
  of Chen--Ruan degree $2$ supported on the twisted sectors, and that invariants with
  insertions from $\mathcal H^{2}_{\mathrm{tw}}$ recover all the invariants with
  insertions from $\mathcal H$, since the target space is a Calabi--Yau
  threefold.
  We have chosen a basis $\{\phi_1 ,\ldots,
  \phi_m\}$ for $\mathcal H^{2}_{\mathrm{tw}}$ in \eqref{eq:t-label}.

Our goal now is to show that we can extract the Gromov--Witten invariants with
arbitrary admissible insertions
$\phi_i$ in \eqref{eq:t-label} from the wall-crossing equation
\eqref{mirroreq}.

We establish this by proving that the mirror map is an invertible change of
coordinates. But first we need to  choose a basis for
quasimap curve classes to write the $I$-function as a power series.

  Recall
  from Section~\ref{sec:curve-classes}
  that we have identified curve classes with tuples of rational numbers
\[
  \beta = (e_0 ,\ldots, e_m) \in \mathbb Q^{m+1},
\]
 such that
$e_i$ is the degree of the pullback of $\mathcal O_{\mathfrak Y}(0 ,\ldots, 1
,\ldots, 0)$, where the only $1$ is in the $i$-th position.
We cannot use this tuple to form the power series, since in general we only have
$e_0 \in \mathbb Q$.
Instead, let $w = \lcm(w_0 ,\ldots, w_n)$ and suppose $\phi_i$ in 
\eqref{eq:t-label} is supported on $X_{\alpha_i}$ for $i = 1 ,\ldots, m$.
In other words,
\[
  \alpha_1 ,\ldots, \alpha_m \in \mathbb Q \cap [0, 1)
\]
are the $\alpha$ we use to construct each row of the extended weight matrix
\eqref{eq:wmatrix} in Section~\ref{sec:extension-data}.
  We make the substitution
  \begin{equation}
    \label{eq:substitution-q}
    q^{\beta} \mapsto q_0^{d_0}\cdots q_m^{d_m}
  \end{equation}
  where
  \begin{equation}
    \label{eq:change-of-basis}
    \begin{cases}
      d_0 = w(e_0 + \alpha_1 e_1 + \cdots + \alpha_m e_m) \\
      d_i = e_i, \quad i = 1 ,\ldots, m.
    \end{cases}
  \end{equation}

  \begin{lemma} \label{lem:positive-basis}
    If $\beta$ is effective, then  $(d_0 ,\ldots, d_m) \in \mathbb Z_{\geq0}^{m+1}$.
  \end{lemma}
  \begin{proof}
    We may assume that $\beta$ is the curve class associated to a quasimap
    with irreducible domain curve $\mathcal C$.
    Since the generic point of $\mathcal C$ is mapped into the
    GIT stable locus, by Lemma~\ref{lem:git-locus}, we have $\beta \cdot \chi_i \geq 0$ for
    at least one $i = 0 ,\ldots, n$, and $e_j \geq 0$ for all $j = 1, \ldots, m$.
    Here, $\chi_i$ is the column vector from \eqref{eq:column}.
    This means that for some $i$ we have
    \[
      e_0 + \frac{\lfloor \alpha_1 w_i\rfloor - \delta_1}{w_i}e_1 + \cdots + \frac{\lfloor
        \alpha_m w_i\rfloor - \delta_m}{w_i} e_m \geq 0,
    \]
    where $\delta_1 ,\ldots, \delta_m \in \{0, 1\}$.
    Since $\alpha_j \geq \frac{\lfloor \alpha_j w_i\rfloor}{w_i}$ and $e_j \geq
    0$ for $j= 1 ,\ldots, m$,
    we have $d_0 \geq 0$.
    Thus $(d_0 ,\ldots, d_m) \in \mathbb Q_{\ge 0}$.

    Similarly, since all orbifold points are mapped into the GIT stable locus, it
    is easy to see that $d_0 ,\ldots, d_m\in \mathbb Z$. This completes
    the proof.
  \end{proof}

From now on we view the $I$-function as an element of
$\mathcal H [z^{-1}]\llbracket q_0 ,\ldots, q_m \rrbracket$ via the
substitution \eqref{eq:substitution-q}.
Using the basis $\phi_1 ,\ldots, \phi_m$, we 
rewrite the $J$-function as
\[
  J(t_1 ,\ldots, t_m, Q, z) = \mathbf 1 + \frac{\sum_{i=1}^m t_i \phi_i}{z}  +
  \sum_{\substack{
      d \geq 0, k \geq 0 \\
      (k, d)\neq (1,0),(0,0)
    }}
  \frac{Q^{d}}{k!} \sum_{p} T_p
  \big\langle
  {\frac{T^p}{z(z-\psi)},
    \sum_{i=1}^m t_i \phi_i,\ldots, \sum_{i=1}^m t_i \phi_i }
  \big\rangle^{X}_{0,1 + k, d/w}
\]
so that $J \in H^{*}_{\mathrm{CR}}(X)[z^{-1}]\llbracket Q, t_1 ,\ldots, t_m
\rrbracket$.
Here, the degree $d/w$ specifies the pairing of the curve class in $X$ with
$H = c_1(\mathcal O_{\mathbb P(\vec w)}(1))$. 
This series contains all the Gromov--Witten invariants with
insertions in $\mathcal H$.

For the wall-crossing, recall $\mu(q, z) = z(I_0(q) - 1) \cdot \mathbf 1 +  I_1(q) +
I_1^{\prime}(q)$, and write
\[
  I_1 = I_{1,H}  H, \quad  I_1^\prime = \sum_{i=1}^m I_{1, \phi_i} \phi_i, \quad
  \text{where }I_{1, H}, I_{1, \phi_i} \in \mathbb Q \llbracket  q_0 ,\ldots, q_m \rrbracket.
\]
Then Corollary~\ref{cor:wallc} becomes:
\begin{corollary}
  \label{cor:wallc-CY3}
  \begin{equation}
    \label{eq:wc-explicit}
        \exp{(-\frac{I_1}{zI_0})}
        \frac{I}{I_0} = J
  \end{equation}
  after the change of variables
  \begin{equation}
    \label{eq:change-of-coordinates}
    t_i = \frac{I_{1, \phi_i}}{I_0}, \quad
    Q = q_0 \exp( I_{1, H}/wI_0),
  \end{equation}
  where $w = \mathrm{lcm}(w_0 ,\ldots, w_n)$. 
\end{corollary}
\begin{proof}
  Simply notice that for a curve class $\gamma$ in $X$, its induced  quasimap
  curve class is $\beta = ( \int_\gamma H, 0, \dots, 0)$. And the substitution
  \eqref{eq:substitution-q} takes $q^{\beta}$ to $q^{w\int_\gamma H}_0$.
\end{proof}

To get all the Gromov--Witten invariants with insertions in $\mathcal H$, it
remains to show that the change of variables \eqref{eq:change-of-coordinates} is invertible.

  \begin{lemma} \label{lem:I-derivative} 
 For $1 \leq i \leq m$, we have 
 \[ 
   \frac{ \partial I}{\partial q_i}(q=0, z) = z^{-1}\phi_i.
 \]
  \end{lemma}
\begin{proof}

    This is the contribution from the curve class
    \[
      \beta = (-\alpha_i, 0 ,\ldots, 1 ,\ldots, 0),
    \]
    where the only $1$ appears in the $i$-th position.
    The result follows immediately from Theorem~\ref{thm:I-general}.
    We work out the more involved case when $\dim X_{\alpha_i} = 0$ and leave
    the other, simpler case to the reader.

    Suppose $\dim X_{\alpha_i} = 0$ and recall that
    $\phi_i = 1_{\overline{U_{\Lambda}}\cap X_{\alpha_i}} $ for 
    some $\Lambda \subset \{j \mid w_j \alpha_i \in \mathbb Z\}$, where
    $\overline{U_{\Lambda}} \subset \mathbb P_{\alpha_i}$ is defined by the common
    vanishing of $x_j$ for $j\in \Lambda$.
    Recall the vector notation $\chi_i$ and $\xi_j$ from \eqref{eq:column} and \eqref{eq:column-multi}, respectively.
    From the extended weight matrix
    \eqref{eq:wmatrix} we see that for 
  $k= 0 ,\ldots, n$, we have
    $\beta \cdot \chi_k = -1$ if $k \in \Lambda$ and $-1 < \beta \cdot \chi_k \leq 0$
      otherwise.
    Similarly, for the extended degrees $\xi_j$, we have 
    $\beta \cdot \xi_j = -1$ if $j \in \Gamma$ as in \eqref{eq:Gamma} and $-1 < \beta
    \cdot \xi_j \leq 0$ otherwise. Note that in this case
    by Lemma~\ref{lem:incidence} we have
    $\mathrm{vdeg}~F_{\beta} = \# \Lambda - \# \Gamma = 0$. Hence the
    contribution to the $I$-function is
    \[
      z^{-1}\frac{\prod_{\beta\cdot \chi_k \in \mathbb Z_{<0}} w_k}
      {\prod_{\beta \cdot \xi_j \in \mathbb Z_{<0}} b_j} \mathbf 1_{\beta}
      = z^{-1} \phi_i.
    \]
     \end{proof}
\begin{theorem}
  \label{thm:invertibility}
  The change of variables \eqref{eq:change-of-coordinates} is invertible and thus
  \eqref{eq:wc-explicit} determines all the genus-zero
   Gromov--Witten invariants
  of $X$ with insertions in the admissible state space $\mathcal H$.
\end{theorem}
\begin{proof}
  By Lemma~\ref{lem:I-derivative}, we have
  \[
    \mu(q_0=0, q_1, \dots, q_m, z) = \sum_{i=1}^m q_i \phi_i + \text{higher
      order terms in }q_1, \dotsc, q_m. 
  \]
  Thus, writing $\mathfrak m = (q_0 ,\ldots, q_m)$,  we have
  \[
    I_0  \equiv \mathbf 1  \mod \mathfrak m,
    \quad
    I_{1, H} \equiv 0 \mod \mathfrak m,
    \quad
    I_{1, \phi_i} \equiv q_i \mod \mathfrak m^2, i = 1 ,\ldots, m.
  \]
  It is easy to see that the Jacobian of \eqref{eq:change-of-coordinates} is the identity, and thus defines an isomorphism of formal power series rings.
\end{proof}

\section{Examples} \label{sec:examples} 

In this section, we compute the Gromov--Witten invariants of some Calabi--Yau
threefolds where convexity does not hold. In each case, we will compute some
explicit invariants,\footnote{These computations were completed in SageMath via a
program written by Yang Zhou.} and highlight points of interest, such as when
the invariants are actually enumerative.

Throughout this section, all our $I$-functions are treated as power series in
$q_0, \dots, q_m$, using the substitution \eqref{eq:substitution-q}, and all the $J$-functions are
power series in $Q, t_1 ,\ldots, t_m$. As before, we use $\mathbf
  1$ to denote the fundamental class of the untwisted sector and $H$ to denote
  the  hyperplane class $c_1(\mathcal O_{\mathbb P(\vec w)}(1))$ on the untwisted
sector. And we will always abbreviate $\mathbf 1_{\mathbf X_{\alpha}}$ as
$\mathbf 1_{\alpha}$. The product of two cohomology classes refers to the Chen--Ruan product.
Thus, for example, $H \cdot \mathbf 1_{\alpha}$ is the pullback of $H$ to
$X_{\alpha}$ via the natural projection $X_{\alpha} \to X$.

\subsection{$X_7$ in $\PP(1,1,1,1,3)$} \label{sec:x7}

Let $X = X_7$ be a smooth hypersurface of degree $7$ in $\PP(1,1,1,1,3)$.
Possible equations for $X_7$ include
 \[ x_0^7 + x_1^7 + x_2^7 + x_3^7 + x_3 x_4^2  = 0,
  \]
  or
  \[ x_0^7 + x_0 x_1^6 + x_1 x_2^6 + x_2 x_3^6 + x_3 x_4^2 =0. \]
The second equation is called an invertible polynomial of chain type
as considered in \cite{Gu19P}; the first equation is a sum of three
Fermat polynomials and one chain-type polynomial.
Note that any equation for $X_7$ has a term
$F_1(x_0, x_1, x_2, x_3) x_4^2$, where $F_1$ is a linear form.
Furthermore, any $X_7$ will pass through the single orbifold point
$B\mu_3 \subset \PP(1,1,1,1,3)$ at $x_0 = x_1 = x_2 = x_3 = 0$.
From the exact sequence
\begin{equation*}
  0 \to T_{X_7} \to T_{\PP(1,1,1,1,3)}|_{X_7} \to \cO(7)|_{X_7} \to 0
\end{equation*}
we see that $X_7$ is a
CY3 orbifold.

The inertia stack of $X_7$ is
\begin{equation*}
  \cI X_7 = X_7 \sqcup B\mu_3 \sqcup B\mu_3,
\end{equation*}
and the rigidified inertia stack \cite[Section~3]{AGV08} is
\begin{equation*}
  \ocI X_7 = X_7 \sqcup \pt \sqcup \pt,
\end{equation*}
where ``pt'' denotes a usual reduced scheme point.

  Let $\mathbf 1$ denote the fundamental class of the untwisted sector and
  $\mathbf 1_{\alpha}$ denote the fundamental class of the twisted sector
  $X_{\alpha}$, for $\alpha = 1/3$ or $2/3$.
  Then a basis of the ambient Chen--Ruan cohomology of $X_7$ is

\begin{equation*}
  \begin{tabular}[centering]{|c|c|c|c|c|c|c|}
    \hline class & $\mathbf 1$ & $H$ & $H^2$ & $H^3$ & $\one_{1/3}$ & $\one_{2/3}$ \\ \hline
    degree & $0$ & $2$ & $4$ & $6$ & $2$ & $4$ \\ \hline
  \end{tabular}
\end{equation*}
The admissible state space is equal to the ambient cohomology.

\subsubsection{Geometric Motivation for the GIT Extension}
\label{sec:motivation}

The motivation of the GIT extension of Section \ref{sec:extension-data} is to
capture the data of a quasimap from a smooth twisted curve with multiple
orbifold markings in terms of a quasimap from a curve with no markings. This is
done by replacing the orbifold markings with base points and using the extra
line bundles from the extended GIT to remember the forgotten orbifold data. We
illustrate this in an example of a quasimap to $X_7$.

Let $C$ be a smooth curve with a marking $p$, and let
$\cC$ be the $3$rd root stack of $C$ along $p$.
Let $q$ be the preimage of $p$ with reduced structure, which is canonically isomorphic to $B\mu_3$.
We view it as orbifold marking. Let $\cO_\cC(q)$ and $t\colon \cO_\cC \to \cO_\cC(q)$ denote
the $3$rd root
of $\cO_C(p)|_\cC$ and its tautological section given by the root stack construction.

Consider a map $f\colon \cC \to \PP(1,1,1,1,3)$ mapping $q$ to
$B\mu_3 \subset \PP(1, 1, 1, 1, 3)$, and such that the restriction
$B\mu_3 \to B\mu_3$ is the identity.
Note that this implies that $q$ maps to the twisted sector $X_{1/3}$
under the evaluation map.
Recall that such a map is equivalent to a choice of line bundle $\cL$ on $\cC$
together with a nowhere vanishing section
$s \in H^0(\cL^{\oplus 4} \oplus \cL^{\otimes 3})$.
The conditions on $f$ and \cite[Corollary 3.1.2]{Cad07}
 imply that any such $\cL$ can be written as
\[
\cL \cong L|_\cC \otimes \cO_\cC(q)
\]
for some line bundle $L$ on $C$, where we furthermore have an isomorphism
\begin{equation} \label{eq:sec-iso}
  H^0(\cC, \cL^{\oplus 4} \oplus \cL^{\otimes 3})
  \cong H^0(C, L^{\oplus 4} \oplus (L^{\otimes 3} \otimes \cO_C(p))).
\end{equation} 
(see \cite[Theorem 3.1.1]{Cad07}). With the choice of line bundles $L$ and $\cO_C(p)$, the image of $s$ under the above isomorphism and
the tautological section $\cO_C \to \cO_C(p)$ defines a map $C \to [V_e / T]$ where
$V_e = \CC^6$, $T = (\CC^*)^2$, and the $T$-action on $V_e$ is specified by
the weight matrix
\begin{equation}
\label{eq:x7-weight}
  \begin{pmatrix}
    1 & 1 & 1 & 1 & 3 & 0 \\
    0 & 0 & 0 & 0 & 1 & 1
  \end{pmatrix}.
\end{equation}
The choice of weights above reflects the isomorphism \eqref{eq:sec-iso}.
On the other hand, \eqref{eq:x7-weight} is also the weight matrix \eqref{eq:wmatrix} obtained from taking the set of admissible classes \eqref{eq:t-label}
to be the class
\[
  \phi_1 = \one_{1/3},
\]
where the latter condition is evident from our condition on the evaluation of $q$.

With respect to the character $\theta_e\colon T \to \CC^*$ given by
$\theta_e(\lambda, \mu) = \lambda \mu$, the set of semi-stable points
is
\begin{equation*}
  (V_e)^{\ss} = V_e \setminus (\{x_0 = \dotsb = x_4 = 0\} \cup \{x_5 = 0\}) \subset \{(x_0, \dotsc, x_5) \in \CC^6\} = V_e,
\end{equation*}
and hence the GIT quotient $[V_e \sslash_{\theta_e} T]$ recovers $\PP(1,1,1,1,3)$ (see Lemma \ref{lem:git-locus}).

Consequently, we have constructed a quasimap $g\colon C \to [V_e / T]$ from the data of the map $f$.

This new quasimap falls into $\PP(1,1,1,1,3)$ outside of $p$, and falls with contact order one
into the unstable locus
\begin{equation*}
  [(\{x_5 = 0\} \setminus (\{x_0 = \dotsb = x_4 = 0\}) / T]
\end{equation*}
at $p$, which is to say the underlying point of the orbifold marking is now a
base point. 
It is not difficult to see that given the quasimap $g$ with base point of order one at $p$, we can reverse the steps above to recover $f$.

We can do an analogous construction for the hypersurface
$X_7 \subset \PP(1,1,1,1,3)$.
The equation of $X_7$ is of the form
\begin{equation}
  \label{eq:x7-poly1}
  F_7(x_0, x_1, x_2, x_3) + x_4 F_4(x_0, x_1, x_2, x_3) + x_4^2 F_1(x_0, x_1, x_2, x_3),
\end{equation}
where $F_7$, $F_4$ and $F_1$ are homogeneous polynomials of degrees
$7$, $4$ and $1$, respectively, and $F_1$ is necessarily non-zero.
We may then write $X_7 = [W \sslash_\theta \CC^*]$, where $W$ is the
zero locus of \eqref{eq:x7-poly1}.
Under the weights of \eqref{eq:x7-weight}, the terms of \eqref{eq:x7-poly1} have degrees $(7,0)$, $(7,1)$ and $(7,2)$ respectively. We then homogenize the terms to obtain an equation of weight $(7,2)$
\begin{equation*}
  x_5^2 F_7(x_0, x_1, x_2, x_3) + x_4 x_5 F_4(x_0, x_1, x_2, x_3) + x_4^2 F_1(x_0, x_1, x_2, x_3),
\end{equation*}
and consider its zero locus $W_e \subset V_e = \mathbb C^6$.
Note that this is the multi-degree of Lemma \ref{lem:F-extension} specialized to
our example. One can then argue as above to show that there is a correspondence
between maps
$f: \cC \to X_7$ with marking evaluating to the twisted sector $(X_7)_{1/3}$,
and quasimaps $g: C \to (\fX_7)_e= [W_e/(\mathbb C^*)^2]$ that map $p$ into
part of the unstable locus.

By iterating the above process, we can capture the data of quasimaps from smooth curves with multiple stacky points in terms of quasimaps with none. In particular, the curves $\PP(1,r)$ in the $I$-function construction can remember the data of this ``extra" stacky structure, motivating why the resulting extended $I$-function can recover the Gromov--Witten invariants with multiple insertions in twisted sectors. 

\begin{remark}
  A similar argument motivates further extending the GIT presentation
  to incorporate the class $\mathbf 1_{2/3}$ by considering a
  hypersurface in the quotient $[\CC^7 / (\CC^*)^3]$ with weight
  matrix given by
  \begin{equation*}
    \begin{pmatrix}
      1 & 1 & 1 & 1 & 3 & 0 & 0 \\
      0 & 0 & 0 & 0 & 1 & 1 & 0 \\
      0 & 0 & 0 & 0 & 2 & 0 & 1
    \end{pmatrix},
  \end{equation*}
  but this is not necessary in our case since $\mathbf 1_{2/3}$ is of
Chen--Ruan
  degree $4 \neq 2$.
  More details can be found in the thesis \cite{sultanithesis} of the second author.
\end{remark}

\subsubsection{$I$-function}

Since the only degree two class of $X_7$ from a twisted sector is $\one_{1/3}$,
we see from the previous subsection that the extended GIT quotient will be $[W_e /
(\CC^*)^2] \subset [\CC^6 / (\CC^*)^2]$ with weight matrix \eqref{eq:x7-weight}
and defining equation extending \eqref{eq:x7-poly1}.

Using the notation in Section~\ref{sec:curve-classes}, the effective classes are
$\beta = (e_0, e_1) \in \mathbb Q \times \mathbb Z_{\geq 0}$ such that 
$3e_0 + e_1 \geq 0$.
 In particular, we note that $e_0$
can be negative, whereas the effective curve classes in the original target
space $[W/\mathbb C^*]$ must have non-negative degrees. On the other hand,
after the substitution \eqref{eq:substitution-q}, the extended $I$-function becomes a
power series in $q_0$ and $q_1$.

Applying Theorem~\ref{thm:I-general},
the extended $I$-function is given by
\begin{equation} \label{IX7}
I_{X_7} = \sum_{d_0, d_1\geq 0} \dfrac{q_0^{d_0} q_1^{d_1}}{(d_1!)z^{d_1}} \dfrac{
\prod\limits_{ \tiny \substack{\langle i \rangle = \langle \frac{d_0-d_1}{3} \rangle \\ i \leq 0}} (H+ iz)^4
 \prod\limits_{ \tiny \substack{\langle j \rangle = \langle \frac{d_0-d_1}{3} \rangle \\ j \leq 2d_0+ \frac{d_0-d_1}{3}}} (7H+jz)
 }
{
\prod\limits_{\tiny \substack{ \langle i \rangle = \langle \frac{d_0-d_1}{3} \rangle \\ i \leq\frac{d_0-d_1}{3} }} (H+ iz)^4
\prod\limits_{\tiny \substack{ \langle j \rangle = \langle \frac{d_0-d_1}{3} \rangle \\ j \leq 0 }} (7H+jz)
 \prod\limits_{k=1}^{d_0} (3H+kz)} \one_{ \langle \frac{d_1-d_0}{3} \rangle}.
\end{equation}

Here, we are following the ``formal division'' convention $H^a/H^b
  = H^{a-b}$ introduced below Theorem~\ref{thm:I-general}.

\subsubsection{Wall-Crossing and Invariants} 

  We write $t = t_1$, $\phi = \phi_1 = \mathbf 1_{1/3}$ and 
  write the $I$-function as 
  \[
  I(q,z) = I_0(q)\mathbf 1 + z^{-1}(I_{1, H}(q)H + I_{1, \phi}(q) \phi) + O(z^{-2}).
  \]
Then the wall-crossing formula applied to this example reads

\begin{equation}
\label{x7m}
 \frac{I}{I_0} \exp{(-\frac{I_{1,H}}{I_0}\frac{H}{z})} =
  \mathbf 1 + \frac{t}{z}\phi +
  \sum_{(d_0,d_1)\neq
    (0,1),(0,0)} \frac{Q^{d_0}t^{d_1}}{d_1!} \sum_{p}T_p \langle { \frac{T^p
    }{z(z-\psi)}, \phi ,\ldots, \phi}
  \rangle_{0,1 + d_1,d_0/3},
\end{equation}

under the mirror map

\begin{equation*}
  \begin{cases}
    Q = q_0 \exp (\frac{I_{1, H}}{3I_0})\\
    t = \frac{I_{1, \phi}}{I_0}
  \end{cases}.
\end{equation*}

We now proceed to compute the Gromov--Witten invariants from \eqref{x7m}.
Define

 \[
   N_{d_0, d_1} = \frac{1}{d_1!} \Big\langle {\phi ,\ldots,
     \phi}\Big \rangle_{0,d_1, d_0/3},
 \]
 and set
\[
\cF(Q, t) = \sum_{d_0>0}\sum_{d_1\geq 0} N_{d_0,d_1}Q^{d_0}t^{d_1} +
  \sum_{d_1 \geq 3} N_{0,d_1}t^{d_1}.
\]
The coefficient of $\frac{H^2}{z^2}$ of the right hand side of \eqref{x7m} may be computed as
\begin{equation*}
  \begin{aligned}
     \sum_{(d_0,d_1)\neq
      (0,1),(0,0)} \frac{Q^{d_0}t^{d_1}}{d_1!} \langle { \frac{3}{7} H,
    \phi ,\ldots, \phi}
    \rangle_{0,1 + d_1,d_0/3} 
    &=  \frac{1}{7}\sum_{d_0>0} \sum_{d_1\geq 0}d_0Q^{d_0}t^{d_1} N_{d_0,d_1}\\
    &=  \frac{1}{7} Q \frac{\partial \cF}{\partial Q}(Q,t). 
  \end{aligned}
\end{equation*}
where we use that $\int_{X_7}H^3 = \frac{7}{3}$, the divisor equation, 
and that $\langle H, \phi,\phi \rangle_{0, 3, 0} = 0$.
Similarly, by looking at the coefficient of
$\mathbf 1_{2/3}$, we get
\begin{equation*}
  \begin{aligned}
     \sum_{(d_0,d_1)\neq (0,1),(0,0)}
    \frac{Q^{d_0}t^{d_1}}{d_1!}
    \langle {3\phi, \phi ,\ldots, \phi}
  \rangle_{0,1 + d_1,{d_0}/3} 
  &= 3\sum_{(d_0,d_1)\neq
    (0,1),(0,0)}  (d_1+1)Q^{d_0}t^{d_1} N_{d_0,d_1+1}\\
  &=  3 \frac{\partial \cF}{\partial t}(Q,t). 
  \end{aligned}
\end{equation*}
We can then look at the corresponding terms of the left-hand side of \eqref{x7m}. From \eqref{IX7}, we have the following formulas
\begin{align*}
  I_0 =  & 1 + 2 q_{0} q_{1} + 840 q_{0}^{3} + 6 q_{0}^{2} q_{1}^{2} + 15120 q_{0}^{4}q_1  + O(q_0,q_1)^6
\end{align*}
\begin{align*}
I_{1, H} = & 15 q_{0} q_{1} + 7266 q_{0}^{3} + \frac{121}{2} q_{0}^{2} q_{1}^{2} +
          144438 q_{0}^{4} q_{1} + O(q_0,q_1)^6.
\end{align*}
\begin{align*}
 I_{1, \phi} = & q_{1} + \frac{385}{3} q_{0}^{2} + \frac{5}{9} q_{0} q_{1}^{2}
                + \frac{130900}{81} q_{0}^{3} q_{1} - \frac{1}{648} q_{1}^{4} + \frac{5084951872}{6075} q_{0}^{5}\\
  & + \frac{220}{243} q_{0}^{2} q_{1}^{3} + O(q_{0}, q_{1})^{6}
\end{align*}
so that the mirror map yields the change of variables
\begin{align*}
  \begin{cases}
    Q=  q_{0} + 5 q_{0}^{2} q_{1} + 2422 q_{0}^{4} + \frac{68}{3} q_{0}^{3}
    q_{1}^{2} + O(q_{0}, q_{1})^{6}\\
    t=  q_{1} + \frac{385}{3} q_{0}^{2}  - \frac{13}{9} q_{0} q_{1}^{2} +
\frac{42070}{81} q_{0}^{3} q_{1} - \frac{1}{648} q_{1}^{4} +
\frac{4430066872}{6075} q_{0}^{5} - \frac{536}{243} q_{0}^{2} q_{1}^{3} +
O(q_{0}, q_{1})^{6}.\\
  \end{cases}
\end{align*}

Inverting and applying this change of variables yields the following expansions
of the coefficients of $\frac{H^2}{z^2}$ and $\one_{2/3}$ on the left-hand side
of \eqref{x7m}:
\begin{align*}
\mathrm{Coeff}_{H^2/z^2} = &
                             4 Q t + \frac{19873}{3} Q^{3} - \frac{47}{9} Q^{2} t^{2} +
                             \frac{617288}{81} Q^{4}t \\
  & + \frac{1}{162} Q t^{4} + O(Q,
                             t)^{6},
\end{align*}
\begin{align*}
  \mathrm{Coeff}_{\mathbf 1_{2/3}} = &
                                  84 Q + \frac{1}{2} t^{2} - \frac{329}{3} Q^{2} t +
                                  \frac{1080254}{27} Q^{4} \\
                                &+ \frac{14}{27} Q t^{3} + \frac{3094}{27}
                                  Q^{3} t^{2} - \frac{1}{1080} t^{5} + O(Q, t)^{6}.
 \end{align*}
Comparing with the right-hand side of \eqref{x7m}, we obtain the partial
derivatives $\frac{\partial \cF}{\partial Q}$ and $\frac{\partial
\cF}{\partial t}$, which imply the expansion
\begin{equation}
  \label{eq:X7-invariants}
  \begin{aligned} 
   \cF(Q,t) = & 28 Q t + \frac{139111}{9} Q^{3} + \frac{1}{18}
    t^{3} - \frac{329}{18} Q^{2}t^{2} \\
    & + \frac{1080254}{81} Q^{4}
    t + \frac{7}{162} Q t^{4} + O(Q, t)^{6}.
  \end{aligned}
\end{equation}
For convenience, a table of some low degree invariants extracted from \eqref{eq:X7-invariants} is provided
below:
\begin{table}[h!]
  \begin{center}
    \begin{tabular}{|c|c|c|c|c|c|c|c|}
      \toprule
      \diagbox[innerwidth=1cm]{$d_0$}{$d_1$}
      & $ 0 $ & $ 1 $ & $ 2 $ & $ 3 $ & $ 4 $ & $ 5 $ & $ 6 $ \\ \midrule    $ 0 $ &   &   &   & $ \frac{1}{3} $ &   &   & $ -\frac{1}{27} $ \\ \hline
      $ 1 $ &   & $ 28$ &   &   & $ \frac{28}{27} $ &   &  \\ \hline
      $ 2 $ &   &   & $ -\frac{329}{9} $ &   &   & $ \frac{707}{243} $ &  \\ \hline
      $ 3 $ & $ \frac{139111}{9} $ &   &   & $ \frac{6188}{81} $ &   &   & $ \frac{10052}{243} $ \\ \hline
      $ 4 $ &   & $ \frac{1080254}{81} $ &   &   & $ \frac{534751}{4374} $ &   &  \\ \hline
      $ 5 $ &   &   & $ -\frac{726355322}{18225} $ &   &   & $ \frac{1672112666}{492075} $ &  \\ \hline
      $ 6 $ & $ \frac{1533417713597}{48600} $ &   &   & $ \frac{5386105627}{36450} $ &   &   & $ \frac{12986899639}{328050} $ \\ \bottomrule
    \end{tabular}
  \end{center}
\end{table}

\subsubsection{Invariants of $[\CC^3/\mu_3]$}

The above computation recovers the non-equivariant Gromov--Witten
theory of $[\CC^3/\mu_3]$, first computed by Coates--Corti--Iritani--Tseng in \cite{CCIT09}.
These are the degree $0$ invariants of $X_7$, which can be obtained
by setting $Q = 0$ in the generating series.
Indeed, when the degree is $0$ and there is at least one orbifold
marking, the map has to factor through the orbifold point $B \mu_{3}
\subset X_7$, whose formal neighborhood is the same as that of
$[\CC^3/\mu_3]$, where $\mu_3$ acts on $\mathbb C^3$ as scalar matrices. Hence
those invariants are the same as the invariants of $[\CC^3/\mu_3]$. 

Under the mirror map,  setting $Q = 0$ amounts
to setting $q_0 = 0$, obtaining 
\begin{align*}
  I_{X_7}(0, q_1)
  = &  \mathbf 1 + \sum_{ d_1 >  0} \dfrac{ q_1^{d_1}}{(d_1!)z^{d_1}}
      \dfrac{\prod\limits_{\tiny \substack{\langle i \rangle = \langle -d_1/3
      \rangle \\ -d_1/3 < i \leq 0}} (H+iz)^4 }
  {\prod\limits_{\tiny \substack{\langle i
  \rangle = \langle -d_1/3 \rangle \\ -d_1/3 < i \leq 0}} (7H+iz)} \one_{\langle
  \frac{d_1}{3} \rangle} \\
  = &  \mathbf 1 + \sum_{d_1 >  0, 3  \nmid d_1} \dfrac{ q_1^{d_1}}{(d_1!)z^{d_1}}
      \prod\limits_{\tiny \substack{\langle i \rangle = \langle -d_1/3
      \rangle \\ -d_1/3 < i < 0}} (iz)^3
  \one_{\langle
  \frac{d_1}{3} \rangle}
  +
  \sum_{d_1 >  0, 3  \mid d_1}     \frac{q_1^{d_1}}{(d_1!)z^{d_1}}
  \prod\limits_{\tiny \substack{\langle i \rangle = \langle -d_1/3
  \rangle \\ -d_1/3 < i < 0}} (iz)^3
  \frac{H^3}{7}.
\end{align*}
Note that $\frac{H^3}{7}$ is just the Poincar\'e dual of the orbifold point.
Hence if we
use $\one^\prime$, $\one^\prime_{1/3}$, and $\one^\prime_{2/3}$ to denote
the fundamental classes of the three irreducible components of $\cI B\mu_3$,
then $I_{X_7}(0, q_1) - \mathbf 1$ is equal to the pushforward of 
\begin{equation}
  \label{eq:C3Z3}
   \sum_{ d_1 > 0} \dfrac{ q_1^{d_1}}{(d_1!)z^{d_1}} \prod\limits_{\tiny
    \substack{\langle i \rangle = \langle -d_1/3 \rangle \\ -d_1/3 < i <0}} (iz)^3
  \one^\prime_{\langle \frac{d_1}{3} \rangle}
\end{equation}
along the obvious inclusion $\cI B \mu_3 \hookrightarrow \cI X_7$.
This agrees with the twisted $I$-function in
\cite[Section~6.3]{CCIT09} after setting their equivariant parameters
$\lambda_1, \lambda_2, \lambda_3$ to zero and setting $x_0 = x_2 = 0, x_1 = q_1$.
The additional parameters $x_0$ and $x_2$ in their case are due to extending by the additional classes $\one$ and $\one_{2/3}$.
Via a similar extension, we could recover these variables as well.

 Thus, we have that 
$\mu(q,z)|_{q_0 = 0}$, which here is the $z^{-1}$ coefficient of
\eqref{eq:C3Z3}, is given by 
\[  \sum_{k \geq 0} \frac{ (-1)^{3k}(q_1)^{3k+1}}{(3k+1)!} \left(
    \dfrac{ \Gamma(k+\frac{1}{3}) }{\Gamma(\frac{1}{3})} \right)^3 \mathbf 1_{1/3}
\]
where $\Gamma(i)$ is the Gamma function. This agrees with the mirror map term
defined as $\tau^1$ in \cite[Section~6.3]{CCIT09}, hence applying
\eqref{eq:wall-crossing} to
\eqref{eq:C3Z3} recovers the Coates--Corti--Iritani--Tseng characterization of the twisted $J$-function of $[\CC^3/\mu_3]$.
It is then straightforward to check by comparing the coefficients of
$\frac{\one_{2/3}}{z^2}$ on both the $I$ and $J$ sides that we recover
the non-equivariant limit of
\cite[Proposition~6.4]{CCIT09} as well, and that all the invariants agree.

\subsubsection{An enumerative invariant}
\label{sss:enumerative}

By \eqref{eq:X7-invariants}, the one-pointed degree $\frac 13$
invariant of $X_7$ with one insertion of $\phi = \mathbf 1_{1/3}$ is equal to
$28$.
It turns out that this invariant is enumerative in the sense that it
agrees with the number of orbifold lines $\PP(1,3)$ inside a generic
septic hypersurface in $\PP(1,1,1,1,3)$.

To see this, notice that $X_7 \subset \PP(1,1,1,1,3)$ is in general given
by an equation of the form
\begin{equation*}
  F_7(x_0, x_1, x_2, x_3) + x_4 \cdot F_4(x_0, x_1, x_2, x_3) + x_4^2 F_1(x_0, x_1, x_2, x_3) = 0,
\end{equation*}
where $F_1$, $F_4$ and $F_7$ are polynomials homogeneous of degrees
$1$, $4$ and $7$, respectively.
A line $\ell \cong \PP(1,3)$ on $X_7$ must pass through the unique
orbifold point, and is uniquely determined by its intersection
$p = (p_0, p_1, p_2, p_3, 0)$ with the hypersurface
$\{x_4 = 0\} \cong \PP^3$.
In order for $p$ to lie on $X_7$, we need that
$F_7(p_0, p_1, p_2, p_3) = 0$.
In addition, for all of $\ell$ to lie on $X_7$, we also need to
require that $F_4(p_0, p_1, p_2, p_3) = F_1(p_0, p_1, p_2, p_3) = 0$.
This defines a set of $7 \cdot 4 \cdot 1$ points on $\{x_4 = 0\}$, and
hence there are $28$ possible choices for $\ell$.

This appears to be the only enumerative or even integral invariant of
$X_7$.
In the case of Calabi--Yau $3$-manifolds, the BPS invariants, when defined to be certain linear
combinations of Gromov--Witten invariants, are known to be
integers, see \cite{IP18}.
\begin{question}
  Is there an analog of BPS invariants for Calabi--Yau $3$-orbifolds
  like $X_7$?
\end{question}
Some work has been done in this direction by Bryan and Pietromonaco (e.g.\
\cite{BP22}), including potential formulas for BPS for orbifolds. We hope that our methods can be used to test their conjectural formulas in some examples. 

\subsection{$X_{4, 4}$ in $\PP(1,1,1,1,1,3)$}

Now let $X = X_{4, 4}$ be a smooth complete intersection of two degree $4$
hypersurfaces in $\PP(1,1,1,1,1,3)$.
It is not straightforward to write down concrete equations for such a
smooth complete intersection, though by a Bertini argument or by using the numerical criterion in \cite{Ia00}, a generic
complete intersection of that type is smooth. 
By looking at the vanishing of generic degree 4 equations, we see that the complete intersection $X_{4, 4}$ must pass through the unique
orbifold point $B\mu_3$ of $\PP(1,1,1,1,1,3)$ at
$x_0 = x_1 = x_2 = x_3 = x_4  = 0$.
To ensure smoothness at the orbifold point, the equations for
$X_{4, 4}$ take the form
\begin{align*}
  F_{4a}(x_0, \dotsc, x_4) + x_5 F_{1a}(x_0, \dotsc, x_4) &= 0 \\
  F_{4b}(x_0, \dotsc, x_4) + x_5 F_{1b}(x_0, \dotsc, x_4) &= 0
\end{align*}
for quartic polynomials $F_{4a}$ and $F_{4b}$, and for linear forms
$F_{1a}$ and $F_{1b}$ whose matrix of coefficients has full rank $2$.
By the adjunction sequence, we see that $X_{4, 4}$ is another example
of a CY3-orbifold.

The inertia stack of $X_{4, 4}$ is
\begin{equation*}
  \cI X_{4, 4} = X_{4, 4} \sqcup B\mu_3 \sqcup B\mu_3,
\end{equation*}
and the rigidified inertia stack is
\begin{equation*}
  \ocI X_{4, 4} = X_{4, 4} \sqcup \pt \sqcup \pt.
\end{equation*}

As before, we denote the fundamental class of the untwisted sector by $\mathbf 1$ and
denote the fundamental class of the twisted sector $X_{\alpha}$ by
$\mathbf
  1_{\alpha}$. The following is a homogeneous basis for the ambient Chen--Ruan
cohomology of $X_{4, 4}$, which coincides with the admissible
  state space:
\begin{equation*}
  \begin{tabular}[centering]{|c|c|c|c|c|c|c|}
    \hline class & $\mathbf 1$ & $H$ & $H^2$ & $H^3$ & $\one_{1/3}$ & $\one_{2/3}$ \\ \hline
    degree & $0$ & $2$ & $4$ & $6$ & $2$ & $4$ \\ \hline
  \end{tabular}
\end{equation*}

\subsubsection{$I$-function}
Like the case of $X_7$, the only  age $1$ class is $\one_{1/3}$. We extend the GIT presentation in a similar manner to obtain $[\CC^7/(\CC^*)^2]$, with weight matrix given by
\[
\begin{pmatrix}
1 &1 &1 & 1& 1 &3 &  0\\
0 &0 &0 &0  & 0 & 1& 1
\end{pmatrix}.
\]
The equations of the complete intersection are extended to two
polynomials of weight $(4, 1)$.

The extended $I$-function is given by
\begin{equation*}
  I_{X_{4,4}} = \sum_{d_0, d_1 \geq 0} \dfrac{q_0^{d_0}q_1^{d_1}}{z^{d_1} d_1!}
  \dfrac{ \prod\limits_{\tiny \substack{ \langle i \rangle = \langle \frac{d_0-d_1}{3}
        \rangle \\ i \leq \frac{4d_0-d_1}{3}}} (4H+iz)^2 \prod\limits_{\tiny \substack{\langle j
        \rangle = \langle \frac{d_0-d_1}{3} \rangle \\ j \leq 0}} (H +jz)^5 }{
    \prod\limits_{\tiny \substack{ \langle i \rangle = \langle \frac{d_0-d_1}{3} \rangle \\ i
        \leq 0}} (4H+iz)^2 \prod\limits_{\tiny \substack{\langle j \rangle = \langle
        \frac{d_0-d_1}{3} \rangle \\ j \leq \frac{d_0-d_1}{3}}} (H +jz)^5 \prod_{k=1}^{d_0} (3H+kz) }
  \cdot \one_{ \langle -\frac{d_0-d_1}{3} \rangle}.
\end{equation*}

\subsubsection{Wall-Crossing and Invariants}
The mirror formula \eqref{eq:wc-explicit} reads exactly the same as in the
$X_7$ case, with the same mirror map as before. The only change here is the
$I$-function, leading to different invariants. Since the computation is
otherwise the same, we just present a partial formula for $\cF(Q, t)$ 
\begin{align*} 
\cF(Q,t) =& 16 Q t + \frac{20800}{9} Q^{3} + \frac{1}{18} t^{3} - \frac{46}{9} Q^{2}
t^{2} + \frac{46490}{81} Q^{4} t + \frac{2}{81} Q t^{4} +
  \frac{2329313056}{6075} Q^{6} \\
  & + \frac{304}{243} Q^{3} t^{3} - \frac{1}{19440}
t^{6} - \frac{9256192}{18225} Q^{5} t^{2} + \frac{77}{7290} Q^{2} t^{5} +
    \frac{1704994246016}{8037225} Q^{7} t \\
  & + \frac{1391}{13122} Q^{4} t^{4} - \frac{29}{229635} Q t^{7} + \frac{1690784332712}{10935} Q^{9} +
    \frac{17945392}{54675} Q^{6} t^{3} \\
  & + \frac{122}{10935} Q^{3} t^{6} +
  \frac{1}{3265920} t^{9} + O(Q, t)^{10}.
\end{align*}
Note that when setting $Q= 0$, we get
\begin{align*} 
 \cF(0,t) = & \frac{1}{18} t^{3} - \frac{1}{19440} t^{6} + \frac{1}{3265920} t^{9}
             - \frac{1093}{349192166400} t^{12} + \frac{119401}{2859883842816000} t^{15} \\
  & - \frac{27428707}{42005973883281408000} t^{18} + O(t^{20}),
\end{align*}
which matches the generating series of $X_7$ with $Q = 0$.
This is expected since the normal bundle of the $B \mu_3$ point in $X_{4,4}$ is isomorphic to $[\CC^3/\mu_3]$, hence the degree $0$
invariants recover the invariants of $[\CC^3/\mu_3]$.

We provide some low-degree invariants in a table:

\begin{table}[h!]
  \begin{center}
    \begin{tabular}{|c|c|c|c|c|c|c|c|}
      \toprule
      \diagbox[innerwidth=1cm]{$d_0$}{$d_1$}
      & $ 0 $ & $ 1 $ & $ 2 $ & $ 3 $ & $ 4 $ & $ 5 $ & $ 6 $ \\ \midrule    $ 0 $ &   &   &   & $ \frac{1}{3} $ &   &   & $ -\frac{1}{27} $ \\ \hline
      $ 1 $ &   & $  16 $ &   &   & $ \frac{16}{27} $ &   &  \\ \hline
      $ 2 $ &   &   & $ -\frac{92}{9} $ &   &   & $ \frac{308}{243} $ &  \\ \hline
      $ 3 $ & $ \frac{20800}{9} $ &   &   & $ \frac{608}{81} $ &   &   & $ \frac{1952}{243} $ \\ \hline
      $ 4 $ &   & $ \frac{46490}{81} $ &   &   & $ \frac{5564}{2187} $ &   &  \\ \hline
      $ 5 $ &   &   & $ -\frac{18512384}{18225} $ &   &   & $ \frac{19492352}{492075} $ &  \\ \hline
      $ 6 $ & $ \frac{2329313056}{6075} $ &   &   & $ \frac{35890784}{18225} $ &   &   & $\frac{83823488}{164025} $ \\ \bottomrule
    \end{tabular}
  \end{center}
\end{table}

\subsubsection{An Enumerative Invariant}
We also have one enumerative invariant, $\langle \one_{1/3} \rangle_{0, 1, 1/3} = 16$, which enumerates
the number of orbifold lines $\PP(1,3)$ in a generic $X_{4, 4}$.
This can be derived in the same way as in the $X_7$ case, where $16$
arises as $(4 \cdot 1)^2$.

\subsection{$X_{17}$ in $\PP(2,2,3,3,7)$} 

Let $X = X_{17}$ be a smooth hypersurface of degree $17$ in $\PP(2,2,3,3,7)$.
A possible equation for $X_{17}$ is
\begin{equation*}
  x_0 x_2^5 + x_2 x_4^2 + x_4 x_1^5 + x_1 x_3^5 + x_3 x_0^7.
\end{equation*}
This is an equation of loop-type.
In fact, there are no equations for $X_{17}$ of Fermat- or chain-type
(or combinations thereof) since $17$ is divisible by neither $2$, $3$
nor $7$.
\footnote{Similar examples are hypersurfaces of degree $17$ in
  $\PP(2,2,3,5,5)$.}
From the adjunction sequence, we again observe that $X_{17}$ is a
CY3-orbifold.
The stacky loci of the ambient $\PP(2,2,3,3,7)$ are a $B\mu_7$-point at
$x_0 = x_1 = x_2 = x_3 = 0$, a gerby projective line $\PP(2,2)$ at
$x_2 = x_3 = x_4 = 0$, and a gerby projective line $\PP(3,3)$ at
$x_0 = x_1 = x_4 = 0$.
We may check that any $X_{17}$ must pass through all of these stacky
loci, and in order to be nonsingular at each of these, the equation of
$X_{17}$ must contain a term linear in $x_2, x_3$ and quadratic in
$x_4$, a term quintic in $x_0, x_1$ and linear in $x_4$, and a term
linear in $x_0, x_1$ and quintic in $x_2, x_3$.

The inertia stack of $X_{17}$ is
\begin{equation*}
  \cI X_{17} = X_{17} \sqcup \bigsqcup_{i = 1}^6 B\mu_7 \sqcup \PP(2,2) \sqcup \bigsqcup_{i = 1}^2 \PP(3,3),
\end{equation*}
and the rigidified inertia stack is
\begin{equation*}
  \ocI X_{17} = X_{17} \sqcup \bigsqcup_{i = 1}^6 \pt \sqcup \PP^1 \sqcup \bigsqcup_{i = 1}^2 \PP^1.
\end{equation*}

Using similar notation as before, the ambient Chen--Ruan cohomology of $X_{17}$ has homogeneous basis:
\begin{equation*}
  \begin{tabular}[centering]{|c|c|c|c|c|c|c|c|c|c|}
    \hline class & $\mathbf 1$ & $H$ & $H^2$ & $H^3$
    & $\one_{1/7}$ & $\one_{2/7}$ & $\one_{3/7}$ & $\one_{4/7}$ & $\one_{5/7}$ \\
    \hline
    degree & $0$ & $2$ & $4$ & $6$ & $2$ & $4$ & $4$ & $2$ & $2$ \\ \hline
  \end{tabular}
\end{equation*}
\begin{equation*}
  \begin{tabular}[centering]{|c|c|c|c|c|c|c|c|}
    \hline class & $\one_{6/7}$ & $\one_{1/2}$ & $H \cdot \one_{1/2}$ & $\one_{1/3}$ & $H \cdot \one_{1/3}$ & $\one_{2/3}$ & $H \cdot \one_{2/3}$ \\ \hline
    degree & $4$ & $2$ & $4$ & $2$ & $4$ & $2$ & $4$ \\ \hline
  \end{tabular}
\end{equation*}

\subsubsection{$I$-function} 
We extend the GIT presentation so that the ambient quotient stack is of the form $[\CC^{11}/(\CC^*)^7]$
with weight matrix given by
\setcounter{MaxMatrixCols}{12}
\begin{equation*}
  \begin{pmatrix}
    2 & 2 & 3 & 3 & 7 & 0 & 0 & 0 & 0 & 0 & 0 \\
    0 & 0 & 0 & 0 & 1 & 1 & 0 & 0 & 0 & 0 & 0 \\
    1 & 1 & 1 & 1 & 4 & 0 & 1 & 0 & 0 & 0 & 0 \\
    1 & 1 & 2 & 2 & 5 & 0 & 0 & 1 & 0 & 0 & 0 \\
    1 & 1 & 1 & 1 & 3 & 0 & 0 & 0 & 1 & 0 & 0 \\
    0 & 0 & 1 & 1 & 2 & 0 & 0 & 0 & 0 & 1 & 0 \\
    1 & 1 & 2 & 2 & 4 & 0 & 0 & 0 & 0 & 0 & 1
  \end{pmatrix}
\end{equation*}
The $6$ extra rows correspond (in order) to the degree $2$ classes $\mathbf 1_{1/7}$,
$\mathbf 1_{4/7}$, $\mathbf 1_{5/7}$, $\mathbf 1_{1/2}$,
$\mathbf 1_{1/3}$, $\mathbf 1_{2/3}$, with weights prescribed as in \eqref{eq:weight-entry}.
As in Lemma \ref{lem:F-extension}, we can extend the equation for $X_{17}$ to an equation of
multi-degree $(17, 2, 9, 12, 8, 5, 11)$.

  For $(d_0, d_1, \dots, d_6) \in \mathbb Z_{\geq 0}^7$, we set
\begin{itemize}
\item
  $\rho_1 = \frac{d_0 - 6d_1 - 3d_2 - 9d_3 - 14d_5 - 7d_6}{21}$
\item
  $\rho_2 = \frac{ d_0 - 6d_1 - 10d_2 - 2d_3 - 7d_4}{14}$
\item
  $\rho_3 = \frac{d_0 - 3d_4 - 2d_5 - 4d_6}{6}$
\item
  $\rho_4 = \frac{17 d_0}{42} - \frac{3d_1}{7} - \frac{5d_2}{7} - \frac{d_3}{7}
  - \frac{d_4}{2} - \frac{2d_5}{3} - \frac{d_6}{3}$
\item
  $\sigma = \frac{d_0 - 6d_1 - 24d_2 - 30d_3 - 21d_4 - 14d_5 -28d_6}{42}$
\end{itemize}

The extended $I$-function is then given by
\begin{equation*}
  I_{X_{17}} = \sum_{d_0 ,\ldots, d_6 \geq 0} \dfrac{ \prod_{i=0}^6 q_i^{d_i}}{\prod_{i=1}^6 (d_i)!z^{d_i}}
  \dfrac{\prod\limits_{\tiny \substack{\langle j \rangle = \langle \rho_1 \rangle \\ j \leq 0}} (2H+jz)^2
    \prod\limits_{\tiny \substack{\langle j \rangle = \langle \rho_2 \rangle \\ j \leq 0}} (3H+jz)^2
    \prod\limits_{\tiny \substack{\langle j \rangle = \langle \rho_3 \rangle \\ j \leq 0}} (7H+jz)
    \prod\limits_{\tiny \substack{\langle j \rangle = \langle \rho_4 \rangle \\ j \leq \rho_4}} (17H+jz)}
  {
    \prod\limits_{\tiny\substack{\langle j \rangle = \langle \rho_1 \rangle \\ j \leq \rho_1}} (2H+jz)^2
    \prod\limits_{\tiny\substack{\langle j \rangle = \langle \rho_2 \rangle \\ j \leq \rho_2}} (3H+jz)^2
    \prod\limits_{\tiny\substack{\langle j \rangle = \langle \rho_3 \rangle \\ j \leq \rho_3}} (7H+jz)
    \prod\limits_{\tiny\substack{\langle j \rangle = \langle \rho_4 \rangle \\ j \leq 0}} (17H+jz)
  }
  \cdot
  \one_{\langle -\sigma \rangle}.
\end{equation*}

  Note that unless $\sigma \in \{\frac{n}{2}, \frac{n}{3}, \frac{n}{7}\}_{n \in
    \mathbb Z}$, the summand
  is understood to be zero.

\subsubsection{Wall-Crossing and Invariants}
We have that \eqref{eq:wc-explicit} in this case looks like

\begin{equation}
  \label{x17m}
  \begin{aligned}
    \frac{I}{I_0} \exp (-\frac{H}{z} \frac{I_{1, H}}{I_0} \big)
    =
    \mathbf 1 +\frac{\mathbf t}{z}  +
    \sum_{(r,d)\neq (1,0),(0,0)}
    \frac{Q^{d}}{r!} \sum_{p} T_p
    \left\langle {
      \frac{T^p }{z(z-\psi)},
    \mathbf t  ,\ldots, \mathbf t
    }
    \right\rangle_{0,1 + r,d/42}
  \end{aligned}
\end{equation}
where we set
\[
  (\phi_1, \dots, \phi_6) = (\one_{1/7}, \one_{4/7}, \one_{5/7}, \one_{1/2},
  \one_{1/3}, \one_{2/3}) \quad \text{and} \quad \mathbf t = \sum_{i = 1}^6 t_i \phi_i.
\]
The mirror map is given by
\[
  Q = q_0 \exp( I_{1, H}/42I_0), \quad t_i = \frac{I_{1, \phi_i}}{I_0}.
\]
with notation as defined in \eqref{eq:change-of-coordinates}. We can define the generating function
  \begin{align*} 
    \cF(Q,t_1 ,\ldots, t_6) = & \sum_{d > 0} \sum_{r = 0}^\infty
                              \frac{Q^{d}}{r!}
                              \langle {
                              \mathbf t ,\ldots, \mathbf t
                              }
                              \rangle_{0, r,d/42} + 
    \sum_{r = 3}^\infty \frac{1}{r!}
        \langle {
        \mathbf t ,\ldots, \mathbf t
        }
        \rangle_{0, r,   0},
  \end{align*}
 and we can again look at coefficients of cohomology classes on both sides to
obtain explicit formulas for the partial derivatives of $F$. The following table
lists the coefficients for each cohomology class on the right hand side of
\eqref{x17m}
\begin{center}
  \begin{tabular}{c|c} 
    base element & coefficient \\
    \hline
    $H^2$ & $\frac{6}{17} Q \frac{\partial}{\partial Q}\cF
            + \frac{63}{34} t_4^2 + \frac{28}{17} t_5t_6$ \\
    $H_{1/2}$ & $2 \frac{\partial}{\partial t_4} \cF$ \\
    $H_{1/3}$ & $3 \frac{\partial}{\partial t_6} \cF$ \\
    $H_{2/3}$ & $3 \frac{\partial}{\partial t_5} \cF$ \\
    $\one_{2/7}$ & $7 \frac{\partial}{\partial t_3} \cF$ \\
    $\one_{3/7}$ & $7 \frac{\partial}{\partial t_2} \cF$ \\
    $\one_{6/7}$ & $7 \frac{\partial}{\partial t_1} \cF$  \\
    $H^3$ & $\frac{252}{17} (\sum_i t_i \frac{\partial}{\partial t_i} - 2) \cF$
  \end{tabular}
\end{center}
Comparing with the left-hand side of \eqref{x17m}, we obtain the following explicit computation of $\cF$
\begin{align*} 
   & \cF(Q,t_1 ,\ldots, t_6) \\
  = & \frac{1}{14} t_{1}^{2} t_{3} + \frac{1}{14} t_{2}
                            t_{3}^{2} - \frac{13}{54} t_{5}^{3} + \frac{7}{54}
                            t_{6}^{3} + Q^{2} t_{3} t_{5}
                            - \frac{1}{147} t_{1}^{3} t_{2} - \frac{1}{98} t_{1}
                            t_{2}^{2} t_{3} +
                            \frac{1}{48} t_{4}^{4} \\
                          & + \frac{1}{18} t_{5}^{2} t_{6}^{2} + 5 Q^{3} t_{2}
                            t_{4} + \frac{1}{7} Q^{2} t_{1} t_{2} t_{5} +
                            \frac{1}{6} Q^{2} t_{3} t_{6}^{2} + \frac{1}{686}
                            t_{1}^{2} t_{2}^{3} + \frac{3}{2744} t_{1} t_{3}^{4} \\
                          & + \frac{1}{8232} t_{2}^{4} t_{3} - \frac{1}{324}
                            t_{5}^{4} t_{6} - \frac{1}{324} t_{5} t_{6}^{4} +
                            \frac{1}{42} Q^{2} t_{1} t_{2} t_{6}^{2} +
                            \frac{1}{294} Q^{2} t_{2}^{3} t_{5} \\
                          & - \frac{1}{18} Q^{2} t_{3} t_{5}^{2} t_{6} -
                            \frac{43}{115248} t_{1}^{4} t_{3}^{2} -
                            \frac{31}{28812} t_{1}^{2} t_{2} t_{3}^{3} -
                            \frac{11}{144060} t_{1} t_{2}^{5} - \frac{5}{28812}
                            t_{2}^{2} t_{3}^{4}\\
                          & + \frac{1}{2880} t_{4}^{6} + \frac{1}{9720}
                            t_{5}^{6} + \frac{1}{486} t_{5}^{3} t_{6}^{3} +
                            \frac{1}{9720} t_{6}^{6} + \frac{85}{6} Q^{6} t_{1}
                            - \frac{1}{4} Q^{4} t_{3}^{2} t_{6} - \frac{5}{98}
                            Q^{3} t_{1} t_{3}^{2} t_{4}\\
                          & - \frac{5}{24} Q^{3} t_{2} t_{4}^{3} + \frac{1}{196}
                            Q^{2} t_{1}^{2} t_{3}^{2} t_{5} - \frac{1}{126}
                            Q^{2} t_{1} t_{2} t_{5}^{2} t_{6} + \frac{1}{1764}
                            Q^{2} t_{2}^{3} t_{6}^{2} + \frac{5}{2058} Q^{2} t_{2}
                            t_{3}^{3} t_{5} \\
                          & + \frac{1}{648} Q^{2} t_{3} t_{5}^{4} - \frac{1}{162} Q^{2}
                            t_{3} t_{5} t_{6}^{3} + \frac{37}{6050520} t_{1}^{7}
                            + \frac{311}{2016840} t_{1}^{5} t_{2} t_{3} +
                            \frac{69}{134456} t_{1}^{3} t_{2}^{2} t_{3}^{2} \\
                          & + \frac{3}{16807} t_{1} t_{2}^{3} t_{3}^{3}
                           + \frac{11}{6050520} t_{2}^{7} + \frac{2}{252105}
                            t_{3}^{7} - \frac{1}{3240} t_{5}^{5} t_{6}^{2} -
                            \frac{1}{3240} t_{5}^{2} t_{6}^{5}\\
                          &+ O(Q, t_{1}, t_{2}, t_{3}, t_{4}, t_{5}, t_{6})^{8}.
\end{align*}

It would be interesting to compute the first invariant without markings, the coefficient of $Q^{42}$, but it takes too long for the program to complete this computation.

\subsubsection{Enumerative Invariants}
Similar to the discussion in Section~\ref{sss:enumerative}, the
integer invariants $\langle \phi_3, \phi_5\rangle_{0, 2, 1/21} = 1$
and $\langle \phi_2, \phi_4\rangle_{0, 2, 1/14} = 5$ are enumerative.

More explicitly, the first invariant is the number of weighted
projective lines $\PP(3,7)$ joining the gerby line $\PP(3,3)$ and the
$B\mu_7$-point.
Such lines are contained in the weighted projective plane $\PP(3,3,7)$
defined via $x_0 = x_1 = 0$.
In this plane, the equation for $X_{17}$ becomes of the form
$x_4^2 \cdot F_1(x_2, x_3)$, where $F_1$ is a linear form.
Hence, the unique line lying inside $X_{17}$ is the one passing
through the point in $\PP(3,3)$ where $F_1$ vanishes.

The second invariant is the number of weighted projective lines
$\PP(2,7)$ between the gerby line $\PP(2,2)$ and the $B\mu_7$-point.
In this case, the number $5$ appears because in the weighted projective
plane $\PP(2,2,7)$ given by $x_2 = x_3 = 0$, the equation for $X_{17}$
becomes of the form $x_4 \cdot F_5(x_0, x_1)$, where $F_5$ is a
homogeneous quintic polynomial.

\subsection{$X_{24}$ in $\PP(1,4,4,6,9)$} \label{sec:x24} 

We now discuss an example in which the admissible state space $\mathcal H$ is strictly larger than
the ambient cohomology.

Let $X = X_{24}$ be a smooth degree $24$ hypersurface in $\PP(1,4,4,6,9)$. A possible equation for $X_{24}$ is 
\[ x_0^{24} + x_1^6 + x_2^6 + x_3^4 + x_3x_4^2. \]
It is again easy to see that $X_{24}$ is a CY3-orbifold.
We may check directly that any hypersurface of this form must contain the stacky
$B\mu_9$ point at $x_0=x_1=x_2=x_3=0$, but unlike the case for $X_{17}$, it need
not contain the entirety of the weighted projective lines $\PP(4,4)$ or $\PP(6,9)$.
For a generic defining equation, the inertia stack has twisted sectors of the
form
\begin{align*}
  X_{i/4} &\cong \textstyle\bigsqcup_{j=1}^6B\mu_4\quad & \text{for $i = {1, 3}$}
  \\
  X_{i/3}  &\cong B\mu_3 \sqcup B\mu_9 \quad &\text{for $i = {1,2}$} \\
  X_{i/9} &\cong B\mu_9 \quad &\text{ for $i = {1, 2, 4, 5, 7, 8}$}
\end{align*}
and $X_{1/2} \subset \mathbb P(4, 4, 6)$ of degree $24$.
Hence the rigidified inertia stack is
\[
  \overline{\cI} X_{24} = X_{24} \sqcup
  \bigsqcup_{j=1}^2 B\mu_3 \sqcup \bigsqcup_{k =1}^{20} \mathrm{pt}
  \sqcup C_{1/2},
\]
where $C_{1/2}$ is a curve which has the same coarse moduli as $X_{1/2}$.

  In this example we may have disconnected $X_{\alpha}$. Indeed, for $\alpha =
  1/3$ or $2/3$, $\mathbb P_{\alpha} = \mathbb P(6, 9)$ and $X_{\alpha}\subset
  \mathbb P_\alpha$ consists of exactly two points, one at $\mathbb P(9)$ and
  the other at a generic position in $\mathbb P(6, 9)$. In the terminology of
  Section~\ref{sec:extendable-classes}, both $\mathbb P(9)$ and $X_{\alpha}$ are
  special, hence the admissible state space is strictly larger than the ambient cohomology.

The admissible state space in this example has the homogeneous basis
\begin{equation*}
  \begin{tabular}[centering]{|c|c|c|c|c|c|c|c|c|c|}
    \hline class & $\one$ & $H$ & $H^2$ & $H^3$ & $\one_{1/9}$ & $\one_{2/9}$ & $\one_{4/9}$ & $\one_{5/9}$ & $\one_{7/9}$ \\ \hline
    degree & $0$ & $2$ & $4$ & $6$ & $2$ & $4$ & $4$ & $2$ & $2$ \\ \hline
  \end{tabular}
\end{equation*}
\begin{equation*}
  \begin{tabular}[centering]{|c|c|c|c|c|c|c|c|c|c|}
    \hline class & $\one_{8/9}$ & $\one_{1/4}$ & $\one_{1/2}$ & $H\cdot \one_{1/2}$ & $\one_{3/4}$
    & $\one_{1/3}$ & $\one_{2/3}$ & $\gamma_{1/3}$ & $\gamma_{2/3}$ \\ \hline
    degree & $4$ & $2$ & $2$  & $4$ & $4$ & $2$ & $4$ & $2$ & $4$\\ \hline
  \end{tabular}
\end{equation*}
  where 
  $\gamma_{\alpha}$ is the class of the $\mathbb P(9)$ in $X_{\alpha}$.
  
  \begin{remark} \label{rem: x24}   By picking specific equations for our hypersurface, we could possibly have more special substacks in our twisted sectors than in the generic case. As an example, consider the following possible equation for $X_{24}$
  \[ x_0^{24}  + x_1x_2^5 + x_2x_1^5 + x_3^4 + x_3x_4^2 .\]
  Using this equation, two of the $B\mu_4$ points in the twisted sectors $X_{1/4}$ will be at coordinates $[0:1]$ and $[1:0]$, hence can be considered special as in \eqref{def:special}. By applying a similar extension, we can obtain 
  invariants with these classes as insertions as well. 
  
  \end{remark}

\subsubsection{Extended GIT}
We will extend the GIT presentation with respect to the set of degree $2$ admissible classes
\begin{equation} \label{eq:x24-extset}
    (\phi_1 ,\ldots, \phi_7) =
 (\one_{1/4}, \,\one_{1/2}, \,\one_{1/9},\, \one_{1/3},\, \one_{5/9},\, \one_{7/9},\, \gamma_{1/3}),
  \end{equation}
but first we digress to explain why requiring our set \eqref{eq:t-label} to be a
basis of degree $2$ admissible classes is necessary for the invertibility
statement of Theorem \ref{thm:invertibility}.
More explicitly, we will show why it is necessary to include the non-ambient class $\gamma_{1/3}$ in the GIT extension even if one is only
interested in ambient classes.

Suppose we extend by only the ambient classes, omitting $\gamma_{1/3}$ from the list above. The GIT extension in this case will be given by $[\CC^{11}/(\CC^*)^7]$ with weight matrix  
\begin{equation*}
  \begin{pmatrix}
    1 & 4 & 4 & 6 & 9 & 0 & 0 & 0 & 0 & 0 & 0 \\
    0 & 1 & 1 & 1 & 2 & 1 & 0 & 0 & 0 & 0 & 0 \\
    0 & 2 & 2 & 3 & 4 & 0 & 1 & 0 & 0 & 0 & 0 \\
    0 & 0 & 0 & 0 & 1 & 0 & 0 & 1 & 0 & 0 & 0 \\
    0 & 1 & 1 & 2 & 3 & 0 & 0 & 0 & 1 & 0 & 0 \\
    0 & 2 & 2 & 3 & 5 & 0 & 0 & 0 & 0 & 1 & 0 \\
    0 & 3 & 3 & 4 & 7 & 0 & 0 & 0 & 0 & 0 & 1
  \end{pmatrix}
\end{equation*}
with the ordering of the rows corresponding to the classes $(\phi_1 ,\ldots, \phi_6)$. The equation defining $X_{24}$ can be extended
to be of degree $(24,6,12,2,8,13,18)$.

Now consider $\beta = (-1/3, 0, 0, 3, 0, 0, 0)$ using the curve class notation in
Section~\ref{sec:curve-classes}. It is straightforward to check that this
class is effective. However, looking at the pairings between $\beta$ and the columns of
the weight matrix, we have $\beta\cdot (6, 1, 3, 0, 2, 3, 4) = -2$, hence by
Proposition~\ref{propgamma}
 the corresponding fixed locus $F_{\beta}$ in the graph space 
  is identified with the $\mathbb P(9)$ in $X_{1/3}$, and
\[ [F_\beta]^{\vir} = [F_{\beta}] \]
 is Poincar\'e dual to $\gamma_{1/3}$. From Theorem~\ref{thm:I-general},
 we see that the $I$-function will contain the term $z^{-1}\gamma_{1/3}$. This means that the class $\gamma_{1/3}$ will appear in the
series $\mu(q,z)$, but since we have no associated Novikov variable, the mirror map
cannot be invertible for dimension
reasons.

Consequently, in order to obtain the values of the individual Gromov--Witten
invariants, we should include the non-ambient class $\gamma_{1/3}$ in our
extension. The resulting quotient will be $[\CC^{12} / (\CC^*)^8]$ with weight
matrix
\begin{equation*}
  \begin{pmatrix}
    1 & 4 & 4 & 6 & 9 & 0 & 0 & 0 & 0 & 0 & 0 & 0 \\
    0 & 1 & 1 & 1 & 2 & 1 & 0 & 0 & 0 & 0 & 0 &0\\
    0 & 2 & 2 & 3 & 4 & 0 & 1 & 0 & 0 & 0 & 0&0 \\
    0 & 0 & 0 & 0 & 1 & 0 & 0 & 1 & 0 & 0 & 0 &0\\
    0 & 1 & 1 & 2 & 3 & 0 & 0 & 0 & 1 & 0 & 0 &0\\
    0 & 2 & 2 & 3 & 5 & 0 & 0 & 0 & 0 & 1 & 0&0 \\
    0 & 3 & 3 & 4 & 7 & 0 & 0 & 0 & 0 & 0 & 1&0\\
    0 & 1 &  1& 1 & 3 & 0 & 0 & 0 & 0 & 0 & 0 & 1
  \end{pmatrix},
\end{equation*}
again with rows corresponding to the ordering in \eqref{eq:x24-extset}. The
equation for the hypersurface is then extended to be of multi-degree
$(24,6,12,2,8,13,18, 7)$.

If we set
\begin{itemize}
\item
  $\rho_1 = \frac{1}{36} \, d_{0} - \frac{1}{4} \, d_{1} - \frac{1}{2} \, d_{2}
  - \frac{1}{9} \, d_{3} - \frac{1}{3} \, d_{4} - \frac{5}{9} \, d_{5} -
  \frac{7}{9} \, d_{6} - \frac{1}{3} \, d_{7}$
\item
  $\rho_2 = \frac{1}{9} \, d_{0} - \frac{4}{9} \, d_{3} - \frac{1}{3} \, d_{4} -
  \frac{2}{9} \, d_{5} - \frac{1}{9} \, d_{6} - \frac{1}{3} \, d_{7}$
\item
  $\rho_3 = \frac{1}{6} \, d_{0} - \frac{1}{2} \, d_{1} - \frac{2}{3} \, d_{3} -
  \frac{1}{3} \, d_{5} - \frac{2}{3} \, d_{6} - d_{7}$
\item
  $\rho_4 = \frac{1}{4} \, d_{0} - \frac{1}{4} \, d_{1} - \frac{1}{2} \, d_{2}$
\item
  $\rho_5 = \frac{2}{3} \, d_{0} - \frac{2}{3} \, d_{3} - \frac{1}{3} \, d_{5} - \frac{2}{3} \, d_{6} - d_{7}$
\item
  $\sigma = \frac{1}{36} \, d_{0} - \frac{1}{4} \, d_{1} - \frac{1}{2} \, d_{2}
  - \frac{1}{9} \, d_{3} - \frac{1}{3} \, d_{4} - \frac{5}{9} \, d_{5} -
  \frac{7}{9} \, d_{6} - \frac{1}{3} \, d_{7}$
\end{itemize}
The extended $I$-function is then given by

\begin{equation*}
  \begin{aligned}
    I_{X_{24}} =  \sum_{d_0 ,\ldots, d_7 \geq 0} &\dfrac{ \prod_{i=0}^7 q_i^{d_i}}{\prod_{i=1}^7 (d_i)!z^{d_i}} \cdot \\
                 &
                   \dfrac{
    \prod\limits_{\tiny\substack{\langle j \rangle = \langle \rho_1 \rangle \\ j \leq 0}} (H+jz)
    \prod\limits_{\tiny \substack{\langle j \rangle = \langle \rho_2 \rangle \\ j \leq 0}} (4H+jz)^2
    \prod\limits_{\tiny \substack{\langle j \rangle = \langle \rho_3 \rangle \\ j \leq 0}} (6H+jz)
    \prod\limits_{\tiny \substack{\langle j \rangle = \langle \rho_4 \rangle \\ j \leq 0}} (9H+jz)
    \prod\limits_{\tiny \substack{\langle j \rangle = \langle \rho_5 \rangle \\ j \leq \rho_5}} (24H+jz)}
  {
    \prod\limits_{\tiny\substack{\langle j \rangle = \langle \rho_1 \rangle \\ j \leq \rho_1}} (H+jz)
    \prod\limits_{\tiny\substack{\langle j \rangle = \langle \rho_2 \rangle \\ j \leq \rho_2}} (4H+jz)^2
    \prod\limits_{\tiny\substack{\langle j \rangle = \langle \rho_3 \rangle \\ j \leq \rho_3}} (6H+jz)
    \prod\limits_{\tiny\substack{\langle j \rangle = \langle \rho_4 \rangle \\ j \leq \rho_4}} (9H+jz)
    \prod\limits_{\tiny\substack{\langle j \rangle = \langle \rho_5 \rangle \\ j \leq 0}} (24H+jz)
  }
  \cdot
    \one_{\beta},
  \end{aligned}
\end{equation*}
where $\one_{\beta} = \one_{\langle - \sigma \rangle}$ unless $\langle   \sigma \rangle= 1/3,
2/3$ and $\rho_3 < 0$, in which case it is $4 [ \mathbb P(9)]$ in the twisted
sector $X_{\langle - \sigma \rangle}$.

\subsubsection{Wall-Crossing and Invariants}

Similar to the previous examples, we may use the wall-crossing formula
to compute the first few terms of the generating function of
Gromov--Witten invariants
\[
  \begin{aligned}
   \cF(Q, t) = &  6 Q t_{1} t_{6} + \frac{3}{4} t_{1}^{2} t_{2} + \frac{1}{18} t_{3}^{2} t_{6}
+ \frac{1}{9} t_{3} t_{4} t_{5} + \frac{1}{9} t_{3} t_{5} t_{7} + \frac{2}{27}
    t_{4}^{3} + \frac{1}{18} t_{4}^{2} t_{7} + \frac{1}{18} t_{4} t_{7}^{2} 
    + \frac{1}{54} t_{7}^{3} \\
    &+ 3 Q^{2} t_{2} t_{5} - \frac{1}{32} t_{1}^{4} -
\frac{5}{96} t_{2}^{4} - \frac{1}{162} t_{3} t_{4} t_{6}^{2} - \frac{1}{162}
t_{3} t_{5}^{2} t_{6} - \frac{1}{162} t_{3} t_{6}^{2} t_{7} - \frac{1}{81}
      t_{4}^{2} t_{5} t_{6}  \\
    &- \frac{1}{162} t_{4} t_{5}^{3} - \frac{2}{81} t_{4} t_{5}
t_{6} t_{7} - \frac{1}{162} t_{5}^{3} t_{7} - \frac{1}{81} t_{5} t_{6} t_{7}^{2}
+ \frac{45}{2} Q^{4} t_{3} + \frac{1}{2} Q^{2} t_{1}^{2} t_{5} + \frac{1}{12}
      Q^{2} t_{2} t_{6}^{2}\\
    & - \frac{3}{16} Q t_{1} t_{2}^{2} t_{6} - \frac{1}{27} Q
t_{1} t_{3}^{2} t_{5} - \frac{1}{27} Q t_{1} t_{3} t_{4}^{2} - \frac{2}{27} Q
t_{1} t_{3} t_{4} t_{7} - \frac{1}{27} Q t_{1} t_{3} t_{7}^{2} + \frac{1}{64}
      t_{1}^{2} t_{2}^{3} \\
    & + \frac{13}{17496} t_{3}^{4} t_{5} + \frac{5}{2916}
t_{3}^{3} t_{4}^{2} + \frac{5}{1458} t_{3}^{3} t_{4} t_{7} + \frac{5}{2916}
t_{3}^{3} t_{7}^{2} + \frac{1}{4374} t_{3} t_{5} t_{6}^{3} + \frac{1}{1458}
      t_{4}^{2} t_{6}^{3} \\
    & + \frac{7}{2916} t_{4} t_{5}^{2} t_{6}^{2} + \frac{1}{729}
t_{4} t_{6}^{3} t_{7} + \frac{13}{17496} t_{5}^{4} t_{6} + \frac{7}{2916}
t_{5}^{2} t_{6}^{2} t_{7} + \frac{1}{1458} t_{6}^{3} t_{7}^{2} + O(Q, t)^{6}
  \end{aligned}
\]

\subsubsection{Enumerative Invariants}

Similar to the discussion in Section~\ref{sss:enumerative}, the coefficient $6$ of
$Qt_1t_6$ is the number of lines $L \cong \mathbb P(4,9)$ connecting one of
the $6$ points of $X_{24} \cap \mathbb P(4, 4)$ and the point $\mathbb P(9)
\cong B\mu_9$. The coefficient $3$ of
$Q^2t_2t_5$ is the contribution of the double
covers of those $6$ lines by $\mathbb P(2,9)$.

We briefly explain the coefficient $3$ of  $Q^2t_2t_5$. The invariant
  counts maps from $\mathbb P(2, 9)$ into $X_{24}$ of degree $\frac{1}{18}$. 
 Let $u, v$
 be the homogeneous  coordinates on $\mathbb P(2, 9)$. Then the map must be given by
\[
    (u, v) \longmapsto (0, a_1 u^2, a_2 u^2, a_3u^3, a_4v). 
\]
A generic defining equation for $X_{24}$ is of the form
\[
  A F_1(x_1, x_2, x_3) + B x_3x_4^2  \quad \mod (x_0).
\]
It is easy to see that the map lands in $X_{24}$ if and only if
$a_3 = 0$ and $F_1(a_1, a_2) = 0$.
 There are $6$ such maps, each having an
  automorphism group of order $2$. Hence the invariant is $3$.

\subsubsection{Invariants of $[\CC^3/\mu_3]$ from Non-ambient Insertions}

The Gromov--Witten invariants of $[\CC^3/\mu_3]$ also appear in this
example.
Interestingly, they are part of the degree zero invariants of $X_{24}$
with non-ambient insertions.
Indeed, as we have mentioned,
$X_{24} \cap \mathbb P(6, 9) = p_1 \sqcup p_2$ is a disjoint union of
two points, where $p_1 \cong B\mu_9$ and $p_2 \cong B\mu_3$ sits at a generic
position of $\mathbb P(6, 9)$ for a generic $X_{24}$.
We may compute the normal bundle $N_{p_2 / X_{24}} \cong [\CC^3/\mu_3]$.
If we denote the Poincar\'e dual of $p_2$ by $\phi$, then
\[
  \langle \phi ,\ldots, \phi\rangle_{0, k, 0}
\]
are the degree $0$ local Gromov--Witten invariants of $[\CC^3/\mu_3]$.
Note that  $\phi = \mathbf 1_{1/3} -  \gamma_{1/3}$.
Thus, to extract the generating function
\[
  \sum_{k \geq 3} \frac{t^k}{k!} \langle \phi ,\ldots, \phi\rangle_{0, k, 0}
\]
from $\cF$, we simply set $Q = t_1 = t_2 = t_3 = t_5 = t_6 = 0$, and $t_4 = -t_7=
t$.
After the mirror map, this simple system of equations becomes 
complicated relations among the variables $q_0, \dotsc, q_7$.
Hence, just by looking at the $I$-function, it would be very hard to
verify that we recover the Gromov--Witten invariants of $[\CC^3/\mu_3]$,
although this is guaranteed by geometric reasoning.
Thus, we content ourselves with checking the first few
terms numerically:
\[
  \cF(0,0,0,0,t,0,0,-t) = \frac{1}{18} t^{3} - \frac{1}{19440} t^{6} + O(t^{9}).
\]
This matches with the $Q = 0$ specialization of the generating function
\eqref{eq:X7-invariants} of $X_7$.

\bibliographystyle{abbrv}
\bibliography{biblio}

\end{document}